\documentclass[11pt,reqno]{amsart} 

\usepackage{graphicx}

\usepackage[top=1.2in,bottom=1.2in,left=1in,right=1in]{geometry}

\usepackage{xcolor}

\usepackage{enumitem}

\usepackage[labelfont=bf,font=sf]{caption}

\usepackage[font={sf,small},labelfont=md]{subfig}

\usepackage[noadjust,nosort]{cite}

\usepackage{amsmath, amsthm, amssymb, tikz, mathtools, array, mathrsfs, tensor}
\usepackage{relsize}

\usetikzlibrary{shapes.geometric}
\tikzset{cross/.style={cross out, draw=black, minimum size=2.5*(#1-\pgflinewidth), inner sep=2pt, outer sep=0.5pt},
	cross/.default={1pt}}
\usetikzlibrary{calc,arrows}
\usepackage{pgfplots}

\usepackage{polynom}

\usepackage{esint}

\usepackage{multicol}

\usepackage{dsfont}
\newcommand{\one}{\mathds{1}}

\usepackage{framed}

\usepackage{upref}

\usepackage{hyperref}
\hypersetup{colorlinks=true,linkcolor=blue,citecolor=blue,urlcolor=blue} 

\allowdisplaybreaks


\def\eps{\varepsilon}

\newcommand{\E}{\mathbb{E}}
\newcommand{\F}{\mathbb{F}}
\newcommand{\N}{\mathbb{N}}
\renewcommand{\P}{\mathbb{P}}

\newcommand{\R}{\mathbb{R}}

\newcommand{\Z}{\mathbb{Z}}

\newcommand{\CC}{\mathcal{C}}

\newcommand{\FF}{\mathcal{F}}
\newcommand{\GG}{\mathcal{G}}
\newcommand{\HH}{\mathcal{H}}

\newcommand{\PP}{\mathcal{P}}

\newcommand{\TT}{\mathcal{T}}
\newcommand{\UU}{\mathcal{U}}

\newcommand{\YY}{\mathcal{Y}}

\newcommand{\GGG}{\mathscr{G}}

\newcommand{\KKK}{\mathscr{K}}

\newcommand{\MMM}{\mathscr{M}}

\newcommand{\PPP}{\mathscr{P}}

\newcommand{\UUU}{\mathscr{U}}

\newcommand{\YYY}{\mathscr{Y}}
\newcommand{\ZZZ}{\mathscr{Z}}

\newcommand{\ttt}{\mathfrak{t}}
\newcommand{\uuu}{\mathfrak{u}}

\newcommand{\bDelta}{\boldsymbol{\Delta}}

\newcommand{\bZ}{\textnormal{\textbf{Z}}}
\newcommand{\bN}{\textnormal{\textbf{N}}}
\newcommand{\bY}{\textnormal{\textbf{Y}}}
\newcommand{\onb}{\textnormal{\textbf{b}}}
\newcommand{\bq}{\textnormal{\textbf{q}}}
\newcommand{\bF}{\textnormal{\textbf{F}}}
\newcommand{\bD}{\textnormal{\textbf{D}}}
\newcommand{\bC}{\textnormal{\textbf{C}}}
\newcommand{\bT}{\textnormal{\textbf{T}}}

\newcommand{\bB}{\textnormal{\textbf{B}}}

\newcommand{\sig}{\underline{\sigma}}
\newcommand{\zet}{\underline{\zeta}}
\newcommand{\uL}{\underline{\texttt{L}}}

\newcommand{\ubx}{\underline{\textbf{x}}}
\newcommand{\ux}{\underline{x}}
\newcommand{\uy}{\underline{y}}
\newcommand{\utau}{\underline{\tau}}

\newcommand{\bx}{\mathbf{x}}

\newcommand{\bsigma}{\boldsymbol{\sigma}}

\newcommand{\bsig}{\underline{\boldsymbol{\sigma}}}

\newcommand{\btau}{\boldsymbol{\tau}}

\newcommand{\lit}{\textnormal{lit}}

\newcommand{\tr}{\textnormal{tr}}

\newcommand{\ind}{\textnormal{ind}}

\newcommand{\dotq}{\dot{q}}

\newcommand{\rr}{{{\scriptsize{\texttt{R}}}}}
\newcommand{\bb}{{{\scriptsize{\texttt{B}}}}}
\newcommand{\fs}{{\scriptsize{\texttt{S}}}}
\newcommand{\ff}{\textnormal{\small{\texttt{f}}}}
\newcommand{\tz}{\small{\texttt{z}}}
\newcommand{\mm}{{{\texttt{m}}}}
\newcommand{\fF}{\scriptsize{\texttt{F}}}
\newcommand{\bs}{{\textbf{s}}}
\newcommand{\tL}{\texttt{L}}

\newcommand{\tZ}{\bZ^{(L),\textnormal{tr}}}

\newcommand{\la}{\lambda}

\newcommand{\bwZ}{\widetilde{\bZ}}

\DeclareMathOperator{\Var}{Var}

\makeatletter
\DeclareFontFamily{OMX}{MnSymbolE}{}
\DeclareSymbolFont{MnLargeSymbols}{OMX}{MnSymbolE}{m}{n}
\SetSymbolFont{MnLargeSymbols}{bold}{OMX}{MnSymbolE}{b}{n}
\DeclareFontShape{OMX}{MnSymbolE}{m}{n}{
	<-6>  MnSymbolE5
	<6-7>  MnSymbolE6
	<7-8>  MnSymbolE7
	<8-9>  MnSymbolE8
	<9-10> MnSymbolE9
	<10-12> MnSymbolE10
	<12->   MnSymbolE12
}{}
\DeclareFontShape{OMX}{MnSymbolE}{b}{n}{
	<-6>  MnSymbolE-Bold5
	<6-7>  MnSymbolE-Bold6
	<7-8>  MnSymbolE-Bold7
	<8-9>  MnSymbolE-Bold8
	<9-10> MnSymbolE-Bold9
	<10-12> MnSymbolE-Bold10
	<12->   MnSymbolE-Bold12
}{}

\let\llangle\@undefined
\let\rrangle\@undefined
\DeclareMathDelimiter{\llangle}{\mathopen}%
{MnLargeSymbols}{'164}{MnLargeSymbols}{'164}
\DeclareMathDelimiter{\rrangle}{\mathclose}%
{MnLargeSymbols}{'171}{MnLargeSymbols}{'171}
\makeatother

\newtheorem{thm}{Theorem}[section]
\newtheorem{prop}[thm]{Proposition}
\newtheorem{cor}[thm]{Corollary}
\newtheorem{lemma}[thm]{Lemma}
\newtheorem{claim}[thm]{Claim}

\theoremstyle{definition}
\newtheorem{defn}[thm]{Definition}

\newtheorem{remark}[thm]{Remark}

\renewcommand{\thefootnote}{\fnsymbol{footnote}}

\title[1RSB of random regular NAE-SAT]{One-step replica symmetry breaking of random regular NAE-SAT II}



\author{Danny Nam}
\address{\newline Department of Mathematics \newline Princeton University \newline Princeton, NJ 08544 \newline \textup{\tt dhnam@math.princeton.edu}}

\author{Allan Sly}
\address{\newline Department of Mathematics \newline Princeton University \newline Princeton, NJ 08544 \newline \textup{\tt asly@math.princeton.edu}}

\author{Youngtak Sohn}
\address{\newline Department of Statistics \newline Stanford University\newline Stanford, CA 94305 \newline \textup{\tt youngtak@stanford.edu}}

\begin{document}
	\bibliographystyle{acm}
	
	\renewcommand{\thefootnote}{\arabic{footnote}} \setcounter{footnote}{0}

	\begin{abstract}
		Continuing our earlier work in \cite{nss20a}, we study the random regular $k$-\textsc{nae-sat} model in the condensation regime. In \cite{nss20a}, the ({\scriptsize 1}\textsc{rsb}) properties of the model were established with positive probability. In this paper, we improve the result to probability arbitrarily close to one. To do so, we introduce a new framework which is the synthesis of two approaches: the small subgraph conditioning and a variance decomposition technique using Doob martingales and discrete Fourier analysis. The main challenge is a delicate  integration of the two methods to overcome the  difficulty arising from applying the moment method to an unbounded state space.
	\end{abstract}
	
	\maketitle
	
	
	
	\setcounter{tocdepth}{1}

	\section{Introduction}

	Building on earlier theory of spin-glasses, statistical physicists in the early 2000's developed a detailed collection of predictions for a broad class of sparse random constraint satisfaction problems (r\textsc{csp}). These predictions describe a series of phase transitions as the constraint density varies, which is governed by \textit{one-step replica symmetry breaking} ({\footnotesize{1}}\textsc{rsb})  (\cite{mpz02,kmrsz07}; cf. \cite{anp05} and  Chapter 19 of \cite{mm09} for a survey).
	We study one of such r\textsc{csp}'s, named the \textit{random regular} $k$-\textsc{nae-sat} model, which is perhaps the most mathematically tractable among the {\footnotesize{1}}\textsc{rsb} class of r\textsc{csp}'s. As a continuation of the companion work \cite{nss20a}, in this paper we complete our program of establishing that the {\footnotesize{1}}\textsc{rsb} prediction for the random regular \textsc{nae-sat} holds with probability arbitrarily close to one.
	
	The \textsc{nae-sat} problem is a random Boolean \textsc{cnf} formula, where $n$ Boolean variables are subject to constraints in the form of clauses which are the ``\textsc{or}'' of $k$ of the variables or their negations chosen uniformly at random and the formula is the ``\textsc{and}'' of the clauses. A variable assignment $\underline{x}\in \{0,1\}^{n}$ is called a \textsc{nae-sat} \textit{solution} if both $\underline{ x}$ and $\neg \underline{ x}$ evaluate to \textsf{true}. We  then choose a uniformly random instance of $d$-regular (each variable appears $d$ times) $k$-\textsc{nae-sat} (each clause has $k$ literals) problem, which gives the \textit{random d-regular k}-\textsc{nae-sat} problem, with clause density $\alpha=d/k$ (see Section \ref{sec:model} for a formal definition).

       Let $Z_n$ denote the number of solutions for a given random $d$-regular $k$-\textsc{nae-sat} instance. The physics prediction is that for each fixed $\alpha$, there exists $\textsf{f}(\alpha)$ called the free energy such that
	$$\frac{1}{n} \log Z_n \;\longrightarrow\; \textsf{f}(\alpha) \quad \textnormal{in probability}. $$
	A direct computation of the first moment $\E Z_n$ gives that 
	$$\E Z_n = 2^n \left(1-2^{-k+1} \right)^m = e^{n\textsf{f}^{\textsf{rs}}(\alpha)}, \quad \textnormal{where}\quad \textsf{f}^{\textsf{rs}}(\alpha)\equiv \log 2+ \alpha \log \left( 1-2^{-k+1}\right), $$
	($\textsf{f}^{\textsf{rs}}(\alpha)$ is called the \textit{replica-symmetric} free energy), so $\textsf{f}\leq \textsf{f}^{\textsf{rs}}$ holds by Markov's inequality. The work of Ding-Sly-Sun \cite{dss16} and Sly-Sun-Zhang \cite{ssz16}  established some of the physics conjectures on the description of $Z_n$ and $\textsf{f}$ given in \cite{zk07,kmrsz07,mrs08}, which are summarized as follows. 
	\begin{itemize}
		\item (\cite{dss16}) For large enough $k$, there exists the \textit{satisfiability threshold} $\alpha_{\textsf{sat}}\equiv \alpha_{\textsf{sat}}(k)>0$ such that
		\begin{equation*}
		\lim_{n\to\infty} \P(Z_n>0)
		=
		\begin{cases}
		1 & \textnormal{ for } \alpha\in(0,\alpha_{\textsf{sat}});\\
		0 & \textnormal{ for }\alpha > \alpha_{\textsf{sat}}.
		\end{cases}
		\end{equation*}

		\item (\cite{ssz16}) For large enough $k$, there exists the \textit{condensation threshold} $\alpha_{\textsf{cond}}\equiv \alpha_{\textsf{cond}}(k)\in(0,\alpha_{\textsf{sat}})$  such that 
		\begin{equation}\label{eq:intro:freeEformula}
		\textsf{f}(\alpha)=
		\begin{cases} 
		\textsf{f}^{\textsf{rs}}(\alpha) &
		\textnormal{ for } \alpha \leq \alpha_{\textsf{cond}};\\
		\textsf{f}^{1\textsf{rsb}}(\alpha) 
		& \textnormal{ for } \alpha > \alpha_{\textsf{cond}},
		\end{cases}
		\end{equation}
		where $\textsf{f}^{1\textsf{rsb}}(\alpha)$ is the {\small{1}}\textsc{rsb} \textit{free energy}. Moreover, $\textsf{f}^{\textsf{rs}}(\alpha) > \textsf{f}^{\textsf{1\textsf{rsb}}}(\alpha)$ holds for $\alpha \in (\alpha_{\textsf{cond}},\alpha_{\textsf{sat}})$. For the explicit formula and derivation of $\textsf{f}^{1\textsf{rsb}}(\alpha)$ and $\alpha_{\textsf{cond}}$, we refer to Section 1.6 of \cite{ssz16} for a concise overview.
	\end{itemize}
	Furthermore, there are more detailed physics predictions that the solution space of the random regular $k$-\textsc{nae-sat} is  \textit{condensed} when $\alpha \in(\alpha_{\textsf{cond}},\alpha_{\textsf{sat}})$ into a finite number of \textit{clusters}. Here, \textit{cluster} is defined by the connected component of the solution space, where we connect two solutions if they differ by one variable. Indeed, in \cite{nss20a}, we proved that for large enough $k$, the solution space of random regular $k$-\textsc{nae-sat} indeed becomes \textit{condensed} in the condensation regime for a positive fraction of the instances. That is, it holds with probability strictly bounded away from 0.
	
	The following theorem strengthens the aforementioned result and shows that the condensation phenomenon holds with probability arbitrarily close to 1.

	\begin{thm}\label{thm1}
		Let $k\geq k_0$ where $k_0$ is a large absolute constant, and let $\alpha \in (\alpha_{\textsf{cond}}, \alpha_{\textsf{sat}})$ such that $d\equiv \alpha k\in \mathbb{N}$. For all $\eps >0$ and $M\in \N$, there exist constants $K\equiv K(\eps,\alpha,k)\in \N$ and $C\equiv C(M,\eps,\alpha,k)>0$ such that with probability at least $1-\eps$, the random $d$-regular $k$-\textsc{nae-sat} instance satisfies the following:
		\begin{enumerate}
			
			\item[\textnormal{(a)}] The $K$  largest solution clusters, $\CC_1,\ldots,\CC_K$, occupy at least $1-\eps$ fraction of the solution space;
			
			\item[\textnormal{(b)}] There are at least $\exp(n\textnormal{\textsf{f}}^{1\textnormal{\textsf{rsb}}}(\alpha) -c^{\star}\log n -C )$ many solutions in $\CC_1,\ldots,\CC_M$, the $M$ largest clusters (see Definition \ref{def:la:s:c:star} for the definition of $c^\star$).
		\end{enumerate} 
	\end{thm}
	\begin{remark}\label{rem:k:adjusted}
	Throughout the paper, we take $k_0$ to be a large absolute constant so that the results of \cite{dss16}, \cite{ssz16} and \cite{nss20a} hold. In addition, it was shown in \cite[Proposition 1.4]{ssz16} that $(\alpha_{\textsf{cond}}, \alpha_{\textsf{sat}})$ is a subset of $(\alpha_{\textsf{lbd}}, \alpha_{\textsf{ubd}})$, where $\alpha_{\textsf{lbd}}\equiv (2^{k-1}-2)\log 2$ and $\alpha_{\textsf{ubd}}\equiv 2^{k-1}\log 2$, so we restrict our attention to $\alpha \in (\alpha_{\textsf{lbd}}, \alpha_{\textsf{ubd}})$.
	\end{remark}

	\subsection{One-step replica symmetry breaking}\label{subsec:intro:rsb}
	
	In the condensation regime $\alpha\in(\alpha_{\textsf{cond}},\alpha_{\textsf{sat}})$, the random regular $k$-\textsc{nae-sat} model is believed to possess a single layer of hierarchy of clusters in the solution space. That is, the solutions are fairly \textit{well-connected} inside each cluster so that no additional hierarchical structure in it. Such behavior is conjectured in various other models such as random graph coloring and random $k$-\textsc{sat}. We remark that there are also other models such as maximum independent set (or high-fugacity hard-core model) in random graphs with small degrees \cite{bkzz13} and Sherrington-Kirkpatrick  model \cite{SK75,t06}, which are expected or proven \cite{acz20} to undergo \textit{full} \textsc{rsb}, which means that there are infinitely many levels of hierarchy inside the solution clusters.
	
	A way to characterize {\small{1}}\textsc{rsb} is to look at the \textit{overlap} between two uniformly drawn solutions. In the condensation regime, there are a bounded number of clusters containing most of the solutions. Thus, the events of two solutions belonging to the same cluster, or different clusters, each happen with a non-trivial probability. According to the description of {\small{1}}\textsc{rsb}, there is no additional structure inside each cluster, so the Hamming distance between the two solutions is expected to concentrate precisely at  \textit{two values}, depending on whether they came from the same cluster or not.
	
	It was verified in \cite{nss20a} that the overlap concentrates at two values for a positive fraction of the random regular \textsc{nae-sat} instances. Theorem \ref{thm2} below verifies that the overlap concentration happens for almost all random regular \textsc{nae-sat} instances.

	\begin{defn}
		For $\underline{x}^1,\underline{x}^2 \in \{0,1\}^n$, let $\underline{y}^i = 2\underline{x}^i - \textbf{1}$. The overlap $\rho(\underline{x}^1,\underline{x}^2)$ is defined by
		\begin{equation*}
		\rho(\underline{x}^1,\underline{x}^2) \equiv \frac{1}{n} \uy^1 \cdot \uy^2 = \frac{1}{n} \sum_{i=1}^n y^1_i y^2_i.
		\end{equation*}
		In words, the overlap is the normalized difference between the number of variables with the same value and the number of those with different values.
	\end{defn}

	\begin{thm}\label{thm2}
		Let $k\geq k_0$, $\alpha \in (\alpha_{\textsf{cond}}, \alpha_{\textsf{sat}})$ such that $d\equiv\alpha k\in \mathbb{N}$, and $p^\star\equiv p^\star(\alpha,k)\in(0,1)$ be a fixed constant (for its definition, see Definition 6.8 of \cite{nss20a}).
		For all $\eps>0$, there exist constants $\delta=\delta(\eps,\alpha,k)>0$ and $C\equiv C(\eps,\alpha,k)$ such that with probability at least $1-\eps$, the random $d$-regular $k$-\textsc{nae-sat} instance $\GGG$ satisfies the following. Let $\underbar{x}^1, \underbar{x}^2\in \{0,1\}^n$ be independent, uniformly chosen satisfying assignments of $\GGG$.  Then, the absolute value $\rho_{\textnormal{abs}} \equiv |\rho|$ of their overlap $\rho \equiv \rho (\underline{x}^1,\underline{x}^2)$ satisfies
		\begin{enumerate}
			\item[\textnormal{(a)}] $\mathbb{P}(\rho_{\textnormal{abs}}\leq n^{-1/3} |\GGG ) \geq \delta$;
			
			\item[\textnormal{(b)}] $\mathbb{P}( \big|\rho_{\textnormal{abs}}-p^\star \big| \leq n^{-1/3} |\GGG )\geq \delta$;
			
			\item[\textnormal{(c)}] $\mathbb{P}( \min\{ \rho_{\textnormal{abs}}, |\rho_{\textnormal{abs}}-p^\star |\} \geq n^{-1/3}|\GGG )\leq Cn^{-1/4}$.
		\end{enumerate}
		
	\end{thm}
	
		\subsection{Related works}\label{subsec:intro:related}
	
	Many of the earlier works on r\textsc{csp}s focused on determining the satisfiability thresholds and for r\textsc{csp} models that are known not to exhibit \textsc{rsb}, such goals were established. These models include random linear equations \cite{acgm17}, random 2-\textsc{sat} \cite{cr92,bbckw01},  random 1-\textsc{in}-$k$-\textsc{sat} \cite{acim01} and $k$-\textsc{xor-sat} \cite{dm02,dgmm10,ps16}. On the other hand, for the models which are predicted to exhibit \textsc{rsb}, intensive studies have been conducted to estimate their satisfiability threshold (random $k$-\textsc{sat} \cite{kkks98,ap04,cp16} , random $k$-\textsc{nae-sat} \cite{am06,cz12,cp12} and random graph coloring \cite{an05,c13,cv13,ceh16}). 
	
	More recently, the satisfiability thresholds for r\textsc{csp}s in {\small{1}}\textsc{rsb} class have been rigorously determined for several models including maximum independent set \cite{dss16maxis}, random regular $k$-\textsc{nae-sat} \cite{dss16},  random regular $k$-\textsc{sat} \cite{cp16} and  random $k$-\textsc{sat} \cite{dss15ksat}). These works carried out a demanding second moment method to the number of \textit{clusters} instead of the number of solutions. Although determining the colorability threshold is left open, the condensation threshold for random graph coloring was established in \cite{bchrv16}, where they conducted a challenging analysis based on a clever ``planting" technique, and the results were generalized to other models in \cite{ckpz18}. Also,~\cite{bc16} identified the condensation threshold for random regular $k$-\textsc{sat}, where each variable appears $d/2$-times positive and $d/2$-times negative. 
	
	Further theory was developed in \cite{ssz16} to establish the {\small{1}}\textsc{rsb} free energy for random regular $k$-\textsc{nae-sat} in the condensation regime by applying second moment method to \textit{$\lambda$-tilted} partition function. 
	Later, our companion paper \cite{nss20a} made further progress in the same model by giving a cluster level description of the condensation phenomenon. Namely, \cite{nss20a} showed that with positive probability, a bounded number of clusters dominate the solution space and the overlap concentrates on two points in the condensation regime. Our main contribution is to push the probability arbitrarily close to one and show that the same phenomenon holds with high probability.
	
	
	Lastly, \cite{bsz19} studied the random $k$-\textsc{max-nae-sat} beyond $\alpha_{\textsf{sat}}$, where they  verified that the {\small{1}}\textsc{rsb} description breaks down before $\alpha \asymp k^{-3}4^k$. Indeed, the \textit{Gardner transition} from {\small{1}}\textsc{rsb} to full \textsc{rsb} is expected at $\alpha_{\textsf{Ga}}\asymp k^{-3}4^k >\alpha_{\textsf{sat}}$ \cite{mr03,kpw04}, and \cite{bsz19} provides evidence of this phenomenon.
	
	\subsection{Proof ideas}\label{subsec:intro:pfapproach} 
	
	In \cite{nss20a}, the majority of the work was to compute moments of the \textit{tilted cluster partition function} $\overline{\bZ}_\lambda$ and $\overline{\bZ}_{\la,s}$, defined as
	\begin{equation}\label{eq:def:bZbar:elementary}
	\overline{\bZ}_\lambda \equiv \sum_{\Upsilon} |\Upsilon|^\lambda \quad\textnormal{and}\quad  \overline{\bZ}_{\lambda,s} := \sum_{\Upsilon} |\Upsilon|^\lambda\, \one{\{|\Upsilon| \in [e^{ns}, e^{ns+1}) \}}.
	\end{equation} 
	where the sums are taken over all clusters $\Upsilon$. Moreover, let $\overline{\bN}_{s}$ denote the number of clusters whose size is in the interval $[e^{ns}, e^{ns +1} )$, i.e. 
	\begin{equation}\label{eq:def:bN:elementary}
	    \overline{\bN}_s :=\overline{\bZ}_{0,s}.
	\end{equation}
	Denote  $s_\circ$ to be the size of the solution space in normalized logarithmic scale from Theorem \ref{thm1}:
	\begin{equation}\label{eq:def:s:circle}
	    s_\circ \equiv s_\circ(n,\alpha,C) \equiv \textsf{f}^{1\textsf{rsb}}(\alpha) -\frac{ c_\star \log n}{n} - \frac{C}{n},
	\end{equation}
	where $c_\star$ is the constant introduced in Theorem \ref{thm1} and $C\in \mathbb{R}$. In \cite{nss20a}, we obtained the estimates on $\bN_{s_\circ}$ from the second moment method showing that $\E [\overline{\bN}_{s_\circ}^2] \lesssim_{k} (\E \overline{\bN}_{s_\circ})^2$, and  that $\E \overline{\bN}_{s_\circ}$ decays exponentially as $C\to -\infty$. Thus, it was shown in \cite{nss20a} that
	\begin{equation*}
	\P (\overline{\bN}_{s_\circ} >0) 
	\begin{cases}
	\to 0, & \textnormal{ as } C\to -\infty;\\
	\geq c>0, & \textnormal{ as } C \to \infty. 
	\end{cases}
	\end{equation*}
	However, in order to establish (a) and (b) of Theorem \ref{thm1}, we need to push the probability in the second line to $1-\eps$.

	To do so, one may hope to have $\E [\overline{\bN}_{s_\circ}^2] \approx (\E \overline{\bN}_{s_\circ})^2$ to deduce $\P (\overline{\bN}_{s_\circ} >0) \to 1$ for large enough $C$, but this is false in the case of random regular \textsc{nae-sat}. The primary reason is that short cycles in the graph causes multiplicative fluctuations in $\overline{\bN}_{s_\circ}$.  Therefore, our approach is to show that if we rescale $\overline{\bN}_{s_\circ}$ according to the effects of short cycles, then the resulting rescaled partition function $\widetilde{\bN}_{s_\circ}$ concentrates. That is, $\E [\widetilde{\bN}_{s_\circ}^2] \approx (\E \widetilde{\bN}_{s_\circ})^2$ (to be precise, this will only be true when $C$ is large enough, due to the intrinsic correlations coming from the largest clusters). Furthermore, we argue that the fluctuations coming from the short cycles are not too big, and hence can be absorbed by $\overline{\bN}_{s_\circ}$ if $\E\overline{\bN}_{s_\circ}$ is large. To this end, we develop a new argument that combines \textit{small subgraph conditioning} \cite{rw92,rw94}, which is a widely used tool in problems on random graphs, and the \textit{Doob martingale approach} used in \cite{dss16maxis,dss16}, which are not effective in our model if used alone.
	
	The \textit{small subgraph conditioning method} (\cite{rw92,rw94}; for a survey, see Chapter 9.3 of \cite{jlrrg}) has proven to be useful in many settings \cite{s10,gsv15,gsv16} to derive a precise distributional limits of partition functions. For example, \cite{gsv15} applied this method  to the proper coloring model of bipartite random regular graphs, where they determined the limiting distribution of the number of colorings. However, this method relies much on algebraic identities specific to the model which are sometimes not tractable, including our case. Roughly speaking, one needs a fairly clear combinatorial formula of the second moment to carry out the algebraic and combinatorial computations.
	
	Another technique that inspired our proof, which we will refer to as the \textit{Doob martingale approach}, was introduced in \cite{dss16maxis,dss16}. This method rather directly controls the multiplicative fluctuations of $\overline{\bN}_{s_\circ}$, by investigating the Doob martingale increments of $\log \overline{\bN}_{s_\circ}$. It has proven to be useful in the study of models like random regular \textsc{nae-sat}, as seen in \cite{dss16}. However, in the spin systems with infinitely many spins like our model, some of the key estimates in the argument become false, due to the existence of rare spins (or large free components).
	
	Our approach blends the two techniques in a novel way to back up each other's limitations. Although we could not algebraically derive the identities required for the small subgraph conditioning, we instead deduce them by a modified Doob martingale approach for the \textit{truncated} model which has a finite spin space. Then, we send the truncation parameter to infinity on these algebraic identities, and show that they converge to the corresponding formulas for the \textit{untruncated} model. This step requires a more refined understanding on the first and second moments of $\widetilde{\bN}_{s_\circ}$, including the constant coefficient of the leading exponential term, whereas the order of the leading order was sufficient in the earlier works \cite{dss16,ssz16}. We then  appeal to the small subgraph conditioning method  to deduce the conclusion based on those  identities. We believe that our approach is potentially applicable to other models with an infinite spin space where the traditional small subgraph conditioning method is not tractable.

	\subsection{Notational conventions}
	For non-negative quantities $f=f_{d,k, n}$ and $g=g_{d,k,n}$, we use any of the equivalent notations $f=O_{k}(g), g= \Omega_k(f), f\lesssim_{k} g$ and $g \gtrsim_{k} f $ to indicate that for each $k\geq k_0$,
	\begin{equation*}
	    \limsup_{n\to\infty} \frac{f}{g} <\infty,
	\end{equation*}
	with the convention $0/0\equiv 1$. We drop the subscript $k$ if there exists a universal constant $C$ such that
	\begin{equation*}
	    \limsup_{n\to\infty} \frac{f}{g} \leq C.
	\end{equation*}
	When $f\lesssim_{k} g$ and $g\lesssim_{k} f$, we write $f\asymp_{k} g$. Similarly when $f\lesssim g$ and $g\lesssim f$, we write $f \asymp g$.

		\section{The combinatorial model}\label{sec:model}

		We begin with introducing the mathematical framework to analyze the clusters of solutions.
		We follow the formulation derived in \cite[Section 2]{ssz16}. In \cite{nss20a}, we needed further definitions in addition to those from \cite{ssz16}, but in this work it is enough to rely on the concepts of \cite{ssz16}. In this section, we briefly review the necessary definitions for completeness.
		
		There is a natural graphical representation to describe  a $d$-regular $k$-\textsc{nae-sat} instance by a labelled $(d,k)$-regular bipartite graph: Let $V=\{ v_1, \ldots , v_n \}$ and $F=\{a_1, \ldots, a_m \}$ be the sets of variables and clauses, respectively. Connect $v_i$ and $a_j$ by an edge if  $v_i$ is one of the variables contained in the clause $a_j$. Let $\GG=(V,F,E)$ be this bipartite graph, and let $\tL_e\in \{0,1\}$ for $e\in E$ be the literal corresponding to the edge $e$. Then, the labelled bipartite graph $\GGG=(V,F,E,\uL)\equiv (V,F,E,\{\tL_{e}\}_{e\in E})$ represents a \textsc{nae-sat} instance.
		
		For each $e\in E$, we denote the variable(resp. clause) adjacent to it by $v(e)$ (resp. $a(e)$).  Moreover, $\delta v$ (resp. $\delta a$) are the collection of adjacent edges to $v\in V$ (resp. $a \in F$). We denote $\delta v \setminus e := \delta v \setminus \{e\}$ and $\delta a \setminus e := \delta a \setminus \{e\}$ for simplicity. Formally speaking, we regard $E$ as a perfect matching between the set of \textit{half-edges} adjacent to variables and those to clauses which are labelled from $1$ to $nd=mk$, and hence a permutation in $S_{nd}$. 
		
		\begin{defn}
			For an integer $l\ge 1$ and $\ubx  =(\bx_i) \in \{0,1\}^l$, define
			\begin{equation}\label{eq:def:INAE}
			I^{\textsc{nae}}(\ubx) :=  \mathds{1} \{\ubx \textnormal{ is neither identically } 0 \textnormal{ nor }1 \}.
			\end{equation} 
			Let $\GGG = (V,F,E,\uL)$ be a \textsc{nae-sat} instance. An assignment $\ubx \in \{0,1\}^V$ is called a \textbf{solution}  if 
			\begin{equation}\label{eq:def:INAE G}
			I^{\textsc{nae}}(\ubx;\GGG) :=
			\prod_{a\in F} I^{\textsc{nae}} \big((\bx_{v(e)} \oplus \tL_e)_{e\in \delta a}\big) =1,
			\end{equation}
			where $\oplus$ denotes the addition mod 2. Denote $\textsf{SOL}(\GGG)\subset\{0,1\}^V$ by the set of solutions and endow a graph structure on $\textsf{SOL}(\GGG)$ by connecting $\ubx \sim \ubx'$ if and only if they have a unit Hamming distance. Also, let $\textsf{CL}(\GGG)$ be the set of \textbf{clusters}, namely the connected components under this adjacency.
		\end{defn}

		\subsection{The frozen configuration}
		Our first step is to define \textit{frozen configuration} which is a basic way of encoding  clusters.
		We introduce \textit{free variable} which we denote by $\ff$, whose Boolean addition is defined as $\ff \oplus 0 := \ff =: \ff \oplus 1$. Recalling the definition of $I^{\textsc{nae}}$ \eqref{eq:def:INAE G}, a \textit{frozen configuration} is defined as follows.
		
		\begin{defn}[Frozen configuration]\label{def:frozenconfig}
			For $\GGG = (V,F,E, \uL)$, $\ux\in\{0,1,\ff \}^V$ is called a \textbf{frozen configuration} if the following conditions are satisfied:
			\begin{itemize}
				\item No \textsc{nae-sat} constraints are violated for $\ux$. That is, $I^{\textsc{nae}}(\ux;\GGG)=1$.
				
				\item For $v\in V$, $x_v\in \{0,1\}$ if and only if it is forced to be so. That is, $x_v\in \{0,1\}$ if and only if there exists $e\in \delta v$ such that $a(e)$ becomes violated if $\tL_e$ is negated, i.e., $I^{\textsc{nae}} (\ux; \GGG\oplus\one_e )=0$ where $\GGG\oplus\one_e$ denotes $\GGG$ with $\tL_e$ flipped. $x_v=\ff$ if and only if no such $e\in \delta v$ exists.
			\end{itemize}
		\end{defn}
		
		We record the observations which are direct from the definition. Details can be found in the previous works (\cite{dss16}, Section 2 and \cite{ssz16}, Section 2).
		
	\begin{enumerate}
		\item We can map a \textsc{nae-sat} solution $\ubx \in \{0,1 \}^V$ to a frozen configuration via the following \textit{coarsening} algorithm: If there is a variable $v$ such that $\bx_v\in\{0,1\}$ and $I^{\textsc{nae}}(\ubx;\GGG) = I^{\textsc{nae}}(\ubx\oplus \one_v ;\GGG)=1$ (i.e., flipping $\bx_v$ does not violate any clause), then set $\bx_v = \ff$. Iterate this process until additional modifications are impossible.
		
		\item All solutions in a cluster $\Upsilon\in \mathsf{CL}(\GGG)$ are mapped to the same frozen configuration $\ux\equiv \ux[\Upsilon] \in \{0,1,\ff\}^V$. However, coarsening algorithm is not necessarily surjective. For instance, a typical instance of $\GGG$ does not have a cluster corresponding to all-free ($\ubx \equiv \ff$).
	\end{enumerate}

		\subsection{Message configurations}
		Although the frozen configurations provide a representation of clusters, it does not tell us how to comprehend the size of clusters.  The main obstacle in doing so comes from the connected structure of free variables which can potentially be complicated. We now introduce the notions to comprehend this issue in a tractable way.
		
		\begin{defn}[Separating and forcing clauses]
			Let $\ux$ be a given frozen configuration on $\GGG = (V,F,E,\uL)$.  A clause $a\in F$ is called \textbf{separating} if there exist $e, e^\prime \in \delta a$ such that $\tL_{e}\oplus x_{v(e)} = 0, \quad \tL_{e^\prime} \oplus x_{v(e^\prime)}=1.$
			We say $a\in F$ is \textbf{non-separating} if it is not a separating clause. Moreover, $e\in E$ is called \textbf{forcing} if $\tL_{e}\oplus x_{v(e)} \oplus 1 = \tL_{e'}\oplus x_{v(e')}\in \{0,1\}$ for all $e'\in \delta a(e) \setminus e$. We say $a\in F$ is \textbf{forcing}, if there exists $e\in \delta a$ which is a forcing edge. In particular, a forcing clause is also separating.
		\end{defn}
		
		Observe that a non-separating clause must be adjacent to at least two free variables, which is a fact frequently used throughout the paper.
		
		\begin{defn}[Free cycles]
			Let $\ux$ be a given frozen configuration on $\GGG = (V,F,E,\uL)$. A cycle in $\GGG$ (which should be of an even length) is called a \textbf{free cycle} if 
			\begin{itemize}
				\item Every variable $v$ on the cycle is $x_v = \ff$;
				
				\item Every clause $a$ on the cycle is non-separating.
			\end{itemize}
		\end{defn}
		
		Throughout the paper, our primary interest is on the frozen configurations which does not contain any free cycles. If $\ux$ does not have any free cycle, then we can easily extend it to a \textsc{nae-sat} solution in $\ubx$ such that $\bx_v = x_v$ if $x_v\in\{0,1\}$, since \textsc{nae-sat} problem on a tree is always solvable.  
		
		\begin{defn}[Free trees]\label{def:freetree:basic}
			Let $\ux$ be a frozen configuration in $\GGG$ without any free cycles. Consider the induced subgraph $H$ of $\GGG$ consisting of free variables and non-separating clauses. Each connected component of $H$ is called \textbf{free piece} of $\ux$ and denoted by $\ttt^{\textnormal{in}}$. For each free piece $\ttt^{\textnormal{in}}$, the {\textbf{free tree}} $\ttt$ is defined by the union of $\ttt^{\textnormal{in}}$  and the \textit{half-edges} incident to $\ttt^{\textnormal{in}}$.
			
			For the pair $(\ux, \GGG)$, we write $\mathscr{F}(\ux,\GGG)$ to denote the collection of free trees inside $(\ux, \GGG)$. We write $V(\ttt)=V(\ttt^{\textnormal{in}})$, $F(\ttt)=F(\ttt^{\textnormal{in}})$ and $E(\ttt)=E(\ttt^{\textnormal{in}})$ to be the collection of variables, clauses and (full-)edges in $\ttt$. Moreover, define $\dot{\partial} \ttt$ (resp. $\hat{\partial} \ttt$) to be the collection of boundary half-edges that are adjacent to $F(\ttt)$ (resp. $V(\ttt)$), and write $\partial\ttt := \dot{\partial}\ttt \sqcup \hat{\partial} \ttt$.

		\end{defn}

		We now introduce the \textit{message configuration}, which enables us to calculate the size of a free tree (that is, number of \textsc{nae-sat} solutions on $\ttt$ that extends $\ux$) by local quantities. The message configuration is given by $\utau = (\tau_e)_{e\in E} \in \mathscr{M}^E$ ($\mathscr{M}$ is defined below). Here, $\tau_e=(\dot{\tau}_e,\hat{\tau}_e)$, where $\dot{\tau}$ (resp. $\hat{\tau}$) denotes the message from $v(e)$ to $a(e)$ (resp. $a(e)$ to $v(e)$).
		A message will carry information of the structure of the free tree it belongs to. To this end, we first define the notion of \textit{joining} $l$ trees at a vertex (either variable or clause) to produce a new tree.
		Let $t_1,\ldots , t_l$ be  a collection of rooted bipartite factor trees satisfying the following conditions:
		\begin{itemize}
			\item Their roots $\rho_1,\ldots,\rho_l$ are all of the same type (i.e., either all-variables or all-clauses) and are all degree one.
			
			\item 
			If an edge in $t_i$ is adjacent to a degree one vertex, which is not the root, then the edge is called a \textbf{boundary-edge}. The rest of the edges are called \textbf{internal-edges}. For the special case where $t_i$ consists of a single edge and a single vertex, we regard the single edge to be a boundary-edge.
			
			\item $t_1,\ldots,t_l$ are \textbf{boundary-labelled trees}, meaning that their variables, clauses, and internal edges are unlabelled (except we distinguish the root), but the boundary edges  are assigned with values from $\{0,1,\fs \}$, where $\fs$ stands for `separating'. 
		\end{itemize}
		
		Then, the joined tree $t \equiv  \textsf{j}(t_1,\ldots, t_l) $ is obtained by identifying all the roots as a single vertex $o$, and adding an edge which joins $o$ to a new root $o'$ of an opposite type of $o$ (e.g., if $o$ was a variable, then $o'$ is a clause). Note that $t= \textsf{j}(t_1,\ldots,t_l)$ is also a boundary-labelled tree, whose labels at the boundary edges are induced by those of $t_1,\ldots,t_l$.
		
		For the simplest trees that consist of single vertex and a single edge, we use $0$ (resp. $1$) to stand for the ones whose edge is labelled $0$ (resp. $1$): for the case of $\dot{\tau}$, the root is the clause, and for the case of $\hat{\tau}$, the root is the variable. Also,  if its root is a variable and its edge is labelled $\fs$, we write the tree as $\fs$. 
		
		We can also define the Boolean addition to a boundary-labelled tree $t$ as follows. For the trees $0,1$, the Boolean-additions $0\oplus \tL$, $1\oplus\tL$ are defined as above ($t\oplus \tL$), and we define $\fs \oplus \tL = \fs$ for $\tL\in\{0,1\}$. For the rest of the trees, $t \oplus 0 := t$, and $t\oplus 1$ is the boundary-labelled tree with the same graphical structure as $t$ and the labels of the boundary Boolean-added by $1$ (Here, we define $\fs \oplus 1 = \fs$ for the $\fs$-labels).
		
		\begin{defn}[Message configuration]\label{def:model:msg config}
			Let $\dot{\MMM}_0:= \{0,1,\star \}$ and $\hat{\MMM}_0:= \emptyset$. Suppose that $\dot{\MMM}_t, \hat{\MMM}_t$ are defined, and we inductively define $\dot{\MMM}_{t+1}, \hat{\MMM}_{t+1}$ as follows: For $\hat{\utau} \in (\hat{\MMM}_t)^{d-1}$, $\dot{\utau} \in (\dot{\MMM}_t)^{k-1}$, we write $\{\hat{\tau}_i \}:= \{\hat{\tau}_1,\ldots,\hat{\tau}_{d-1} \}$ and similarly for $\{\dot{\tau}_i \}$. We define
			\begin{eqnarray}
			\hat{T}\left(\dot{\utau} \right):= 
			\begin{cases}
			0 & \{\dot{\tau}_i \} = \{  1 \};\\
			1 & \{ \dot{\tau }_i \} = \{0  \};\\
			\fs & \{\dot{\tau}_i \} \supseteq \{ 0,1 \};\\
			\star & \star \in \{ \dot{\tau}_i \}, \{0,1 \} \nsubseteq \{\dot{\tau}_i \};\\
			\mathsf{j}\left(\dot{\utau} \right) & \textnormal{otherwise}, 
			\end{cases}
			&  
			\dot{T}(\hat{\utau}) :=
			\begin{cases}
			0 & 0 \in \{\hat{\tau}_i \} \subseteq \hat{\MMM}_t \setminus\{ 1\};\\
			1 &  1\in \{\hat{\tau}_i\} \subseteq \hat{\MMM}_t \setminus \{0 \};\\
			\tz & \{0,1 \} \subseteq \{\hat{\tau}_i \};\\
			\star & \star \in \{\hat{\tau}_i \} \subseteq \hat{\MMM}_t \setminus \{0,1\};\\
			\mathsf{j}\left(\hat{\utau} \right) & \{\hat{\tau}_i \} \subseteq \hat{\MMM}_t\setminus \{0,1,\star \}.
			\end{cases}
			\end{eqnarray}
			Further, we set $\dot{\MMM}_{t+1} := \dot{\MMM}_t \cup \dot{T}( \hat{\MMM}_t^{d-1}  ) \setminus \{\tz \}$, and $\hat{\MMM}_{t+1}:= \hat{\MMM}_t \cup \hat{T}(\dot{\MMM}_t^{k-1} )$, and define $\dot{\MMM}$ (resp. $\hat{\MMM}$) to be the union of all $\dot{\MMM}_t$ (resp. $\hat{\MMM}_t$) and $\MMM:= \dot{\MMM} \times \hat{\MMM}$. Then, a (valid) \textbf{message configuration} on $\GGG=(V,F,E,\uL)$ is a configuration $\utau \in \MMM^E$ that satisfies (i) the local equations given by
			\begin{equation}\label{eq:def:localeq:msg}
			\tau_e = (\dot{\tau}_e, \hat{\tau}_e) = \left(\dot{T}\big(\hat{\utau}_{\delta v(e)\setminus e} \big), \tL_e \oplus  \hat{T} \big((\uL + \dot{\utau})_{\delta a(e)\setminus e} \big) \right),
			\end{equation}
			for all $e\in E$, and $(ii)$ if one element of $\{\dot{\tau}_e,\hat{\tau}_e\}$ equals $\star$ then the other element is in $\{0,1\}$.
		\end{defn} 
		
		In the definition, $\star$ is the symbol introduced to cover  cycles, and $\tz$ is an error message. See Figure 1 in Section 2 of \cite{ssz16} for an example of $\star$ message. 
		
		When a frozen configuration $\ux$ on $\GGG$ is given, we can construct a message configuration $\utau$ via the following procedure:
		\begin{enumerate}
			\item For a forcing edge $e$, set $\hat{\tau}_e=x_{v(e)}$. Also, for an edge $e\in E$, if there exists $e^\prime \in \delta v(e) \setminus e$ such that $\hat{\tau}_{e^\prime} \in \{0,1\}$, then set $\dot{\tau}_e=x_{v(e)}$.
			\item For an edge $e\in E$, if there exist $e_1,e_2\in \delta a(e)\setminus e$ such that $\{\tL_{e_1}\oplus\dot{\tau}_{e_1}, \tL_{e_2}\oplus\dot{\tau}_{e_2}\}=\{0,1\}$, then set $\hat{\tau}_e = \fs$.
			
			\item After these steps, apply the local equations \eqref{eq:def:localeq:msg} recursively to define $\dot{\tau}_e$ and $\hat{\tau}_e$ wherever possible.
			
			\item For the places where it is no longer possible to define their messages until the previous step, set them to be $\star$.
		\end{enumerate}
		
		In fact, the following lemma shows the relation between the frozen and message configurations. We refer to \cite{ssz16}, Lemma 2.7 for its proof.
		
		\begin{lemma}\label{lem:model:bij:frozen and msg}
			The mapping explained above defines a bijection
			\begin{equation}\label{eq:frz msg 1to1}
			\begin{Bmatrix}
			\textnormal{Frozen configurations } \ux \in \{0,1,\ff \}^V\\
			\textnormal{without free cycles}
			\end{Bmatrix}
			\quad
			\longleftrightarrow
			\quad
			\begin{Bmatrix}
			\textnormal{Message configurations}\\
			\utau \in \MMM^E
			\end{Bmatrix}.
			\end{equation}
		\end{lemma}

		Next, we introduce a dynamic programming method based on \textit{belief propagation} to calculate the size of a free tree by local quantities from a message configuration. 
		
		\begin{defn}\label{def:msg:m}
			Let $\mathcal{P}\{0,1\} $ denote the space of probability measures on $\{0,1\}$. We define the mappings $\dot{\mm}:\dot{\MMM} \rightarrow \PP\{0,1\}$ and $\hat{\mm} : \hat{\MMM} \rightarrow \PP \{0,1\}$ as follows. For $\dot{\tau}\in\{0,1\}$ and $\hat{\tau}\in\{0,1\}$, let $\dot{\mm}[\dot{\tau}] =\delta_{\dot{\tau}}$, $\hat{\mm}[\hat{\tau}] = \delta_{\hat{\tau}}$. For $\dot{\tau}\in\dot{\MMM} \setminus \{0,1,\star\}$ and $\hat{\tau}\in\hat{\MMM} \setminus \{0,1,\star \}$, $\dot{\mm}[\dot{\tau}]$ and $\hat{\mm}[\hat{\tau}]$ are recursively defined:
			\begin{itemize}
				\item Let $\dot{\tau} = \dot{T}(\hat{\tau}_1,\ldots,\hat{\tau}_{d-1})$, with $\star \notin \{\hat{\tau}_i \}$. Define
				\begin{equation}\label{eq:def:bethe:bpmsg:dot}
				\dot{z}[\dot{\tau}] := \sum_{\bx\in\{0,1\} }
				\prod_{i=1}^{d-1} \hat{\mm}[\hat{\tau}_i](\bx), \quad 
				\dot{\mm}[\dot{\tau}](\bx) :=
				\frac{1}{\dot{z}[\dot{\tau}]} \prod_{i=1}^{d-1} \hat{\mm}[\hat{\tau}_i](\bx).
				\end{equation}
				Note that these equations are well-defined, since $(\hat{\tau}_1,\ldots, \hat{\tau}_{d-1})$ are well-defined up to permutation.

				\item Let $\hat{\tau} = \hat{T} ( \dot{\tau}_1,\ldots,\dot{\tau}_{k-1})$, with $\star \notin \{\dot{\tau}_i \}$. Define
				\begin{equation}\label{eq:def:bethe:bpmsg:hat}
				\hat{z}[\hat{\tau}] := 2-\sum_{\bx\in\{0,1\}} \prod_{i=1}^{k-1} \dot{\mm}[\dot{\tau}_i](\bx), \quad 
				\hat{\mm}[\hat{\tau}](\bx) :=
				\frac{1}{\hat{z}[\hat{\tau}]}
				\left\{1- \prod_{i=1}^{k-1} \dot{\mm}[\dot{\tau}_i](\bx) \right\} .
				\end{equation}
				Similarly as above, these equations are well-defined.
			\end{itemize}
			Moreover, observe that inductively, $\dot{\mm}[\dot{\tau}], \hat{\mm}[\hat{\tau}] $ are not Dirac measures unless $\dot{\tau}, \hat{\tau}\in \{0,1\}$. 
		\end{defn} 
		It turns out that $\dot{\mm}[\star], \hat{\mm}[\star]$ can be arbitrary measures for our purpose, and hence we assume that they are uniform measures on $\{0,1\}$.
		
		The equations \eqref{eq:def:bethe:bpmsg:dot} and \eqref{eq:def:bethe:bpmsg:hat} are known as \textit{belief propagation} equations. We refer the detailed explanation to \cite{ssz16}, Section 2 where the same notions are introduced, or to \cite{mm09}, Chapter 14 for more fundamental background. From these quantities, we define the following local weights which are going to lead us to computation of cluster sizes. 
		\begin{equation}\label{eq:def:phi}
		\begin{split}
		&\bar{\varphi} (\dot{\tau}, \hat{\tau}) := \bigg\{ \sum_{\bx\in \{0,1\}} \dot{\mm}[\dot{\tau}](\bx) \hat{\mm}[\hat{\tau}](\bx) \bigg\}^{-1}\,;\\
        &\hat{\varphi}^{\textnormal{lit}} (\dot{\tau}_1,\ldots, \dot{\tau}_k):= 1-\sum_{\bx \in \{0,1\}} \prod_{i=1}^k \dot{\mm}[\dot{\tau}_i](\bx)=\frac{\hat{z}\big(\hat{T}\big((\dot{\tau}_j)_{j\neq i})\big)\big)}{\bar{\varphi}\big(\dot{\tau}_i,\hat{T}\big((\dot{\tau}_j)_{j\neq i}\big)\big)}\,;\\
		&\dot{\varphi} (\hat{\tau}_1,\ldots,\hat{\tau}_d):=
		\sum_{\bx\in\{0,1\} }\prod_{i=1}^d \hat{\mm}[\hat{\tau}_i](\bx)=\frac{\dot{z}\big(\dot{T}\big((\hat{\tau}_j)_{j\neq i})\big)\big)}{\bar{\varphi}\big(\dot{T}\big((\hat{\tau}_j)_{j\neq i}\big), \hat{\tau}_i\big)}\,,
		\end{split}
		\end{equation}
		where the last identities in the last two lines hold for any choices of $i$. These weight factors can be used to derive the size of a free tree. Let $\ttt$ be a free tree in $\mathscr{F}(\ux,\GGG)$, and let $w^{\lit} (\ttt; \ux,\GGG)$ be the number of \textsc{nae-sat} solutions that extend $\ux$ to $\{0,1\}^{V(\ttt)}$.  Further, let $\textsf{size}(\ux,\GGG)$ denote the total number of \textsc{nae-sat} solutions that extend $\ux$ to $\{0,1\}^V.$
		\begin{lemma}[\cite{ssz16}, Lemma 2.9 and Corollary 2.10; \cite{mm09}, Ch. 14]\label{lem:size:msg and trees}
			Let $\ux$ be a frozen configuration on $\GGG=(V,F,E,\uL)$ without any free cycles, and $\utau$ be the corresponding message configuration. For a free tree $\ttt \in \mathscr{F}(\ux;\GGG)$, we have that
			\begin{equation}\label{eq:free:tree:weight:lit}
			w^{\textnormal{lit}}(\ttt,\ux,\GGG)= \prod_{v\in V(\ttt)} \left\{ \dot{\varphi}(\hat{\utau}_{\delta v}) \prod_{e\in \delta v} \bar{\varphi}(\tau_e) \right\} \prod_{a\in F(\ttt)} \hat{\varphi}^{\textnormal{lit}}\big( (\dot{\utau} \oplus \uL)_{\delta a} \big).
			\end{equation}
			Furthermore, let $\Upsilon \in \textsf{CL}(\GGG)$ be the cluster corresponding to $\ux$. Then, we have
			\begin{equation*}
			\textsf{size}(\ux;\GGG)=	|\Upsilon| = \prod_{v\in V} \dot{\varphi} (\hat{\utau}_{\delta v}) \prod_{a\in F} \hat{\varphi}^{\textnormal{lit}}\big((\dot{\utau}\oplus \uL)_{\delta a} \big) \prod_{e \in E} \bar{\varphi} (\tau_e).
			\end{equation*} 
		\end{lemma}
		
		\subsection{Colorings}\label{subsubsec:model:col}
		In this subsection, we introduce the \textit{coloring configuration}, which is a simplification of the message configuration. We give its definition analogously as in \cite{ssz16}.

		Recall the definition of $\MMM=\dot{\MMM}\times \hat{\MMM}, $ and let $\{ \fF \} \subset \MMM$ be defined by $\{\fF\}:= \{\tau \in \MMM: \, \dot{\tau} \notin \{ 0,1,\star\} \textnormal{ and } \hat{\tau}\notin \{ 0,1,\star\} \}$.
		Note that $\{\fF \}$ corresponds to the messages on the edges of free trees, except the boundary edges labelled either 0 or 1.
		Define $\Omega :=  \{\rr_0, \rr_1, \bb_0, \bb_1\} \cup \{\fF \}$ and let $\textsf{S}: \MMM\setminus\{(\star,\star)\} \to \Omega$ be the projections given by
		\begin{equation*}
		\textsf{S}(\tau) := 
		\begin{cases}
		\rr_0 & \hat{\tau}=0;\\
		\rr_1 & \hat{\tau}=1;\\
		\bb_0 & \hat{\tau } \neq 0, \, \dot{\tau}=0;\\
		\bb_1 & \hat{\tau} \neq 1, \, \dot{\tau}=1;\\
		\tau & \textnormal{otherwise, i.e., } \tau \in \{ \fF \}.
		\end{cases}
		\end{equation*}
        Here, we note that a (valid) message configuration $\utau=(\tau_e)_{e\in E} \in \MMM^{E}$ cannot have an edge $e$ such that $\tau_e=(\star,\star)$ (see Definition \ref{def:model:msg config}), thus we may safely exclude the spin $(\star,\star)$ from $\MMM$.
        
		For convenience, we abbreviate $\{\rr \}= \{\rr_0, \rr_1 \}$ and $\{\bb \} = \{\bb_0, \bb_1 \}$, and  define the Boolean addition  as $\bb_\bx \oplus \tL := \bb_{\bx \oplus \tL}$, and similarly for $\rr_\bx$. Also, for $\sigma \in \{ \rr,\bb,\fs\}$, we set $\dot{\sigma} :=  \sigma=:\hat{\sigma}$.
		
		
		
		\begin{defn}[Colorings]\label{def:model:col}
			For $\sig \in \Omega ^d$, let
			\begin{equation*}
			\dot{I}(\sig) : =
			\begin{cases}
			1 & \rr_0 \in \{\sigma_i\} \subseteq \{\rr_0, \bb_0 \};\\
			1& \rr_1 \in \{\sigma_i\} \subseteq \{\rr_1,\bb_1 \};\\
			1 & \{\sigma_i \} \subseteq \{\fs\}\cup  \{ {\fF} \}, \textnormal{ and } \dot{\sigma}_i = \dot{T}\big( (\hat{\sigma}_j)_{j\neq i};0 \big), \ \forall i;\\
			0 & \textnormal{otherwise}.
			\end{cases}
			\end{equation*}
			Also, define $	\hat{I}^{\textnormal{lit}}: \Omega^k \to \mathbb{R}$ to be
			\begin{equation*}
			\begin{split}
			\hat{I}^{\textnormal{lit}}(\sig)&:=
			\begin{cases}
			1 & \exists i:\, \sigma_i = \rr_0 \textnormal{ and } \{\sigma_j \}_{j\neq i} = \{\bb_1 \};\\
			1 & \exists i:\, \sigma_i = \rr_1 \textnormal{ and } \{\sigma_j \}_{j\neq i} = \{\bb_0 \};\\
			1 & \{\bb \} \subseteq \{\sigma_i \} \subseteq \{\bb\} \cup \{\sigma \in \{\fF \}: \,\hat{\sigma}=\fs \};\\
			1 & \{\sigma_i \} \subseteq \{ \bb_0, {\fF} \}, \, |\{i: \sigma_i\in  \{{\fF} \}\} | \ge 2 , \textnormal{ and } \hat{\sigma}_i = \hat{T}((\dot{\sigma}_j)_{j\neq i}; 0), \ \forall i \textnormal{ s.t. } \sigma_i \neq \bb_0;\\
			1 & \{\sigma_i \} \subseteq \{ \bb_1, {\fF} \}, \, |\{i: \sigma_i\in  \{{\fF} \}\}| \ge 2 , \textnormal{ and } \hat{\sigma}_i = \hat{T}((\dot{\sigma}_j)_{j\neq i}; 0), \ \forall i \textnormal{ s.t. } \sigma_i \neq \bb_1;\\
			0 & \textnormal{otherwise}.
			\end{cases}
			\end{split}
			\end{equation*}
			On a \textsc{nae-sat} instance $\GGG = (V,F,E,\uL)$, $\sig\in \Omega^E$ is a (valid) \textbf{coloring} if $\dot{I}(\sig_{\delta v})=\hat{I}^{\textnormal{lit}}((\sig\oplus\uL)_{\delta a}) =1 $ for all $v\in V, a\in F$. 
			
		\end{defn}
		
		Given \textsc{nae-sat} instance $\GGG$, it was shown in Lemma 2.12 of \cite{ssz16} that there is a bijection
		\begin{equation}\label{eq:msg col 1to1}
		\begin{Bmatrix}
		\textnormal{message configurations}\\
		\utau \in \MMM^E
		\end{Bmatrix}
		\ \longleftrightarrow \
		\begin{Bmatrix}
		\textnormal{colorings} \\
		\sig \in \Omega^E
		\end{Bmatrix}\,.
		\end{equation}
		The weight elements for coloring, denoted by $\dot{\Phi}, \hat{\Phi}^{\textnormal{lit}}, \bar{\Phi}$, are defined as follows. For $\sig \in \Omega^d,$ let
		\begin{equation*}
		\begin{split}
		\dot{\Phi}(\sig) := 
		\begin{cases}
		\dot{\varphi}(\hat{\sig}) & \dot{I}(\sig) =1 \textnormal{ and } \{\sigma_i \} \subseteq \{\fF \};\\
		1 & \dot{I}(\sig) =1 \textnormal{ and } \{\sigma_i \}\subseteq \{\bb, \rr \};\\
		0 & \textnormal{otherwise, i.e., } \dot{I}(\sig)=0.
		\end{cases}
		\end{split}
		\end{equation*}
		For $\sig \in \Omega^k$, let
		\begin{equation*}
		\hat{\Phi}^{\textnormal{lit}}(\sig) :=
		\begin{cases}
		\hat{\varphi}^\lit((\dot{\tau}(\sigma_i))_i)
		& \hat{I}^{\textnormal{lit}} (\sig) = 1 \textnormal{ and } \{\sigma_i \} \cap \{\rr \} = \emptyset;\\
		1 & \hat{I}^{\textnormal{lit}}(\sig) = 1 \textnormal{ and } \{\sigma_i \} \cap \{\rr \} \neq \emptyset;\\
		0 & \textnormal{otherwise, i.e., } \hat{I}^{\textnormal{lit}}(\sig)=0.
		\end{cases}
		\end{equation*}
		(If $\sigma \notin \{\rr \}, $ then $\dot{\tau}(\sigma_i)$ is well-defined.)
		Lastly, let
		\begin{equation*}
		\bar{\Phi}(\sigma) := 
		\begin{cases}
		\bar{\varphi} (\sigma) & \sigma \in \{\fF \};\\
		1 & \sigma \in \{\rr, \bb \}.
		\end{cases}
		\end{equation*}
		Note that if $\hat{\sigma}=\fs$, then $\bar{\varphi}(\dot{\sigma},\hat{\sigma})=2$ for any $\dot{\sigma}$. 
		The rest of the details explaining the compatibility of $\varphi$ and $\Phi$ can be found in \cite{ssz16}, Section 2.4.
		Then, the formula for the cluster size we have  seen in Lemma \ref{lem:size:msg and trees} works the same for the coloring configuration.
		
		\begin{lemma}[\cite{ssz16}, Lemma 2.13]\label{lem:model:size:col}
			Let 	$\ux \in \{0,1,\ff \}^V$ be a frozen configuration on $\GGG=(V,F,E,\uL)$, and let $\sig \in \Omega^E$ be the corresponding coloring. Define
			\begin{equation*}
			w_\GGG^{\textnormal{lit}}(\sig):= \prod_{v\in V}\dot{\Phi}(\sig_{\delta v}) \prod_{a\in F} \hat{\Phi}^{\textnormal{lit}} ((\sig \oplus \uL)_{\delta a}) \prod_{e\in E} \bar{\Phi}(\sigma_e).
			\end{equation*}
			Then, we have $\textsf{size}(\ux;\GGG) = w_\GGG^{\textnormal{lit}}(\sig)$.\
		\end{lemma}
		Among the valid frozen configurations, we can ignore the contribution from the configurations with too many free or red colors, as observed in the following lemma.
		\begin{lemma}[\cite{dss16} Proposition 2.2 ,\cite{ssz16} Lemma 3.3]
			For a frozen configuration $\ux \in \{0,1,\ff \}^{V}$, let $\rr(\ux)$ count the number of forcing edges and $\ff(\ux)$ count the number of free variables.
			There exists an absolute constant $c>0$ such that for $k\geq k_0$, $\alpha \in [\alpha_{\textsf{lbd}}, \alpha_{\textsf{ubd}}]$, and $\lambda \in(0,1]$, 
			\begin{equation*}
			\sum_{\ux \in \{0,1,\ff\}^V} \E \left[	\textsf{size}(\ux;\GGG)^\lambda\right] \mathds{1}\left\{ \frac{\rr(\ux)}{nd} \vee \frac{\ff(\ux)}{n}> \frac{7}{2^k} \right\}  \le e^{-cn},
			\end{equation*}
		where $\textsf{size}(\ux;\GGG)$ is the number of \textsc{nae-sat} solutions $\ubx \in \{0,1\}^{V}$ which extends $\ux\in \{0,1,\ff\}^{V}$.
			
		\end{lemma}
		
		Thus, our interest is in counting the number of frozen configurations and colorings such that the fractions of red edges and the fraction of free variables are bounded by $7/2^k$. To this end, we define
		\begin{equation}\label{eq:def:Z:lambda}
		\begin{split}
		&\bZ_{\la}:=	\sum_{\ux \in \{0,1,\ff\}^V} 	\textsf{size}(\ux;\GGG)^\lambda \mathds{1}\left\{ \frac{\rr(\ux)}{nd} \vee \frac{\ff(\ux)}{n}\leq \frac{7}{2^k} \right\};\quad \bZ_{\lambda}^{\tr}:= \sum_{\sig \in \Omega^E}w_\GGG^{\textnormal{lit}} (\sig)^\lambda \mathds{1}\left\{ \frac{\rr(\sig)}{nd} \vee \frac{\ff(\sig)}{n} \le \frac{7}{2^k} \right\} ;\\
		&\bZ_{\lambda,s}:= 	\sum_{\ux \in \{0,1,\ff\}^V} 	\textsf{size}(\ux;\GGG)^\lambda \mathds{1}\left\{ \frac{\rr(\ux)}{nd} \vee \frac{\ff(\ux)}{n}\leq \frac{7}{2^k}, ~~~~~e^{ns} \le \textsf{size}(\ux;\GGG) < e^{ns+1}\right\};\\
		&\bZ_{\lambda,s}^{\tr}:= \sum_{\sig\in \Omega^E} w_\GGG^{\textnormal{lit}}(\sig)^{\la} \mathds{1}\left\{ \frac{\rr(\sig)\vee \fs(\sig)}{nd} \le \frac{7}{2^k},~~~~~e^{ns} \le w_\GGG^{\textnormal{lit}}(\sig) < e^{ns+1} \right\},
		\end{split}
		\end{equation}
        where $\rr(\sig)$ and $\fs(\sig)$ respectively counts the number of edges $e\in E$ such that $\sigma_e=\rr$ and $\sigma_e=\fs$ in the coloring $\sig$. $\ff(\sig)$ count the number of free variables in the frozen configuration $\underline{x}$ corresponding to the coloring $\sig$ via the bijections~\eqref{eq:frz msg 1to1} and \eqref{eq:msg col 1to1}. The superscript $\tr$ is to emphasize that the above quantities count the contribution from frozen configurations which only contain free trees, i.e. no free cycles (Recall that by Lemma \ref{lem:model:bij:frozen and msg} and \eqref{eq:msg col 1to1}, the space of coloring has a bijective correspondence with the space of frozen configurations without free cycles). Similarly, recalling the definition of $\overline{\bN}_s$ in \eqref{eq:def:bN:elementary}, total number of clusters of size in $[e^{ns},e^{ns+1})$ , $\bN_s$ is defined to be
	\begin{equation*}
	    \bN_s:=\bZ_{0,s}\quad\textnormal{and}\quad \bN_s^{\tr}:=\bZ^{\tr}_{0,s}.
	\end{equation*}
	Hence, $e^{-n\la s-\la}\bZ_{\la,s}\leq \bN_{s}\leq e^{-n\la s} \bZ_{\la,s}$ holds.
	
		\begin{defn}[Truncated colorings]
			Let $1\leq L< \infty$,  $\ux $ be a frozen configuration on $\GGG$ without free cycles and $\sig\in \Omega^E$ be the coloring corresponding to $\ux$. Recalling the notation $\mathscr{F}(\ux;\GGG)$ (Definition \ref{def:freetree:basic}), we say $\sig$ is a (valid) $L$-\textbf{truncated coloring} if $|V(\ttt)| \le L$ for all $\ttt \in \mathscr{F}(\ux;\GGG)$. For an equivalent definition, let $|\sigma|:=v(\dot{\sigma})+v(\hat{\sigma})-1$ for $\sigma \in \{\fF\}$, where $v(\dot{\sigma})$ (resp. $v(\hat{\sigma})$) denotes the number of variables in $\dot{\sigma}$ (resp. $\hat{\sigma}$). Define $\Omega_L := \{\rr,\bb \}\cup\{\fF \}_L$, where $\{\fF \}_L$ is the collection of $\sigma\in \{\fF \}$ such that $|\sigma| \le L$. Then,  $\sig$ is a (valid) $L$-truncated coloring if $\sig \in \Omega_L^E$.
			
			To clarify the names, we often call the original coloring $\sig\in \Omega^E$ the \textbf{untruncated coloring}.
		\end{defn}
		Analogous to \eqref{eq:def:Z:lambda}, define the truncated partition function
		\begin{equation*}
		\begin{split}
		&\bZ_{\lambda}^{(L),\tr} := \sum_{\sig \in \Omega_L^E}w_\GGG^{\textnormal{lit}} (\sig)^\lambda \mathds{1}\left\{ \frac{\rr(\sig)}{nd} \vee \frac{\ff(\sig)}{n} \le \frac{7}{2^k} \right\} ;\\
		&\bZ_{\lambda,s}^{(L), \tr}:= \sum_{\sig\in \Omega_L^E} w_\GGG^{\textnormal{lit}}(\sig)^{\la} \mathds{1}\left\{ \frac{\rr(\sig)}{nd} \vee \frac{\ff(\sig)}{n} \le \frac{7}{2^k},~~~~~e^{ns} \le w_\GGG^{\textnormal{lit}}(\sig) < e^{ns+1} \right\}.
		\end{split}
		\end{equation*}

		\subsection{Averaging over the literals}\label{subsubsec:model:avglit:1stmo}
		Let $\GGG=(V,F,E,\uL)$ be a \textsc{nae-sat} instance and $\GG=(V,F,E)$ be the factor graph without the literal assignment. Let $\mathbb{E}^{\textnormal{lit}}$ denote the expectation over the literals $\uL \sim \textnormal{Unif} [\{0,1\}^E]$. Then, for a coloring $\sig \in \Omega^{E}$, we can use Lemma \ref{lem:model:size:col} to write $\mathbb{E}^{\textnormal{lit}}[w_\GGG^{\textnormal{lit}}(\sig) ]$ as 
		\begin{equation*}
		w_{\GGG}(\sig)^{\la}:=\E^{\textnormal{lit}} [ w_\GGG^{\textnormal{lit}}(\sig)^\lambda] = \prod_{v\in V} \dot{\Phi}(\sig_{\delta v})^\lambda \prod_{a\in F} \E^{\textnormal{lit}} \hat{\Phi}^{\textnormal{lit}}((\sig\oplus \uL)_{\delta a})^\lambda \prod_{e\in E} \bar{\Phi}(\sigma_e)^\lambda.
		\end{equation*}
		To this end, define $$\hat{\Phi}(\sig_{\delta a})^\lambda := \E^{\textnormal{lit}}[ \hat{\Phi}^{\textnormal{lit}}((\sig\oplus\uL)_{\delta a})^\lambda]. $$ We now recall a property of $\hat{\Phi}^{\lit}$ from \cite{ssz16}, Lemma 2.17:
		\begin{lemma}[\cite{ssz16}, Lemma 2.17]\label{lem:decompose:Phi:hat}
			$\hat{\Phi}^{\lit}$ can be factorized as $\hat{\Phi}^{\lit}(\sig\oplus\uL) = \hat{I}^{\lit}(\sig \oplus \uL) \hat{\Phi}^{\textnormal{m}}(\sig)$ for 
			\begin{equation}\label{eq:def:Phi:hat:max}
			\hat{\Phi}^{\textnormal{m}}(\sig) := \max\big\{\hat{\Phi}^{\lit}(\sig\oplus\uL): \uL \in \{0,1\}^k \big\}=
			\begin{cases}
			1 & \sig \in \{\rr,\bb\}^{k},\\
			\frac{\hat{z}[\hat{\sigma}_j]}{\bar{\varphi}(\sigma_j)} &\sig \in \Omega^{k}\textnormal{ with } \sigma_j \in \{\fF\}. 
			\end{cases}
			\end{equation}
		\end{lemma}
		As a consequence, we can write $\hat{\Phi}(\sig)^\lambda = \hat{\Phi}^{\textnormal{m}}(\sig)^\lambda \hat{v}(\sig)$, where
		\begin{equation}\label{eq:def:vhat:basic}
		\hat{v}(\sig) := \E^{\lit} [ \hat{I}^{\lit}(\sig\oplus\uL)]. 
		\end{equation}

			\subsection{Empirical profile of colorings} The \textit{coloring profile}, defined below, was introduced in \cite{ssz16}. Hereafter, $\mathscr{P}(\mathfrak{X})$ denotes the space of probability measures on $\mathfrak{X}$.
	
	\begin{defn}[coloring profile and the simplex of coloring profile, Definition 3.1 and 3.2 of \cite{ssz16}]\label{def:empiricalssz}
		Given a \textsc{nae-sat} instance $\GGG$ and a coloring configuration $\sig \in \Omega^E $, the \textit{coloring profile} of $\underline{\sigma}$ is the triple $H[\underline{\sigma}]\equiv H\equiv (\dot{H},\hat{H},\bar{H}) $ defined as follows. 
		\begin{equation*}
		\begin{split}
		\dot{H}\in \mathscr{P}(\Omega^d), \quad &\dot{H}(\underline{\tau})
		 =
		  |\{v\in V: \underline{\sigma}_{\delta v}=\underline{\tau} \} | / |V| \quad
		  \textnormal{for all } \underline{\tau}\in \Omega^d;\\
		  \hat{H}\in \mathscr{P}(\Omega^k), \quad &\hat{H}(\underline{\tau})
		  =
		  |\{a\in F: \underline{\sigma}_{\delta a}=\underline{\tau} \} | / |F| \quad
		  \textnormal{for all } \underline{\tau}\in \Omega^k;\\
		  \bar{H}\in \mathscr{P}(\Omega), \quad &\bar{H}(\tau)
		  =
		  |\{e\in E: \sigma_e=\tau \} | / |E| \quad
		  \textnormal{for all } \tau \in \Omega.
		\end{split}
		\end{equation*}
	A valid $H$ must satisfy the following compatibility equation:
	\begin{equation}\label{eq:compatibility:coloringprofile}
	       \frac{1}{d} \sum_{\utau \in \Omega^{d}}\dot{H}(\utau)\sum_{i=1}^{d}\one\{\tau_i=\tau\} = \bar{H}(\tau) = \frac{1}{k}\sum_{\utau \in \Omega^k}\hat{H}(\utau)\sum_{j=1}^{k}\one\{\tau_j = \tau \}\quad\textnormal{for all}\quad \tau \in \Omega\,.
	\end{equation}
	The \textit{simplex of coloring profile} $\bDelta$ is the space of triples $H=(\dot{H},\hat{H},\bar{H})$ which satisfies the following conditions:
	\begin{enumerate}
	\item[$\bullet$] $\dot{H} \in \mathscr{P}(\textnormal{supp}\,\dot{\Phi}), \hat{H} \in \mathscr{P}(\textnormal{supp}\,\hat{\Phi})$ and $\bar{H} \in \mathscr{P}(\Omega)$.
	\item[$\bullet$] $\dot{H},\hat{H}$ and $\bar{H}$ satisfy \eqref{eq:compatibility:coloringprofile}.
	\item[$\bullet$] Recalling the definition of $\bZ_{\la}$ in \eqref{eq:def:Z:lambda}, $\dot{H},\hat{H}$ and $\bar{H}$ satisfy $\max\{\bar{H}(\ff),\bar{H}(\rr)\} \leq \frac{7}{2^k}$.
	\end{enumerate}
	For $L <\infty$, we let $\bDelta^{(L)}$ be the subspace of $\bDelta$ satisfying the following extra condition:
	\begin{enumerate}[resume]
	    \item[$\bullet$] $\dot{H} \in \mathscr{P}(\textnormal{supp}\,\dot{\Phi}\cap \Omega_{L}^{d}), \hat{H} \in \mathscr{P}(\textnormal{supp}\,\hat{\Phi}\cap \Omega_L^{k})$ and $\bar{H} \in \mathscr{P}(\Omega_L)$.
	\end{enumerate}
	\end{defn}
	Given a coloring profile $H\in \bDelta$, denote $\bZ_{\lambda}^{\tr}[H]$ by the contribution to $\bZ_{\lambda}^{\tr}$ from the coloring configurations whose coloring profile is $H$. That is, $\bZ_\lambda^{\tr}[H] := \sum_{\underline{\sigma}:\; H[\underline{\sigma}] = H} w^{\lit}(\underline{\sigma})^\lambda$. For $H \in \bDelta^{(L)}$, $\tZ_{\lambda}[H]$ is analogously defined. In \cite{ssz16}, they showed that $\E \tZ_\lambda[H]$ for the \textit{L-truncated} coloring model can be written as the following formula, which is a result of Stirling's approximation:
	\begin{equation}\label{eq:1stmo dec by H}
	\begin{split}
	&\E \tZ_\lambda[H] = n^{O_{L}(1)} \exp\left\{n F_{\lambda,L}(H)\right\}\quad\textnormal{for}\quad F_{\la,L}(H):=\Sigma(H)+\la s(H),\quad H\in \bDelta^{(L)},\quad\textnormal{where}\\
	&\Sigma(H):=\sum_{\sig\in \Omega^{d}}\dot{H}(\sig) \log\Big(\frac{1}{\dot{H}(\sig)}\Big)
	+ \frac{d}{k} \sum_{\sig\in \Omega^{k}} \hat{H}(\sig) \log\Big(\frac{\hat{v}(\sig)}{\hat{H}(\sig)} \Big)
	+ d\sum_{\sigma\in \Omega}  \bar{H}(\sigma) \log\big(\bar{H}(\sigma)\big)\quad\textnormal{and}\\
	&s(H):=\sum_{\sig\in \Omega^{d}}\dot{H}(\sig) \log\big(\dot{\Phi}(\sig)\big)
	+ \frac{d}{k} \sum_{\sig\in \Omega^{k}} \hat{H}(\sig) \log\big(\hat{\Phi}^{\textnormal{m}}(\sig) \big)
	+ d\sum_{\sigma\in \Omega}  \bar{H}(\sigma) \log\big(\bar{\Phi}(\sigma)\big).
	\end{split}
	\end{equation}
	Similar to $F_{\la,L}(H)$ for $H\in \bDelta^{(L)}$, the untruncated free energy $F_{\la}(H)$ for $H\in \bDelta$ is defined by the same equation $F_{\la}(H):=\Sigma(H)+\la s(H)$.
	\subsection{Belief propagation fixed point and optimal profiles}
	It was proven in \cite{ssz16} that the truncated free energy $F_{\la,L}(H)$ is maximized at the \textit{optimal profile} $H^\star_{\la,L}$, defined in terms of \textit{Belief Propagation(BP) fixed point}. In this subsection, we review the necessary notions to define $H^\star_{\la,L}$ (cf. Section 5 of \cite{ssz16}). To do so, we first define the BP functional $\dot{\textnormal{BP}}_{\lambda,L}:\PPP(\hat{\Omega}_{L}) \rightarrow \PPP(\dot{\Omega}_{L})$ and $\hat{\textnormal{BP}}_{\lambda,L}:\PPP(\dot{\Omega}_L) \rightarrow \PPP(\hat{\Omega }_L)$. For $\hat{q}\in \PPP(\hat{\Omega}_L)$ and $\dot{q} \in \PPP(\dot{\Omega}_L)$, let the probability measures $\dot{\textnormal{BP}}_{\la,L}(\hat{q})\in \PPP(\dot{\Omega}_L)$ and $\hat{\textnormal{BP}}_{\la,L}(\dot{q})\in \PPP(\hat{\Omega}_L)$ be defined as follows. For $\dot{\sigma}\in \dot{\Omega}_L$ and $\hat{\sigma}\in \hat{\Omega}_L$, define 
    \begin{equation}\label{eq:def:BP}
    \begin{split}
    &[\dot{\textnormal{BP}}_{\la,L}(\hat{q})](\dot{\sigma})=\big(\dot{\ZZZ}_{\hat{q}}\big)^{-1} \cdot \bar{\Phi}(\dot{\sigma}, \hat{\sigma}^\prime)^\lambda \sum_{\sig \in \Omega_L^{d}}\one\{\sigma_1=(\dot{\sigma},\hat{\sigma}^\prime)\}\dot{\Phi}(\sig)^{\lambda}\prod_{i=2}^{d}\hat{q}(\hat{\sigma}_i)\,,\\
    &[\hat{\textnormal{BP}}_{\la,L}(\dot{q})](\hat{\sigma})=\big(\hat{\ZZZ}_{\dot{q}}\big)^{-1} \cdot \bar{\Phi}(\dot{\sigma}^\prime, \hat{\sigma})^\lambda \sum_{\sig \in \Omega_L^{k}}\one\{\sigma_1=(\dot{\sigma}^\prime, \hat{\sigma})\}\hat{\Phi}(\sig)^{\lambda}\prod_{i=2}^{k}\dot{q}(\dot{\sigma}_i)\,,
    \end{split}
    \end{equation}
    where $\hat{\sigma}^\prime \in \hat{\Omega}_L$ and $\dot{\sigma}^\prime \in \dot{\Omega}_L$ are arbitrary with the only exception that when $\dot{\sigma}\in \{\rr,\bb\}$ (resp. $\hat{\sigma}\in \{\rr,\bb\}$), then we take $\hat{\sigma}^\prime = \dot{\sigma}$ (resp. $\dot{\sigma}^\prime = \hat{\sigma}$) so that the RHS above is non-zero. From the definition of $\dot{\Phi},\hat{\Phi}$, and $\bar{\Phi}$, it can be checked that the choices of $\hat{\sigma}^\prime \in \hat{\Omega}_L$ and $\dot{\sigma}^\prime \in \dot{\Omega}_L$ do not affect the values of the RHS above (see \eqref{eq:def:phi}). The normalizing constants $\dot{\ZZZ}_{\hat{q}}$ and $\hat{\ZZZ}_{\dot{q}}$ are given by
    \begin{equation}\label{eq:BP:normalization}
 \begin{split}
 &\dot{\ZZZ}_{\hat{q}}\equiv\sum_{\dot{\sigma} \in \dot{\Omega}_L} \bar{\Phi}(\dot{\sigma}, \hat{\sigma}^\prime)^\lambda \sum_{\sig \in \Omega_L^{d}}\one\{\sigma_1=(\dot{\sigma},\hat{\sigma}^\prime)\}\dot{\Phi}(\sig)^{\lambda}\prod_{i=2}^{d}\hat{q}(\hat{\sigma}_i)\,,\\
 &\hat{\ZZZ}_{\dot{q}}\equiv \sum_{\hat{\sigma} \in \hat{\Omega}_L}\bar{\Phi}(\dot{\sigma}^\prime, \hat{\sigma})^\lambda \sum_{\sig \in \Omega_L^{k}}\one\{\sigma_1=(\dot{\sigma}^\prime, \hat{\sigma})\}\hat{\Phi}(\sig)^{\lambda}\prod_{i=2}^{k}\dot{q}(\dot{\sigma}_i)\,.
 \end{split}
 \end{equation}
Here, $\hat{\sigma}^\prime \in \hat{\Omega}_L$ and $\dot{\sigma}^\prime \in \dot{\Omega}_L$ are again arbitrary. We then define the \textit{Belief Propagation functional} by $\textnormal{BP}_{\lambda,L}:= \dot{\textnormal{BP}}_{\lambda,L}\circ \hat{\textnormal{BP}}_{\lambda,L}$. The untruncated BP map, which we denote by $\textnormal{BP}_{\lambda}:\PPP(\dot{\Omega}) \to \PPP(\dot{\Omega})$, is analogously defined, where we replace $\dot{\Omega}_L$(resp. $\hat{\Omega}_L$) with $\dot{\Omega}$(resp. $\hat{\Omega}$).
\begin{remark}
    In defining the untruncated BP map, note that $\dot{\Omega}$ and $\hat{\Omega}$ are not a finite set, thus the normalizing constant, analogue of \eqref{eq:BP:normalization}, is not obviously finite. However, from the definitions of $\dot{\Phi},\hat{\Phi}$, and $\bar{\Phi}$, we have that $\bar{\Phi}(\sigma_1)\dot{\Phi}(\sig)\leq 2$ and $\bar{\Phi}(\tau_1)\hat{\Phi}(\utau)\leq 2$ for $\sig=(\sigma_1,\ldots,\sigma_d) \in \Omega^{d}$ and $\utau=(\tau_1,\ldots,\tau_k) \in \Omega^k$. Thus, it follows that the normalizing constants for the untruncated BP map are at most $2$. We also remark that $\sig=((\dot{\sigma}_1,\hat{\sigma}_1),\ldots,(\dot{\sigma}_d,\hat{\sigma}_d))\in \Omega^{d}$ such that $\dot{\Phi}(\sig)\neq 0$ is fully determined by $(\dot{\sigma}_1,\hat{\sigma}_1)$ and $\hat{\sigma}_2,\ldots,\hat{\sigma}_d$. Thus, the second sum $\sig\in \Omega_L^d$ in the definition of $\dot{\ZZZ}_{\hat{q}}$ in \eqref{eq:BP:normalization} can be replaced with the sum over $\sigma_1\in \Omega, \hat{\sigma}_2,\ldots, \hat{\sigma}_d\in \hat{\Omega}$. The analogous remark holds for the $\hat{\ZZZ}_{\dot{q}}$ and for the untruncated model.
\end{remark}
Let $\mathbf{\Gamma}_C$ be the set of $\dot{q} \in \PPP(\dot{\Omega})$ such that 
	\begin{equation}\label{eq:def:bp:contract:set:1stmo}
	    \dot{q}(\dot{\sigma})=\dot{q}(\dot{\sigma}\oplus 1)\quad\text{for}\quad\dot{\sigma} \in \dot{\Omega },\quad\text{and}\quad \frac{\dot{q}(\rr)+2^k\dot{q}(\ff)}{C}\leq \dot{q}(\bb) \leq \frac{\dot{q}(\rr)}{1-C2^{-k}},
	\end{equation}
	where $\{\rr\}\equiv\{\rr_0,\rr_1\}$ and $\{\bb\}\equiv \{\bb_0,\bb_1\}$. The proposition below shows that the BP map contracts in the set $\mathbf{\Gamma}_C$ for large enough $C$, which guarantees the existence of \textit{Belief Propagation fixed point}.
	\begin{prop}[Proposition 5.5 item a,b of \cite{ssz16}]
	\label{prop:BPcontraction:1stmo}
	For $\lambda \in [0,1]$, the following holds:
	\begin{enumerate}
	    \item There exists a large enough universal constant $C$ such that the map $\textnormal{BP}\equiv\textnormal{BP}_{\lambda,L}$ has a unique fixed point $\dot{q}^\star_{\lambda,L}\in \mathbf{\Gamma}_C$. Moreover, if $\dotq \in \mathbf{\Gamma}_C$, $\textnormal{BP}\dotq \in \mathbf{\Gamma}_C$ holds with
	    \begin{equation}\label{eq:BPcontraction:1stmo}
	        ||\textnormal{BP}\dotq-\dotq^\star_{\lambda,L}||_1\lesssim k^2 2^{-k}||\dotq-\dotq^\star_{\lambda,L}||_1.
	    \end{equation}
	    The same holds for the untruncated BP, i.e. $\textnormal{BP}_{\la}$, with fixed point $\dot{q}^\star_{\lambda}\in \Gamma_C$. $\dot{q}^\star_{\la,L}$ for large enough $L$ and $\dot{q}^\star_{\la}$ have full support in their domains.
	    \item In the limit $L \to \infty$, $||\dot{q}^\star_{\lambda,L}-\dot{q}^\star_{\lambda}||_1 \to 0$.
	\end{enumerate}
	\end{prop}
	For $\dot{q} \in \PPP(\dot{\Omega})$, denote $\hat{q}\equiv \hat{\textnormal{BP}}\dot{q}$, and define $H_{\dot{q}}=(\dot{H}_{\dot{q}},\hat{H}_{\dot{q}}, \bar{H}_{\dot{q}})\in \bDelta$ by
	\begin{equation}\label{eq:H:q:1stmo}
	    \dot{H}_{\dot{q}}(\sig)=\frac{\dot{\Phi}(\sig)^{\lambda}}{\dot{\mathfrak{Z}}}\prod_{i=1}^{d}\hat{q}(\hat{\sigma}_i),\quad \hat{H}_{\dot{q}}(\sig)=\frac{\hat{\Phi}(\sig)^{\lambda}}{\hat{\mathfrak{Z}}}\prod_{i=1}^{k}\dot{q}(\dot{\sigma}_i),\quad \bar{H}_{\dot{q}}(\sigma)=\frac{\bar{\Phi}(\sigma)^{-\lambda}}{\bar{\mathfrak{Z}}}\dot{q}(\dot{\sigma})\hat{q}(\hat{\sigma}),
	\end{equation}
	where $\dot{\mathfrak{Z}}\equiv \dot{\mathfrak{Z}}_{\dot{q}},\hat{\mathfrak{Z}}\equiv\hat{\mathfrak{Z}}_{\dot{q}}$ and $\bar{\mathfrak{Z}}\equiv \bar{\mathfrak{Z}}_{\dot{q}}$ are normalizing constants.
	\begin{defn}[Definition 5.6 of \cite{ssz16}]\label{def:opt:coloring:profile}
	    The \textit{optimal coloring profile} for the truncated model and the untruncated model is the tuple $H^\star_{\lambda,L}=(\dot{H}^\star_{\lambda,L},\hat{H}^\star_{\lambda,L},\bar{H}^\star_{\lambda,L})$ and $H^\star_{\lambda}=(\dot{H}^\star_{\lambda},\hat{H}^\star_{\lambda},\bar{H}^\star_{\lambda})$, defined respectively by $ H^\star_{\lambda,L}:= H_{\dot{q}^\star_{\lambda,L}}$ and $H^\star_{\lambda}:=H_{\dot{q}^\star_{\lambda}}$.
	\end{defn}
	\begin{defn}\label{def:la:s:c:star}
	    For $k\geq k_0,\alpha \in (\alpha_{\textsf{cond}}, \alpha_{\textsf{sat}})$ and $\la\in [0,1]$, define the optimal $\la$-tilted truncated weight $s^\star_{\la,L}\equiv s^\star_{\la,L}(\alpha,k)$ and untruncated weight $s^\star_{\la} \equiv s^\star_{\la}(\alpha,k)$ by
	    \begin{equation}\label{eq:def:s:star:lambda}
	    \begin{split}
	       s^\star_{\la,L}&:=s(H^\star_{\la,L})\equiv \big\langle \log \dot{\Phi}, \dot{H}^\star_{\la,L}\big\rangle+\big\langle \log \hat{\Phi}^{\textnormal{m}}, \hat{H}^\star_{\la,L}\big\rangle+\big\langle \log \bar{\Phi}, \bar{H}^\star_{\la,L}\big\rangle;\\
	       s^\star_{\la}&:=s(H^\star_{\la})\equiv \big\langle \log \dot{\Phi}, \dot{H}^\star_{\la}\big\rangle+\big\langle \log \hat{\Phi}^{\textnormal{m}}, \hat{H}^\star_{\la}\big\rangle+\big\langle \log \bar{\Phi}, \bar{H}^\star_{\la}\big\rangle.
	    \end{split}
	    \end{equation}
	    Then, define the optimal tilting constants $\la^\star_L\equiv\la^\star_L(\alpha,k)$ and $\la^\star \equiv \la^\star(\alpha, k)$ by
	    \begin{equation}\label{eq:def:optimal:lambda}
	       \la^\star_L:=\sup\{\lambda\in [0,1]: F_{\la,L}(H^\star_{\lambda,L}) \geq \lambda s^\star_{\la,L} \}\quad\textnormal{and}\quad\lambda^\star :=
	      \sup\{\lambda\in [0,1]: F_{\la}(H^\star_\lambda) \geq \lambda s^\star_{\la} \}.
	    \end{equation}
	    Finally, we define $s^\star_L\equiv s^\star_L(\alpha,k),s^\star\equiv s^\star(\alpha,k)$ and $c^\star\equiv c^\star(\alpha,k)$ by
	    	\begin{equation}\label{eq:def:sstarandlambdastar}
	       s^\star_{L}:= s^\star_{\lambda^\star_L,L},\quad s^\star\equiv s^\star_{\la^\star},\quad\textnormal{and}\quad c^\star\equiv (2\la^\star)^{-1}.
	     \end{equation}
	 We remark that $s^\star = \textsf{f}^{1\textsf{rsb}}(\alpha)$  and $\la^\star \in (0,1)$ holds for $\alpha \in (\alpha_{\textsf{cond}}, \alpha_{\textsf{sat}})$ (see Proposition 1.4 of \cite{ssz16}).
	\end{defn}
	To end this section, we define the optimal coloring profile in the second moment (cf. Definition 5.6 of \cite{ssz16}). Define the analogue of $(\dot{\Phi},\hat{\Phi},\bar{\Phi})$ in the second moment $(\dot{\Phi}_2,\hat{\Phi}_2,\bar{\Phi}_2)$ by $\dot{\Phi}_2:=\dot{\Phi}\otimes \dot{\Phi}$, $\bar{\Phi}_2:=\bar{\Phi}\otimes \bar{\Phi}$ and 
	\begin{equation*}
	    \hat{\Phi}_2(\bsig)^{\la}:=\E^{\lit}\Big[\hat{\Phi}^{\lit}(\sig^1\oplus \uL^1)^{\la}\hat{\Phi}^{\lit}(\sig^2\oplus \uL^2)^{\la} \Big]\quad\textnormal{for}\quad \bsig=(\sig^1,\sig^2)\in \Omega^{2k}.
	\end{equation*}
	Then, the BP map in the second moment $\prescript{}{2}{\textnormal{BP}}_{\la,L}:\PPP\big((\dot{\Omega}_L)^2\big)\to\PPP\big((\dot{\Omega}_L)^2\big)$ is defined by replacing $(\dot{\Phi},\hat{\Phi},\bar{\Phi})$ in \eqref{eq:def:BP} by $(\dot{\Phi}_2,\hat{\Phi}_2,\bar{\Phi}_2)$. Moreover, analogous to \eqref{eq:H:q:1stmo}, define $\prescript{}{2}{H}_{\dot{q}}\equiv \big(\prescript{}{2}{\dot{H}}_{\dot{q}},\prescript{}{2}{\hat{H}}_{\dot{q}},\prescript{}{2}{\bar{H}}_{\dot{q}}\big)$ for $\dot{q}\in \PPP\big((\dot{\Omega}_L)^2\big)$ by replacing $(\dot{\Phi},\hat{\Phi},\bar{\Phi})$ in \eqref{eq:H:q:1stmo} by $(\dot{\Phi}_2,\hat{\Phi}_2,\bar{\Phi}_2)$. Here, $\prescript{}{2}{\dot{H}}_{\dot{q}}\in \PPP(\textnormal{supp}\,\dot{\Phi}_2), \prescript{}{2}{\hat{H}}_{\dot{q}}\in \PPP(\textnormal{supp}\,\hat{\Phi}_2)$ and $\prescript{}{2}{\bar{H}}_{\dot{q}}\in \PPP(\Omega^2)$.
		\begin{defn}[Definition 5.6 of \cite{ssz16}]\label{def:opt:coloring:profile:2ndmo}
	    The \textit{optimal coloring profile} in the second moment for the truncated model is the tuple $H^{\bullet}_{\lambda,L}=(\dot{H}^{\bullet}_{\lambda,L},\hat{H}^{\bullet}_{\lambda,L},\bar{H}^{\bullet}_{\lambda,L})$ defined by $ H^{\bullet}_{\lambda,L}:= \prescript{}{2}{H}_{\dot{q}^\star_{\lambda,L}\otimes\dot{q}^\star_{\lambda,L}}$.
	\end{defn}
	\section{Proof outline}
	Recall that $\bN_s^{\tr}\equiv \bZ_{0,s}^{\tr}$ counts the number of valid colorings with weight between $e^{ns}$ and $e^{ns+1}$, which do not contain a free cycle. Also, recalling the constant $s_\circ(C)\equiv s_\circ(n,\alpha,C)$ in \eqref{eq:def:s:circle}. It was shown in \cite{nss20a} that for fixed $C\in \R$, $\E\bN^{\tr}_{s_{\circ}(C)}\asymp_{k} e^{\la^\star C}$ holds and we have the following:
	\begin{equation}\label{eq:2ndmo of Ns}
	\E (\bN_{s_\circ(C)}^{\tr})^2 \lesssim_{k} (\E\bN_{s_\circ(C)}^{\tr})^{2} +\E\bN_{s_\circ(C)}^{\tr}.
	\end{equation}
	Hence, the Cauchy-Schwarz inequality shows that there is a constant $C_k<1$, which only depends on $\alpha$ and $k$, such that for $C>0$,
	\begin{equation*}
	\P \left(\bN_{s_\circ(C)}^{\tr} >0 \right)>C_k.
	\end{equation*}
	The remaining work is to push this probability close to $1$. The key to proving Theorem \ref{thm1} and \ref{thm2} is the following theorem.
	
	\begin{thm}\label{thm:lowerbd}
		Let $k\geq k_0$, $\alpha \in (\alpha_{\textsf{cond}}, \alpha_{\textsf{sat}})$, and set $\lambda^\star, s^\star$ as in Definition \ref{def:la:s:c:star}. For every $\eps>0$, there exist constants $C(\eps,\alpha,k)>0$ and $\delta\equiv \delta(\eps,\alpha,k)>0$ such that we have for $n\geq n_0(\eps,\alpha,k)$ and $C\geq C(\eps,\alpha,k)$,
		\begin{equation*}
		\P\bigg(\bN^{\tr}_{s_\circ(C)}\geq \delta \E\bN^{\tr}_{s_\circ(C)}\bigg)\geq 1-\eps,
		\end{equation*}
		where $s_\circ(C)\equiv s_\circ(n,\alpha,C)\equiv s^\star -\frac{  \log n}{2\lambda^\star n} - \frac{C}{n}$.
	\end{thm}
	Theorem \ref{thm:lowerbd} easily implies Theorem \ref{thm1} and \ref{thm2}: in \cite[Remark 6.11]{nss20a}, we have already shown that Theorem \ref{thm:lowerbd} implies Theorem \ref{thm2}, so we are left to prove Theorem \ref{thm1}.
	
	\begin{proof}[Proof of Theorem \ref{thm1}]
	By Theorem 3.22 of \cite{nss20a}, $\E \bN^{\tr}_{s_{\circ}(C)}\asymp e^{\la^\star C}$, so Theorem \ref{thm:lowerbd} implies Theorem \ref{thm1}-(b). Hence, it remains to prove Theorem \ref{thm1}-(a). Fix $\eps>0$ throughout the proof. By Theorem \ref{thm:lowerbd}, there exists $C_1\equiv C_1(\eps,\alpha,k)$ such that
		\begin{equation}\label{eq:thm1bpf:1}
		\P(\bN^{\tr}_{s_\circ(C_1)}\geq 1) \geq 1-\frac{\eps}{4}.
		\end{equation}
		Note that on the event $\bN_{s_\circ(C_1)}^{\tr}>0$, we have
		\begin{equation*}
		Z_n \ge \bZ^\tr_{1} \geq e^{ns_\circ(C_1)} = e^{-C_1} n^{-\frac{1}{2\lambda^\star}} e^{ns^\star},
		\end{equation*}
		where $Z_n$ denotes the number of \textsc{nae-sat} solutions in $\GGG$. Moreover, it was shown in \cite[Theorem 1.1-(a)]{nss20a} that for $C_2\leq n^{1/5}$, we have
		\begin{equation*}
		    \sum_{s\leq s_{\circ}(C_2)}\E \bZ_{1,s}\leq n^{-\frac{1}{2\la^\star}}\exp\big(ns^\star-(1-\la^\star)C_2+C_k\big),
		\end{equation*}
		where $C_k$ is a constant that depends only on $k$ and the sum in the \textsc{lhs} is for $s\in n^{-1}\Z$.
		Therefore, by Markov's inequality, we can choose $C_2\equiv C_2(\eps,\alpha,k)$ to be large enough so that
		\begin{equation}\label{eq:thm1bpf:3}
		\P \bigg( \sum_{s\leq s_{\circ}(C_2)}\bZ_{1,s}\geq \eps e^{-C_1} n^{-\frac{1}{2\lambda^\star}} e^{ns^\star}  \bigg) \leq \frac{\eps}{4}.
		\end{equation}
		Furthermore, by Theorem 1.1-(a) of \cite{nss20a}, there exists $C_3\equiv C_3(\eps,\alpha,k)$ such that
		\begin{equation}\label{eq:thm1bpf:4}
		\P\bigg( \sum_{s\geq s_\circ(C_3)} \bN_{s} \geq 1 \bigg) \leq \frac{\eps}{4}.
		\end{equation}
		Finally, Theorem 3.24 and Proposition 3.25 of \cite{nss20a} show that for $|C|\leq n^{1/4}$, $\E \bN_{s_{\circ}(C)}\asymp_{k}e^{-\la^\star C}$ holds. Thus, we can choose $K\equiv K(\eps,\alpha,k)\in \N$ large enough so that
		\begin{equation}\label{eq:thm1bpf:5}
		\P \bigg( \sum_{s\in n^{-1}\Z:~s_\circ(C_2)\leq s\leq s_\circ(C_3)} \bN_s \geq K \bigg) \leq \frac{\eps}{4}.
		\end{equation}
		Therefore, by \eqref{eq:thm1bpf:1}--\eqref{eq:thm1bpf:5}, the conclusion of Theorem \ref{thm1}-(a) holds with $K=K(\eps,\alpha,k)$.		
	\end{proof}
	\subsection{Outline of the proof of Theorem \ref{thm:lowerbd}}
	In this subsection, we discuss the outline of the proof of Theorem \ref{thm:lowerbd}. We begin with  a natural way of characterizing cycles in $\GGG = (\GG, \uL)$ which was also used in \cite{cw18}.
	
	\begin{defn}[$\zeta$-cycle]\label{def:zetcycle}
		Let $l>0$ be an integer and for each ${\zeta}\in\{0,1\}^{2l}$, a $\zeta$-\textit{cycle} in $\GGG=(\GG,\underline{\texttt{L}})$ consists of 
		\begin{equation*}
		\mathcal{Y}(\zeta) =\{v_i,a_i,(e_{v_i}^j, e_{a_i}^j)_{j=0,1} \}_{i=1}^l
		\end{equation*}
		which satisfies the following conditions:
		\begin{itemize}
			\item $v_1,\ldots,v_l\in [n]\equiv V$ are distinct variables, and for each $i\in[l]$, $e_{v_i}^0$ and $e_{v_i}^1$ are distinct half-edges attached to $v_i$.
			
			\item $a_1,\ldots,a_l\in [m]\equiv F$ are distinct clauses, and for each $i\in[l]$, $e_{a_i}^0$ and $e_{a_i}^1\in[k]$ are distinct half-edges attached to $a_i$. Moreover, 
			\begin{equation}\label{eq:def:cyclelabel}
			a_1 = \min\{a_i:i\in [l] \}, \quad \textnormal{and} \quad 
			e_{a_1}^0<e_{a_1}^1.
			\end{equation}
			
			\item $(e_{v_i}^1,e_{a_{i+1}}^0)$ and $(e_{a_i}^1,e_{v_i}^0)$ are edges in $\GG$ for each $i\in[l]$. ($a_{l+1}=a_1$)
			
			\item The literal on the half-edge $\texttt{L}({e_{a_i}^j})$ is given by $\texttt{L}({e_{a_i}^j}) = \zeta_{2(i-1)+j }$ for each $i\in[l]$ and $j\in\{0,1\}$. ($\zeta_0=\zeta_{2l}$)
		\end{itemize}
		Note that \eqref{eq:def:cyclelabel} is introduced in order to prevent overcounting. Also, we denote the \textit{size of} $\zeta$ by $||\zeta||$, defined as
			\begin{equation}\label{eq:def:sizezeta}
			||\zeta||=l.
			\end{equation}
	\end{defn}
	
	Furthermore, we denote by $X({\zeta})$ the number of $\zeta$-cycles in $\GGG=(\GG,\uL)$. For $\zeta\in\{0,1\}^{2l}$, it is not difficult to see that
	\begin{equation}\label{eq:cycleasympt}
	X({\zeta})\; {\overset{\textnormal{d}}{\longrightarrow}}\; \textnormal{Poisson}(\mu({\zeta})), \quad \textnormal{where  } \;\;  \mu({\zeta}) := \frac{1}{2l}2^{-2l}(k-1)^l(d-1)^l.
	\end{equation}
	Moreover,   $\{X({\zeta})\}$ are asymptotically jointly independent in the sense that for any $l_0>0$,
	\begin{equation}\label{eq:cyclejointasymp}
	\lim_{n\to\infty}	\P\bigg(\bigcap_{\zet:\, ||\zeta|| \leq l_0} \left\{X({\zeta}) = x_{\zeta} \right\} \bigg) 
	=
	\prod_{\zeta:\,||\zeta||\leq  l_0} \P \left(\textnormal{Poisson}(\mu({\zeta})) = x_{\zeta}\right).
	\end{equation}
	Both \eqref{eq:cycleasympt} and \eqref{eq:cyclejointasymp} follow from an application of the method of moments (e.g., see Theorem 9.5 in \cite{jlrrg}). Given these definitions and properties, we are ready to state the small subgraph conditioning method, appropriately adjusted to our setting.

	\begin{thm}[Small subgraph conditioning \cite{rw92,rw94}]\label{thm:smallsubcon}
		Let $\GGG = (\GG, \uL)$ be a random $d$-regular $k$-\textsc{nae-sat} instance and let $X({\zeta})\equiv X({\zeta,n})$ be the number of $\zeta$-cycles in $\GGG$ with $\mu({\zeta})$ given as \eqref{eq:cycleasympt}. Suppose that a random variable $Z_n\equiv Z_n(\GGG)$ satisfies the following conditions:
		\begin{enumerate}
			\item [\textnormal{(a)}] For each $l\in \N$ and $\zeta\in\{0,1\}^{2l}$, the following limit exists:
			\begin{equation}\label{eq:thm:ssg:cyceff}
			1+ \delta({\zeta}) \equiv \lim_{n \to\infty} \frac{\E[ Z_nX({\zeta})]}{\mu({\zeta})\E Z_n }.
			\end{equation}
			Moreover, for each $a,l\in \N$ and $\zeta \in \{0,1 \}^{2l}$, we have
			\begin{equation*}
			\lim_{n\to\infty} \frac{\E [Z_n (X(\zeta))_a] }{ \E Z_n} = (1+\delta(\zeta))^a \mu(\zeta)^a,
			\end{equation*}
			where $(b)_a$ denotes the falling factorial $(b)_a := b(b-1)\cdots (b-a+1)$.

			\item [\textnormal{(b)}] The following limit exists:
			\begin{equation*}
			C\equiv \lim_{n \to\infty} \frac{\E Z_n^2}{(\E Z_n)^2}.
			\end{equation*}

			\item [\textnormal{(c)}]
			We have
			$\sum_{l=1}^\infty \sum_{\zeta\in\{0,1\}^{2l}} \mu({\zeta}) \delta({\zeta})^2 <\infty.$

			\item [\textnormal{(d)}] Moreover, the constant $C$ satisfies $C\leq \exp\left(\sum_{l=1}^\infty \sum_{\zeta\in\{0,1\}^{2l}} \mu({\zeta}) \delta({\zeta})^2 \right)$.

		\end{enumerate}
		Then, we have the following conclusion:
		\begin{equation}\label{eq:ssg:distlim}
		\frac{Z_n}{\E Z_n} \overset{\textnormal{d}}{\longrightarrow}
		W\equiv \prod_{l=1}^\infty \prod_{\zeta\in\{0,1\}^{2l}} \big(1+\delta({\zeta})\big)^{\bar{X}({\zeta})} \exp\big(-\mu({\zeta}) \delta(\zeta)\big), 
		\end{equation}
		where $\bar{X}({\zeta})$ are independent Poisson random variables with mean $\mu({\zeta})$.
	\end{thm}
	
	We briefly explain a way to understand the crux of the theorem as follows. Since $\{X({\zeta})\}$ jointly converges to $\{\bar{X}({\zeta})\}$, it is not hard to see that
	\begin{equation*}
	\frac{\E \left[\E\big[Z_n\,\big|\,\{X({\zeta})\}\big]^2 \right]}{(\E Z_n)^2} \to \exp\bigg(\sum_\zeta \mu({\zeta}) \delta({\zeta})^2\bigg),
	\end{equation*}
	using the conditions (a),(b),(c),(d) (e.g. see Theorem 9.12  in \cite{jlrrg} and its proof).
	Therefore, conditions (b) and (d) imply that the conditional variance of $Z_n$ given $\{X({\zeta})\}$ is negligible compared to $(\E Z_n)^2$, and hence the distribution of $Z_n$ is asymptotically the same as that of $\E \big[Z_n\big| \{X({\zeta})\}\big]$ as addressed in the conclusion of the theorem.
	
	Having Theorem \ref{thm:smallsubcon} in mind, our goal is to (approximately) establish the four assumptions for (a truncated version of) $\bZ_{\lambda,s_\circ(C)}^{\tr},$ for $s_{\circ}(C)\equiv s^\star-\frac{\log n}{2\la^\star n}-\frac{C}{n}$. The condition (b) has already been obtained from the moment analysis from \cite{nss20a}. The condition (a) will be derived in Proposition \ref{prop:cycleeffect:moments} below and (c) will be derived in Lemma \ref{lem:deltabound} below. The condition (d), however, will be true only in an approximate sense, where the approximation error becomes smaller when we take the constant $C$ larger because of within-cluster correlations.
	
	 In the previous works \cite{rw92,rw94,s10,gsv15,gsv16}, the condition (d) could be obtained through a direct calculation of the second moment in a purely combinatorial way. However, this approach seems to be intractable in our model; for instance, the main contributing empirical measure to the first moment $H^\star_{\la}$ barely has combinatorial meaning. 
	 
	 Instead, we first establish \eqref{eq:ssg:distlim} for the $L$-\textit{truncated model}, by showing the concentration of the \textit{rescaled partition function}, introduced in \eqref{eq:def:rescaledPF} below. The truncated model will be easier to work with since it has a finite spin space unlike the untruncated model. Then, we rely on the convergence results regarding the leading constants of first and second moments, derived in \cite{nss20a}, to deduce (d) for the untruncated model in an approximate sense. We then apply ideas behind the proof of Theorem \ref{thm:smallsubcon} to deduce Theorem \ref{thm:lowerbd} (for details, see Section \ref{subsec:whp:smallsubcon}).
	
	 We now give a more precise description on how we establish (d) for the truncated model. Let $1\leq L,l_0 <\infty$ and $\lambda \in (0,\lambda^\star_L)$, where $\la^\star_L$ is defined in Definition \ref{def:la:s:c:star}. Then, define the \textit{rescaled partition function} $\bY^{(L)}_{\lambda, l_0}$ 
	\begin{equation}\label{eq:def:rescaledPF}
	\begin{split}
\bY^{(L)}_{\lambda, l_0}&\equiv \bY^{(L)}_{\lambda, l_0}(\GGG)
	:=
	\bwZ_{\lambda}^{(L),\tr}  \prod_{||\zeta|| \le l_0 } \left(1+\delta_{L}(\zeta) \right)^{-X(\zeta)}\quad\textnormal{where}\\
	\bwZ_{\lambda}^{(L),\tr}&\equiv \bwZ_{\lambda}^{(L),\tr}(\GGG):=\sum_{||H-H^\star_{\la,L}||_1 \leq n^{-1/2}\log^{2}n}\bZ^{(L),\tr}_{\la}[H]\,.
	\end{split}
	\end{equation}
	Here, $\delta_{L}(\zeta)$ is the constant defined in \eqref{eq:thm:ssg:cyceff} for $Z_n = \bZ_{\lambda}^{(L),\tr}$, assuming its existence. The precise definition of $\delta_{L}(\zeta)$ is given in \eqref{eq:def:delta by trace}. The reason to consider $\bwZ^{(L),\tr}_{\la}$ instead of $\bZ^{(L),\tr}_{\la}$ is to ignore the contribution from near-identical copies in the second moment. Then, Proposition \ref{prop:concenofrescaled} below shows that the rescaled partition function is concentrated for each $L<\infty$. Its proof is provided in Section \ref{subsec:whp:rescaled}.
	
	\begin{prop}\label{prop:concenofrescaled}
		Let $k\geq k_0$, $L<\infty$, and $\lambda \in (\la^\star_L-0.01\cdot 2^{-k},\la^\star_L)$. Then we have
		\begin{equation*}
		\lim_{l_0\to\infty}
		\lim_{n \to\infty} \frac{ \E \big(\bY^{(L)}_{\lambda, l_0}\big)^{2}}{ \big(\E\bY^{(L)}_{\la,l_0} \big)^2} =1.	
		\end{equation*}
	\end{prop}

	\begin{remark}\label{remark:intrinsic:correlation}
		An important thing to note here is that Proposition \ref{prop:concenofrescaled} is not true for $\lambda=\lambda^\star_L$. If $\lambda<\lambda^\star_L$, then $s^\star_{\la,L}<s^\star_L$, so there should exist exponentially many clusters of size $e^{n s^\star_{\la,L}}$. Therefore, the intrinsic correlations within clusters are negligible (that is, when we pick two clusters at random, the probability of selecting the same one is close to $0$) and the fluctuation is taken over by cycle effects. However, when there are bounded number of clusters of size $e^{ns^\star_{\la,L}}$ (that is, when $\la$ is very  close to $\la^\star_L$), within-cluster correlations become non-trivial. Mathematically, we can see this from \eqref{eq:2ndmo of Ns}, where we can ignore the first moment term in the \textsc{rhs} of (\ref{eq:2ndmo of Ns}) if (and only if) it is large enough.
		
		Nevertheless, for $s_\circ(C)$ defined as in Theorem \ref{thm:lowerbd}, we will see in Section  \ref{subsec:whp:smallsubcon} that if we set $C$ to be large, then (d) holds, and hence the conclusion of Theorem \ref{thm:smallsubcon} holds with a small error.
		\end{remark}

	
	
		\vspace{3mm}
		\noindent \textbf{Further notations.} Throughout this paper, we will often use the following multi-index notation. Let $\underline{a}=(a_\zeta)_{||\zeta||\leq l_0}$, $\underline{b}=(b_\zeta)_{||\zeta||\leq l_0}$ be tuples of integers indexed by $\zeta$ with $||\zeta||\leq l_0$. Then, we write
		\begin{equation*}
		(\underline{a})^{\underline{b}} = \prod_{\zeta:||\zeta||\leq l_0} a_\zeta^{b_\zeta}; \quad\quad\quad
		(\underline{a})_{\underline{b}} = \prod_{\zeta:||\zeta||\leq l_0 } (a_\zeta)_{b_\zeta}= \prod_{\zeta:||\zeta||\leq l_0} \prod_{i=0}^{b_\zeta -1} (a_\zeta-i). 
		\end{equation*}
       Moreover, for non-negative quantities $f=f_{d,k, L, n}$ and $g=g_{d,k,L, n}$, we use any of the equivalent notations $f=O_{k,L}(g), g= \Omega_{k,L}(f), f\lesssim_{k,L} g$ and $g \gtrsim_{k,L} f $ to indicate that $f\leq C_{k,L}\cdot g$ holds for some constant $C_{k,L}>0$, which only depends on $k,L$.
	\section{The effects of cycles}\label{sec:whp}

	
	In this section, our goal is to obtain the condition (a) of Theorem \ref{thm:smallsubcon} for (truncated versions of) $\bZ^{(L),\tr}_{\la}$ and $\bZ^{\tr}_{\la,s_n}$, where $|s_n-s^\star_{\la}|=O(n^{-2/3})$ (see Proposition \ref{prop:cycleeffect:moments} below). To do so, we first introduce necessary notations to define $\delta(\zeta)$ appearing in Theorem \ref{thm:smallsubcon}.

	 For $\la\in [0,1]$, recall the optimal coloring profile of the untruncated model $H^\star_{\la}$ and truncated model $H^\star_{\la,L}$ from Definition \ref{def:opt:coloring:profile}. We  denote the two-point marginals of $\dot{H}^\star_{\la}$  by
	\begin{equation*}
	\dot{H}^\star_{\la} (\tau_1,\tau_2) := 
	\sum_{\underline{\sigma}\in \Omega^d} \dot{H}^\star_{\la} (\underline{\sigma})\one{\{\sigma_1=\tau_1, \sigma_2=\tau_2 \}},\quad (\tau_1,\tau_2)\in \Omega^{2}
	\end{equation*}
	and similarly for  $\dot{H}^\star_{\la,L}$. On the other hand, for ${\uL}\in\{0,1\}^k$, consider the optimal clause empirical measure $\hat{H}^{\uL}_{\la}$ given the literal assignment $\uL\in \{0,1\}^{k}$ around a clause, namely for $\sig \in \Omega^k$,
	\begin{equation}\label{eq:def:hatHlit}
	\hat{H}^{\uL}_{\la} (\sig) := \frac{1}{\hat{\mathfrak{Z}}^{\uL}_{\la}} \hat{\Phi}^{\textnormal{lit}}(\sig\oplus{\uL})^\lambda \prod_{i=1}^k \dot{q}^\star_{\lambda} (\dot{\sigma}_i),
	\end{equation}
	where $\hat{\mathfrak{Z}}^{\uL}_{\la}$ is the normalizing constant. Note that $\hat{\mathfrak{Z}}^{\uL}_{\la} = \hat{\mathfrak{Z}}_{\la}$ is independent of $\uL$ due to the symmetry $\dot{q}_\lambda^\star(\dot{\sigma})=\dot{q}_\lambda^\star(\dot{\sigma}\oplus 1)$. Similarly, define $\hat{H}^{\uL}_{\la,L}$ for the truncated model. Given the literals $\texttt{L}_1,\texttt{L}_2$ at the first two coordinates of a clause, the two point marginal of $\hat{H}^{\uL}_{\la}$ is defined by
	\begin{equation}\label{eq:def:hatHtwopt}
	\begin{split}
	\hat{H}^{\texttt{L}_1,\texttt{L}_2}_{\la} (\tau_1,\tau_2 )
	&:=\frac{1}{2^{k-2}} \sum_{\texttt{L}_3,\ldots \texttt{L}_k\in \{0,1\}} \sum_{\sig\in \Omega^k}
	\hat{H}^{\uL}_{\la} (\sig) \one\{\sigma_1=\tau_1,\sigma_2=\tau_2 \}\\
	&= \sum_{\sig\in \Omega^k}
	\hat{H}^{\uL}_{\la} (\sig) \one\{\sigma_1=\tau_1,\sigma_2=\tau_2 \},
		\end{split}
		\end{equation}
	where the second equality holds for any $\uL\in\{0,1\}^k$ that agrees with $\texttt{L}_1,\texttt{L}_2$ at the first two coordinates, due to the symmetry $\hat{H}^{\uL  }_{\la}(\underline{\tau})=\hat{H}_{\la}^{\uL \oplus \uL'}(\underline{\tau}\oplus \uL')$. The symmetry also implies that
	\begin{equation*}
	\sum_{\tau_2} \hat{H}^{\texttt{L}_1,\texttt{L}_2}_{\la}(\tau_1,\tau_2) = \bar{H}^\star_{\la} (\tau_1),
	\end{equation*}
	for any $\texttt{L}_1, \texttt{L}_2 \in\{0,1\}$ and $\tau_1\in \Omega$. We also define $\hat{H}^{\texttt{L}_1,\texttt{L}_2}_{\la,L}$ analogously for the truncated model.
	
	Then, we define $\dot{A}\equiv \dot{A}_{\la},\hat{A}^{\texttt{L}_1,\texttt{L}_2}\equiv \hat{A}^{\texttt{L}_1,\texttt{L}_2}_{\la}$ to be the $\Omega\times \Omega$ matrices as follows:
	\begin{equation}\label{eq:def:dotAhatA}
	\dot{A}(\tau_1,\tau_2) = \frac{\dot{H}^\star_{\la} (\tau_1,\tau_2)}{\bar{H}^\star_{\la}(\tau_1)}, \quad
	\hat{A}^{\texttt{L}_1,\texttt{L}_2}(\tau_1,\tau_2) = \frac{\hat{H}_{\la}^{\texttt{L}_1,\texttt{L}_2}(\tau_1,\tau_2)}{\bar{H}^\star_{\la}(\tau_1)},
	\end{equation}
	and $\Omega_L\times\Omega_L$ matrices $\dot{A}_{\la,L}$ and $\hat{A}_{\la,L}^{\texttt{L}_1,\texttt{L}_2}$ are defined analogously using $\dot{H}^\star_{\la,L}, \hat{H}^{\texttt{L}_1,\texttt{L}_2}_{\la,L}$ and $\bar{H}^\star_{\la,L}$. Note that both matrices have row sums equal to $1$, and hence  their largest eigenvalue is $1$. For $\zeta\in\{0,1\}^{2l}$, we introduce the following notation for convenience:
	\begin{equation}\label{eq:def:AdotAhat zet}
	(\dot{A} \hat{A} )^{\zeta} \equiv
	\prod_{i=0}^{l-1} \left(\dot{A} \hat{A}^{\zeta_{2i},\zeta_{2i+1}} \right),
	\end{equation}
	where $\zeta_0 = \zeta_{2l}$. Moreover, we define $(\dot{A}_L \hat{A}_L)^\zeta$ analogously. Then, the main proposition of this section is given as follows.
		
		\begin{prop}\label{prop:cycleeffect:moments}
			Let  $L,l_0>0$  and let $\underline{X} = \{X({\zeta})\}_{||\zeta||\leq l_0}$ denote the number of $\zeta$-cycles in $\GGG$. Also, set $\mu({\zeta})$ as \eqref{eq:cycleasympt}, and for each $\zeta\in \cup_l\{0,1\}^{2l}$, define
			\begin{equation}\label{eq:def:delta by trace}
				\begin{split}
				\delta(\zeta) \equiv \delta({\zeta};\lambda) &:= Tr\left[(\dot{A} \hat{A})^\zeta \right]-1, \\ 
				\delta_L(\zeta) \equiv \delta_{L} (\zeta;\lambda) &:= Tr \left[ (\dot{A}_L \hat{A}_L)^{\zeta}\right]-1.
				\end{split}
				\end{equation}
			Then, there exists a constant $c_{\textsf{cyc}}=c_{\textsf{cyc}}(l_0)$ such that the following statements hold true:
			\begin{enumerate}
				\item For $\lambda \in (0,1)$ and any tuple of nonnegative integers $\underline{a}=(a_\zeta)_{||\zeta||\leq l_0}$, such that $||\underline{a}||_\infty \leq c_{\textsf{cyc}} \log n$, we have
				\begin{equation}\label{eq:cycmo:trun:1st}
						\E \left[ \bwZ_{\lambda}^{(L),\tr} \cdot (\underline{X})_{\underline{a}} \right] = \left(1+ err(n,\underline{a}) \right) \left(\underline{\mu} ( 1+ \underline{\delta}_L)\right)^{\underline{a}} \E \bwZ_{\lambda}^{(L),\tr},
				\end{equation}
				where $\bwZ^{(L),\tr}_{\la}$ is defined in \eqref{eq:def:rescaledPF} and $err(n,\underline{a}) = O_k \left(||\underline{a}||_1 n^{-1/2} \log^2 n \right)$.
				
	
		
			\item For $\lambda \in (0,\lambda_L^\star)$, where $\la^\star_{L}$ is defined in Definition \ref{def:la:s:c:star}, the analogue of \eqref{eq:cycmo:trun:1st} holds for the second moment as well. That is, for $\underline{a}=(a_\zeta)_{||\zeta||\leq l_0}$ with $||\underline{a}||_\infty \leq c_{\textsf{cyc}} \log n$, we have
			\begin{equation}\label{eq:cycmo:trun:2nd}
			\E \left[ \big(\bwZ_{\lambda}^{(L),\tr}\big)^2 \cdot (\underline{X})_{\underline{a}} \right] = \left(1+ err(n,\underline{a}) \right) \left(\underline{\mu} ( 1+ \underline{\delta}_L)^2\right)^{\underline{a}} \E \big(\bwZ_{\lambda}^{(L),\tr}\big)^2.
			\end{equation}
			
			\item  Under a slightly weaker error given by $err'(n,\underline{a}) = O_k(||\underline{a}||_1 n^{-1/8})$, the analogue of (1) holds the same for the untruncated model with any $\lambda \in (0,1)$. Namely, analogously to \eqref{eq:def:rescaledPF}, define $\bwZ_{\la}^{\tr}$ and $\bwZ^{\tr}_{\la,s}$ by
			\begin{equation}\label{eq:def:tilde:Z:untruncated}
            \begin{split}
			 &\bwZ^{\tr}_{\la}:=\sum_{||H-H^\star_{\la}||_1\leq n^{-1/2}\log^2 n}\bZ^{\tr}_{\la}[H];\\
    &\bwZ^{\tr}_{\la,s}:=\sum_{||H-H^\star_{\la}||_1\leq n^{-1/2}\log^2 n}\bZ^{\tr}_{\la}[H]\one\big\{s(H)\in [ns,ns+1)\big\}\,.
             \end{split}
			\end{equation}
			Then, \eqref{eq:cycmo:trun:1st} continues to hold when we replace $\bwZ_{\lambda}^{(L),\tr},err(n,\underline{a})$ and $\underline{\delta}_L$ by $\bwZ_{\lambda}^{\tr}, err'(n,\underline{a})$ and $\underline{\delta}$ respectively. 
			Moreover, \eqref{eq:cycmo:trun:1st} continues to hold when we replace $\bwZ_{\lambda}^{(L),\tr},err(n,\underline{a})$ and $\underline{\delta}_L$ by $\bZ_{\lambda,s_n}^{\tr}, err'(n,\underline{a})$ and $\underline{\delta}$ respectively, where $|s_n-s^\star_{\la}|=O(n^{-2/3})$.
			\item For each $\zeta \in \cup_l \{0,1\}^{2l}$, we have
			$\lim_{L\to\infty} \delta_L (\zeta) = \delta(\zeta)$.
			
			\end{enumerate} 
		\end{prop}

	In the remainder of this section, we focus on proving (1) of Proposition \ref{prop:cycleeffect:moments}. In the proof, we will be able to see that (2) of the proposition follow by an analogous argument (see Remark \ref{rmk:Prop:2}). The proofs of (3) and (4) are deferred to Appendix  \ref{subsec:app:deltabd},  since they require substantial amounts of additional technical work.

	For each $\zeta \in\{0,1\}^{2l}$ and a nonnegative integer $a_\zeta$, let $\YY_i \equiv \YY_i(\zeta) \in \big\{\{v_{\iota},a_{\iota},(e_{v_{\iota}}^{j}, e_{a_{\iota}}^{j})_{j=0,1} \}_{\iota=1}^l\big\}$, $i\in[a_\zeta]$ denote the possible locations of $a_\zeta$ $\zeta$-cycles defined as Definition \ref{def:zetcycle}.  Then, it is not difficult to see that
	\begin{equation}\label{eq:Xbysumofindicators}
	(X(\zeta))_{a_\zeta} = \sum \one \{\YY_1,\ldots, \YY_{a_\zeta} \in \GG,\textnormal{ and } \uL(\YY_i;\GGG)=\zeta\,,\,\forall i\leq a_{\zeta}\} \equiv \sum \one \{\YY_1,\ldots, \YY_{a_\zeta}  \} ,
	\end{equation}
	where the summation runs over distinct $\YY_1,\ldots,\YY_{a_\zeta}$, and $\uL(\YY_i;\GGG)$ denotes the literals on $\YY_i$ inside $\GGG$. Based on this observation, we will show (1) of Proposition \ref{prop:cycleeffect:moments} by computing the cost of planting cycles at specific locations $\{\YY_i\}$. Moreover, in addition to $\{\YY_i\}$,  prescribing a particular coloring on those locations will be useful. In the following definition, we introduce the formal notations to carry out such an idea.
	
	\begin{defn}[Empirical profile on $\YY$]\label{def:cycprofile}
		Let $L,l_0>0$ be given integers and let $\underline{a}=(a_\zeta)_{||\zeta||\leq l_0}$. Moreover, let 
		\begin{equation*}
		\YY \equiv \{\YY_i(\zeta) \}_{i\in[a_\zeta], ||\zeta||\leq l_0}
		\end{equation*}
		denote the distinct $a_\zeta$ $\zeta$-cycles for each $||\zeta||\leq l_0$ inside $\GGG$ (Definition \ref{def:zetcycle}), and let $\sig$ be a valid coloring on $\GGG$. 
		We define   ${\Delta} \equiv {\Delta}[\sig;\YY]$, the \textit{empirical profile on} $\YY$, as follows. 
	 \begin{itemize}
	 	\item Let $V(\YY)$ (resp. $F(\YY)$) be the set of variables (resp. clauses) in $\cup_{||\zeta||\leq l_0} \cup_{i=1}^{a_\zeta} \YY_i(\zeta)$, and let $E_c(\YY)$ denote the collection of variable-adjacent half-edges included in $\cup_{||\zeta||\leq l_0} \cup_{i=1}^{a_\zeta} \YY_i(\zeta)$. We write $\sig_\YY$ to denote the restriction of $\sig$ onto $V(\YY)$ and $F(\YY)$.
	 	
	 	

	 	\item $\Delta\equiv \Delta[\sig;\YY] \equiv (\dot{\Delta}, (\hat{\Delta}^{\uL})_{\uL\in\{0,1\}^k} , \bar{\Delta}_c)$ is the counting measure of coloring configurations around $V(\YY), F(\YY)$ and $E_c(\YY)$ given as follows.
	 	\begin{equation}\label{eq:def:Deltaprofile}
	 	\begin{split}
	 	&\dot{\Delta} (\underline{\tau}) = 
	 	|\{v\in V(\YY): \underline{\sigma}_{\delta v} = \underline{\tau} \} |, \quad \textnormal{for all } \underline{\tau} \in \dot{\Omega}_L^d;\\
	 	&\hat{\Delta}^{\uL} (\underline{\tau}) = 
	 	|\{a\in F(\YY): \underline{\sigma}_{\delta a} = \underline{\tau}, \;\uL_{\delta a} = \uL \} |, \quad \textnormal{for all } \underline{\tau} \in \dot{\Omega}_L^k, \; \uL\in\{0,1\}^k;\\
	 	&\bar{\Delta}_c ({\tau}) = 
	 	|\{e\in E_c(\YY): \sigma_{e} = \tau \} |, \quad
	 	\textnormal{for all } \tau\in\dot{\Omega}_L.
	 	\end{split}
	 	\end{equation}
	 	
	 	\item We write $|\dot{\Delta}| \equiv \langle \dot{\Delta},1  \rangle$, and  define $|\hat{\Delta}^{\uL} |$, $|\bar{\Delta}_c|$  analogously.

	 \end{itemize}
	 Note that $\Delta$ is well-defined if $\YY$ and $\sig_\YY$ are given.
	\end{defn}
	
	In the proof of Proposition \ref{prop:cycleeffect:moments}, we will fix $\YY$, the locations of $\underline{a}$ $\zeta$-cycles, and a coloring configuration $\underline{\tau}_\YY$ on $\YY$, and compute the contributions from $\GGG$ and $\sig$ that has cycles on $\YY$ and satisfies $\sig_\YY = \underline{\tau}_\YY$. Formally, abbreviate $\bZ^\prime\equiv \bwZ^{(L),\tr}_{\la}$ for simplicity and define
	\begin{equation*}
	\bZ^\prime[\utau_\YY] = \sum_{\sig } w^\lit(\sig)^\lambda \mathds{1} \{\sig_{\YY} = \utau_{\YY} \}.
	\end{equation*}
	Then, we express that 
	\begin{equation}\label{eq:ZXbyindicators}
	\begin{split}
	\E \left[\bZ^\prime (\underline{X})_{\underline{a}}  \right]
	&= \sum_{\YY} \sum_{\underline{\tau}_\YY} 
	\E \left[\bZ^\prime[\underline{\tau}_\YY] \one \{\YY_i(\zeta) \in \GGG, \;\forall i\in[a_\zeta] , \, \forall ||\zeta||\leq l_0 \}  \right] \\
	&\equiv
	\sum_{\YY} \sum_{\underline{\tau}_\YY} \E \left[\bZ^\prime \one\{\YY, \underline{\tau}_\YY \} \right],
	\end{split}
	\end{equation}
	where the notation in the last equality is introduced for convenience. The key idea of the proof is to study the \textsc{rhs} of the above equation. We follow the similar idea developed in \cite{dss16}, Section 6, which is to decompose $\bZ^\prime$ in terms of empirical profiles of $\sig$ on $\GG$. The main contribution of our proof is to suggest a method that overcomes the complications caused by the indicator term (or the planted cycles).

	\begin{proof}[Proof of Proposition \ref{prop:cycleeffect:moments}-(1)]
	 As discussed above, our goal is to understand $\E [\bZ^\prime \one\{\YY,\underline{\tau}_\YY \} ]$ for given $\YY$ and $\underline{\tau}_\YY$.
	 To this end, we decompose the partition function in terms of \textit{coloring profiles}. It will be convenient to work with \begin{equation}\label{eq:def:g}
	 g\equiv g(H)\equiv (\dot{g},(\hat{g}^{\uL})_{\uL\in\{0,1\}^k},\bar{g}) \equiv \bigg(n\dot{H}, \Big(\frac{m}{2^k}\hat{H}^{\uL}\Big)_{\uL\in \{0,1 \}^k}, nd\bar{H}  \bigg),
	 \end{equation}
	 the non-normalized empirical counts of $H$. Moreover, if $g$ is given, then the product of the weight, clause, and edge factors is also determined. Let us denote this by $w(g)$, defined by
	 \begin{equation}\label{eq:def:wg}
	 w(g) \equiv w(\dot{g},(\hat{g}^{\uL})_{\uL}) \equiv  {\prod_{\underline{\tau}\in \dot{\Omega}_L^d } \dot{\Phi}(\underline{\tau})^{\dot{g}(\underline{\tau}) } \prod_{\uL\in \{0,1\}^k}\prod_{\underline{\tau}\in \dot{\Omega}_L^k} \hat{\Phi}^{\textnormal{lit}} (\underline{\tau}\oplus\uL )^{\hat{g}^{\uL}(\underline{\tau})}    \prod_{\tau \in \dot{\Omega}_L} \bar{\Phi}(\tau)^{\dot{M}\dot{g}(\tau) }}.
	 \end{equation}
 	
 	Recalling the definition of $\bZ^\prime \equiv \bwZ^{(L),\tr}_{\la}$ in \eqref{eq:def:rescaledPF}, we consider $g$ such that  $||g-g^\star_{\lambda,L}||_1\leq \sqrt{n}\log^2 n $, where we defined
 	\begin{equation}\label{eq:def:g:star}
 	    g^\star_{\la,L}:= g(H^\star_{\la,L})\quad\textnormal{and}\quad g^\star_{\la}:=g(H^\star_{\la}).
 	\end{equation}

 		Now,  fix the literal assignment $\uL_E$ on $\GGG$ which agrees with those on the cycles given by $\YY$.  Finally, let  $\Delta=(\dot{\Delta},\hat{\Delta}, \bar{\Delta}_c)$ denote the empirical profile on $\YY$ induced by $\underline{\tau}_\YY$. Then, we have
 	\begin{equation}\label{eq:cycleef:trun:Zg expansion}
 	\begin{split}
 	\E \left[\left.\bZ^\prime[g] \one\{\YY, \underline{\tau}_\YY\} \,\right|\, \uL_E \right]
 	&=
 	\frac{(\bar{g}-\bar{\Delta}_c)!}{(nd)!} {n- |\dot{\Delta}| \choose \dot{g}-\dot{\Delta}}
 	\prod_{\uL \in\{0,1\}^k}{| \hat{g}^{\uL}-\hat{\Delta}^{\uL} | \choose \hat{g}^{\uL}-\hat{\Delta}^{\uL}} \times w(g)^\lambda \\
 	&=
 	\frac{1}{(n)_{|\dot{\Delta}|} (m)_{|\hat{\Delta}|}} {n\choose \dot{g}}
 	\left\{\prod_{\uL}{| \hat{g}^{\uL} | \choose \hat{g}^{\uL}} \right\}{nd \choose \bar{g}}^{-1}
 	\frac{(\dot{g})_{\dot{\Delta}} \prod_{\uL }(\hat{g}^{\uL})_{\hat{\Delta}^{\uL}} }{(\bar{g})_{\bar{\Delta}_c}}\times w(g)^\lambda\\
 	&=
 	\frac{1+ O_k\left( ||\underline{a}||_1 n^{-1/2}\log^2n\right)}{(nd)^{|\bar{\Delta}_c|}}\E [\bZ'[g]\,|\,\uL_E] 
 	\frac{ (\dot{H}^\star)^{\dot{\Delta}} \prod_{\uL} 
 		(\hat{H}^{\uL})^{\hat{\Delta}^{\uL}}}{(\bar{H}^\star)^{\bar{\Delta}_c}}
 	,
 	\end{split}
 	\end{equation}  
 	where we wrote $H^\star = H^\star_{\lambda,L}$ and the last equality follows from $||g-g^\star_{\lambda,L}||\leq \sqrt{n}\log^2 n$.
 	
 	In the remaining, we sum the above over $\YY $ and $\underline{\tau}_\YY$, depending on the structure of $\YY$. To this end, we define $\eta\equiv\eta(\YY)$ to be
 	\begin{equation}\label{eq:def:eta}
 	\eta \equiv \eta(\YY) := |\bar{\Delta}_c|-|\dot{\Delta}| - |\hat{\Delta}|  ,
 	\end{equation}
 	where $|\hat{\Delta}| = \sum_{\uL} |\hat{\Delta}^{\uL}|$ and noting that $|\dot{\Delta}| , |\hat{\Delta}|$ and $|\bar{\Delta}_c|$ are well-defined if $\YY$ is given. Note that $\eta $  describes the number of disjoint components in $\YY$, in the sense that 
 	$$\#\{\textnormal{disjoint components of }\YY \}
 	= ||\underline{a}||_1 - \eta.$$  
 	
 	Firstly, suppose that all the cycles given by $\YY$ are disjoint, that is, $\eta(\YY)=0$. In other words, all the variable sets $V(\YY_i (\zeta))$, $i\in [a_\zeta], ||\zeta||\leq l_0$  are pairwise disjoint, and the same holds for the clause sets $F(\YY_i (\zeta))$. In this case, the effect of each cycle can be considered to be independent when summing \eqref{eq:cycleef:trun:Zg expansion} over $\underline{\tau}_\YY$, which gives us
 	\begin{equation}\label{eq:ZsplitbyYtau1}
 	\begin{split}
 	\frac{ \sum_{\underline{\tau}_\YY} \E [\bZ^\prime[g] \one\{\YY, \underline{\tau}_\YY\} \,|\,\uL_E ]}{
 		\E [\bZ^\prime[g]\,|\,\uL_E] 
 	}
 	=\frac{1+ O_k\left(||\underline{a}||_1 n^{-1/2}\log^2 n\right)}{(nd)^{|\bar{\Delta}_c|}}
 	\prod_{||\zeta||\leq l_0} \left(Tr\left[ (\dot{A}_L \hat{A}_L)^\zeta \right]\right)^{a_\zeta},
 	\end{split}
 	\end{equation}
 	where $(\dot{A}_L \hat{A}_L)^\zeta$ defined as \eqref{eq:def:AdotAhat zet}. Also, note that although ${\Delta}$ is defined depending on $\underline{\tau}_\YY$, $|\bar{\Delta}_c|$ in the denominator is well-defined given $\YY$. Thus, averaging the above over all $\uL_E$ gives
 	\begin{equation}\label{eq:cycleef:trun:Zg tau summed}
 	\begin{split}
 	\frac{  \E [\bZ^\prime[g] \one\{\YY\} ]}{
 		\E [\bZ^\prime[g]] 
 	}
 	&=\frac{1+ O_k\left( ||\underline{a}||_1 n^{-1/2}\log^2 n\right)}{(2nd)^{|\bar{\Delta}_c|}}
 	\prod_{||\zeta||\leq l_0} \left(Tr\left[ (\dot{A}_L \hat{A}_L)^\zeta \right]\right)^{a_\zeta}\\
 	&=
 	\left(1+ O_k\left( ||\underline{a}||_1 n^{-1/2}\log^2 n\right)\right)\P(\YY)
 	\prod_{||\zeta||\leq l_0} \left(Tr\left[ (\dot{A}_L \hat{A}_L)^\zeta \right]\right)^{a_\zeta}.
 	\end{split}
 	\end{equation} 
 	Moreover, setting $a^\dagger = \sum_{||\zeta||\leq l_0} a_\zeta ||\zeta||$, the number of ways of choosing $\YY$ to be $\underline{a}$ \textit{disjoint} $\zeta$-cycles can be written by
 	\begin{equation}\label{eq:choice of disj YY}
 	(n)_{a^\dagger} (m)_{a^\dagger }
 	(d(d-1)k(k-1))^{a^\dagger } \prod_{||\zeta||\leq l_0} \left(\frac{1}{2||\zeta||} \right)^{a_\zeta}. 
 	\end{equation}
 	Having this in mind, summing \eqref{eq:cycleef:trun:Zg tau summed} over all  $\YY$ that describes  \textit{disjoint} $\underline{a}$ $\zeta$-cycles, and then over all $||g-g^\star_{\lambda,L}||\leq n^{2/3}$, we obtain that
 	\begin{equation}\label{eq:cycleef:trun:smallg}
 	\frac{ \sum_{||g-g^\star_{\la,L} ||\leq \sqrt{n}\log^2 n} \sum_{\YY \textnormal{ disjoint}} \E[\bZ^\prime[g] \one\{\YY\} ] }{\E[\bZ^\prime]} = \left(1+O_k(||\underline{a}||_1 n^{-1/2}\log^2 n) \right) \left(\underline{\mu} (1+\underline{\delta}_L ) \right)^{\underline{a}}, 
 	\end{equation} 
 	where $\underline{\mu}, \underline{\delta}_L$ are defined as in the statement of the proposition.

 	Our next goal is to deal with $\YY$ such that $\eta(\YY)=\eta>0$ and to show that such $\YY$ provide a negligible contribution. Given $\eta>0$, 
 	this implies that at least $||\underline{a}||_1 - 2\eta$ cycles of $\YY$ should be disjoint from everything else in $\YY$. Therefore, when summing the terms with $H^\star$ in \eqref{eq:cycleef:trun:Zg expansion} over $\underline{\tau}_\YY$, all but at most $2\eta$  cycles contribute by $(1+{\delta}_L(\zeta))$, while the others with intersections can become a different value.  Thus, we obtain that
 	\begin{equation}\label{eq:cycleef:trun:Zg bad YY}
 		\begin{split}
 		\frac{ \sum_{\underline{\tau}_\YY} \E [\bZ^\prime[g] \one\{\YY, \underline{\tau}_\YY\} \,|\,\uL_E ]}{
 			\E [\bZ^\prime[g]\,|\,\uL_E] 
 		}
 		\leq \frac{ 	(1+\underline{\delta}_L)^{\underline{a}}\, C^{2\eta}}{(nd)^{|\bar{\Delta}_c|}}
 	,
 		\end{split}
 	\end{equation}
 	for some constant $C>0$ depending on $k, L, l_0$.

 	Then, similarly to \eqref{eq:choice of disj YY}, we can bound the number of choices $\YY$ satisfying $\eta(\YY)=\eta$. Since all but $2\eta$ of the cycles are disjoint from the others, we have
 	\begin{equation}\label{eq:choice of eta YY}
 	\begin{split}
 	\#\{ \YY \textnormal{ such that }\eta(\YY) = \eta
 	\} &\leq 
 	\left\{ (n)_{|\dot{\Delta}|} (m)_{|\hat{\Delta}| }
 	(d(d-1))^{|\dot{\Delta}|}(k(k-1))^{{}\hat{\Delta}|} (d-2)^{|\bar{\Delta}_c|-2|\dot{\Delta}|} (k-2)^{|\bar{\Delta}| - 2|\hat{\Delta}|} \right\}\\
 &\quad 	\times \left\{ \prod_{||\zeta||\leq l_0} \left(\frac{1}{2||\zeta||} \right)^{a_\zeta} \times (2l_0)^{2\eta} \right\} \times \left\{(a^\dagger)^{\eta } d^{2a^\dagger- |\bar{\Delta}_c| }\right\}.  
 	\end{split}
 	\end{equation}
 	The formula in the \textsc{rhs} can be described as follows.
 	\begin{enumerate}
 		\item The first bracket describes the number of ways to choose variables and clauses, along with the locations of half-edges described by $\YY$. Note that at this point we have not yet chosen the places of variables, clauses and half-edges that are given by the intersections of cycles in $\YY$.
 		
 		\item The second bracket is introduced to prevent overcounting the locations of cycles that are disjoint from all others. Multiplication of $(2l_0)^{2\eta}$ comes from the observation that there can be at most $2\eta$ intersecting cycles.
 		
 		\item The third bracket bounds the number of ways of choosing where to put overlapping variables and clauses, which can be understood as follows.
 		\begin{itemize}
 			\item Choose where to put an overlapping variable (or clause): number of choices bounded by $a^\dagger$.
 			
 			\item If there is an overlapping half-edge adjacent to the chosen variable (or clause), we decide where to put the clause at its endpoint: number of choices bounded by $d$.
 			
 			\item Since there are $2a^\dagger - |\bar{\Delta}_c|$ overlapping half-edges and $2a^\dagger - |\dot{\Delta}|-|\hat{\Delta}|$ overlapping variables and clauses, we obtain the expression \eqref{eq:choice of eta YY}.

 		\end{itemize}
 	\end{enumerate}

 	To conclude the analysis, we need to sum \eqref{eq:cycleef:trun:Zg bad YY} over $\YY$ with $\eta(\YY) = \eta$, using \eqref{eq:choice of eta YY} (and average over $\uL_E$). One thing to note here is the following relation among $|\dot{\Delta}|, |\hat{\Delta}|,$ and $\bar{\Delta}_c$:
 	$$\min \{a^\dagger - |\dot{\Delta}| , \,a^\dagger - |\hat{\Delta}| \} \geq 2a^\dagger - |\bar{\Delta}_c|, $$
 	which comes from the fact that for each overlapping edge, its endpoints count as  overlapping variables and clauses. Therefore, we can simplify \eqref{eq:choice of eta YY} as
 	\begin{equation}\label{eq:choice of eta YY:simplified}
 	\begin{split}
 	\#\{ \YY \textnormal{ such that }\eta(\YY) = \eta
 	\} \leq (nd)^{|\dot{\Delta}| + |\hat{\Delta}|}  2^{2a^\dagger} \underline{\mu}^{\underline{a}} \times (4l_0^2 a^\dagger d^3 k^2)^\eta.
 	\end{split}
 	\end{equation}
 
 	 Thus, we obtain that
 	\begin{equation}\label{eq:ZsumoverYY}
 	\begin{split}
 \sum_{\YY: \eta(\YY)=\eta} \sum_{\underline{\tau}_\YY}	\frac{  \E [\bZ^\prime[g] \one\{\YY, \underline{\tau}_\YY\},|\, \uL_E ]}{
 		\E [\bZ^\prime[g],|\, \uL_E] 
 	}
 	\leq 	2^{2a^\dagger} \left(\underline{\mu}(1+\underline{\delta}_L)\right)^{\underline{a}}\, \left(\frac{C' a^\dagger}{n} \right)^{\eta},
 	\end{split}
 	\end{equation}
 	for another constant $C'$ depending on $k, L, l_0$. We  choose $c_{\textsf{cyc}}=c_{\textsf{cyc}}(l_0)$ to be $2^{2a^\dagger} \leq n^{1/3}$ if $||\underline{a}||_\infty \leq c_{\textsf{cyc}} \log n$.
 	Then, summing this over $\eta\geq 1$ and all $g$ with $||g-g^\star_{\lambda,L}|| \leq \sqrt{n}\log^2 n$ shows that the contributions from $\YY$ with $\eta(\YY)\geq 1$ is negligible for our purposes.  Combining with (\ref{eq:cycleef:trun:smallg}), we deduce the conclusion.
	\end{proof}
	\begin{remark}
	Although we will not use it in the rest of the paper, the analogue of Proposition \ref{prop:cycleeffect:moments}-(1) for $\bZ^{(L),\tr}_{\la}$ holds under the same condition. That is, we have
	\begin{equation}\label{eq:prop:cyceffect:Z:truncated}
	    	\E \left[ \bZ_{\lambda}^{(L),\tr} \cdot (\underline{X})_{\underline{a}} \right] = \left(1+ err(n,\underline{a}) \right) \left(\underline{\mu} ( 1+ \underline{\delta}_L)\right)^{\underline{a}} \E \bZ_{\lambda}^{(L),\tr}.
	\end{equation}
	To prove the equation above, we have from Proposition 3.4 of \cite{ssz16} that
 	\begin{equation}\label{eq:g:faraway:negligbile:truncated}
 	\begin{split}
 	&\sum_{g: ||g-g_{\lambda,L}^\star||_1 \geq \sqrt{n} \log^2 n} \E \big[ \bZ_{\lambda}^{(L),\tr}[g] (\underline{X})_{\underline{a}} \big] \\
 	&\leq 	\sum_{||g-g_{\lambda,L}^\star||_1 \geq \sqrt{n}\log^2 n} \E \big[\bZ_{\lambda}^{(L),\tr}[g]\big] n^{O_k(\log^2 n)} + \E\big[\bZ_{\lambda}^{(L),\tr} (\underline{X})_{\underline{a}} \mathds{1}\{ ||\underline{X}||_{\infty} \ge \log^2 n \}  \big]  \\ &\leq e^{-\Omega_k(\log^4 n)}  \E\bZ_{\lambda}^{(L),\tr}.
 	\end{split}
 	\end{equation}  
    In the second line, we controlled the second term  crudely by using   $\bZ_{\lambda}^{(L),\tr}\le 2^n$ and \eqref{eq:cyclejointasymp}. Having \eqref{eq:g:faraway:negligbile:truncated} in hand, the rest of the proof of \eqref{eq:prop:cyceffect:Z:truncated} is the same as the proof of Proposition \ref{prop:cycleeffect:moments}-(1).
	\end{remark}
    \begin{remark}\label{rmk:Prop:2}
    Having proved Proposition \ref{prop:cycleeffect:moments}-(1), the proof of Proposition \ref{prop:cycleeffect:moments}-(2) is almost identical. Namely, if we consider the empirical coloring profile in the second moment and consider the analogue of $w(g)$ \eqref{eq:def:wg} in the second moment (i.e. replace $\dot{\Phi},\hat{\Phi}^{\lit}$, and $\bar{\Phi}$ in \eqref{eq:def:wg} respectively by $\dot{\Phi}\otimes \dot{\Phi}, \hat{\Phi}^{\lit}(\cdot\oplus\uL_1)\otimes \hat{\Phi}^{\lit}(\cdot\oplus\uL_2)$, and $\bar{\Phi}\otimes \bar{\Phi}$), then the rest of the argument is the same.  
    \end{remark}
	As a corollary, we make an observation that the contribution to $\E\tZ_{\la}$ and $\E\big(\bwZ^{(L),\tr}_{\la}\big)^2$ from too large $X(\zeta)$ is negligible.
	
	\begin{cor}\label{cor:toomanycyc:Zg}
		Let $c>0$, $L>0$, $\lambda\in (0,\lambda^\star_L)$ and $\zeta\in \cup_l \{0,1\}^{2l}$ be fixed. Then,  the following estimates hold true:
		\begin{enumerate}
			\item $\E\Big[ \bwZ_{\lambda}^{(L),\tr} \one \{X(\zeta)\geq c\log n \}\Big] = n^{-\Omega(\log\log n)} \E \bwZ_{\lambda}^{(L)}$;
			
			\item $\E \Big[ \big(\bwZ_{\lambda}^{(L),\tr}\big)^2 \one \{X(\zeta)\geq c\log n \}\Big] = n^{-\Omega(\log\log n)} \E \Big[\big(\bwZ_{\lambda}^{(L),\tr}\big)^2\Big]$;
			
			\item The analogue of (1) is true for the untruncated model with $\lambda \in (0,\lambda^\star)$. Namely, (1) continues to hold when we replace $\bZ_{\lambda}^{(L),\tr}$ by $\bZ_{\lambda}^{\tr}$.
		\end{enumerate}
	\end{cor}
	
	\begin{proof}
		We present the proof of (1) of the corollary; the others will follow by the same idea due to Proposition \ref{prop:cycleeffect:moments}. Let $c_{\textsf{cyc}}=c_{\textsf{cyc}}(||\zeta||)$ be as in Proposition \ref{prop:cycleeffect:moments}, and set  $c'=\frac{1}{2}(c\wedge c_{\textsf{cyc}})$. Then, by Markov's inequality, we have
		\begin{equation*}
		\begin{split}
		\E \left[ \bwZ_{\lambda}^{(L),\tr} \one \{X(\zeta)\geq c\log n \}\right]
		\leq
		\left(\frac{c}{2}\log n
		\right)^{-c'\log n} 
		\E \left[\bwZ_{\lambda}^{(L),\tr} \cdot \left(X(\zeta)\right)_{{c'\log n}} \right].
		\end{split}
		\end{equation*}
		Then, plugging the estimate from Proposition \ref{prop:cycleeffect:moments}-(1) in the \textsc{rhs} implies the conclusion.
	\end{proof}

	To conclude this section, we present an estimate that bounds the sizes of $\delta(\zeta)$ and $\delta_L(\zeta)$. One purpose for doing this is to obtain Assumption (c) of Theorem \ref{thm:smallsubcon}. 
	
	\begin{lemma}\label{lem:deltabound}
		In the setting of Proposition \ref{prop:cycleeffect:moments}, let $\lambda \in (0,\lambda^\star]$ and  $\delta_L$ be defined as \eqref{eq:def:delta by trace}. Then, there exists an absolute constant $C>0$ such that for all $\zeta \in \cup_l \{0,1\}^{2l}$ and $L$ large enough,
		\begin{equation}\label{eq:deltabound}
		 \delta_L(\zeta;\lambda)  \leq (k^C 2^{-k})^{||\zeta||}.
		\end{equation}
		Hence, $\delta(\zeta;\lambda)  \leq (k^C 2^{-k})^{||\zeta||}$ holds by Proposition \ref{prop:cycleeffect:moments}-(4), and we have for large enough $k$,
		\begin{equation*}
		\sum_{\zeta}\mu(\zeta)\delta_L(\zeta;\la)^2 \leq \sum_{l=1}^{\infty} \frac{1}{2l} (k-1)^{l}(d-1)^{l}(k^C 2^{-k})^{2l}<\infty,
		\end{equation*}
		where the last inequality holds because $d\leq k 2^k$ holds by Remark \ref{rem:k:adjusted}. Replacing $\delta_L(\zeta;\la)$ by $\delta(\zeta;\la)$ in the equation above, the analogue also holds for the untruncated model.
	\end{lemma}

	\section{The rescaled partition function and its concentration}\label{subsec:whp:rescaled}
	
	In random regular $k$-\textsc{nae-sat}, it is believed that the primary reason for non-concentration of $\bZ_{\la}^{\tr}$ is the existence of short cycles in the graph. Based on the computations done in the previous section, we show that the partition function is indeed concentrated if we rescale it by the cycle effects. However,   we work with the \textit{truncated} model, since some of our important estimates break down in the untruncated model. The goal of this section is to establish Proposition \ref{prop:concenofrescaled}. 
	
	  To this end, we write the variance of the rescaled partition by the sum of squares of Doob martingale increments with respect to the clause-revealing filtration, and study each increment by using a version of discrete Fourier transform. Although such an idea was also used in \cite{dss16} to study $\bZ_0$, the rescaling factors of the partition function make the analysis more involved and ask for more delicate estimates (for instance, Proposition \ref{prop:cycleeffect:moments}) than what is done in \cite{dss16}. Moreover, an important thing to note is that due to the rescaling, the result we obtain in Proposition \ref{prop:concenofrescaled} is stronger than Proposition 6.1 in \cite{dss16}. This improvement  describes the underlying principle more clearly, which says that the multiplicative fluctuation of the partition function originates from the existence of cycles. 
	
	Although the setting in this section is similar to that in Section 6, \cite{dss16}, we begin with explaining them in brief for completeness.  Then, we focus on the point where the aforementioned improvement comes from, and outline the other technical details  which are essentially analogous to those in \cite{dss16}. \textbf{Throughout this section, we fix  $L\geq 1$, $\lambda\in(0,\lambda^\star_L)$  and $l_0>0$}, which all can be arbitrary. Recall the \textit{rescaled partition function} $\bY \equiv
	\bY_{\lambda,l_0}^{(L)}(\GGG)$ defined in \eqref{eq:def:rescaledPF}:
	\begin{equation*}
	\bY \equiv
	\bY_{\lambda,l_0}^{(L)}(\GGG)
	\equiv
	\bwZ_{\lambda}^{(L),\tr}
	\prod_{\zeta: \,||\zeta||\leq  l_0} (1+\delta_{L}(\zeta)  )^{-X({\zeta})},
	\end{equation*}
	where $\bwZ^{(L),\tr}_{\la}\equiv \sum_{||H-H^\star_{\la,L}||_1\leq n^{-1/2}\log^{2}n}\bZ^{(L),\tr}_{\la}[H]$. We sometimes write $\bY(\GGG)$ to emphasize the dependence on $\GGG = (\GG, \uL)$, the underlying random $(d,k)$-regular graph.
	
	Let $\FF_i$ be the $\sigma$-algebra generated by the first $i$ clauses $a_1,\ldots,a_i$ and the matching of the half-edges adjacent to them. Then, we can write
	\begin{equation*}
	\Var (\bY)
	=
	\sum_{i=1}^m \E \left(\E\left[\left.\bY\right| \FF_i \right] - E\left[\left.\bY\right| \FF_{i-1} \right]\right)^2\equiv \sum_{i=1}^m \Var_i(\bY).
	\end{equation*}   
	For each $i$, let $A$ denote the set of clauses with indices between $i \vee (m-k+1)$ and $m$. Set $\KKK$ to be the collection of \textit{variable-adjacent} half-edges that are matched to $A$. Further, let $\acute{\GGG} = (\acute{\GG}, \acute{\tL})$ be the random $(d,k)$-regular graph coupled to $\GGG$, which has the same clauses $a_1,\ldots,a_{\max{\{i-1,m-k\}}}$ and literals adjacent to them as $\GGG$ and randomly resampled clauses and their literals adjacent to $\KKK$:
	\begin{equation*}
	\begin{split}
	&A \equiv (a_{\max\{i, m-k+1\}},\ldots, a_m );\\
	&\acute{A}\equiv (\acute{a}_{\max{\{i,m-k+1\}}},\ldots, \acute{a}_m).
	\end{split}
	\end{equation*} 
	Let $G^\circ\equiv \GG\setminus A$ be the graph obtained by removing $A$ and the half-edges adjacent to it from $\GG$. Then, for $i\leq m-k+1$, Jensen's inequality implies that
	\begin{equation*}
	\Var_i(\bY)
	\leq
	\E \left(\bY(\GGG) - \bY (\acute{\GGG}) \right)^2
	\leq
	\sum_{A,\acute{A}}
	\E \left(\bY(G^\circ \cup A)
	-
	\bY(G^\circ\cup \acute{A})  \right)^2,
	\end{equation*} 
	where the summation in the \textsc{rhs} runs over all possible matchings $A, \acute{A}$ of $\KKK$ by $k$ clauses (we refer to  Section 6.1 in \cite{dss16} for the details). Note that the sum runs over the finitely many choices depending only on $k$, which is affordable in our estimate. Also, we can write down the same inequality with $i>m-k+1$, for which the only difference is the size of $\KKK$ being smaller than $k^2$. Thus, in the remaining subsection, our goal is to show that for $|\KKK|=k^2 $, there exists an absolute constant $C>0$ such that 
	\begin{equation}\label{eq:incrembd}
	\E \left(\bY( A)
	-
	\bY(\acute{A})  \right)^2 \lesssim_{k,L}
	\frac{(k^C4^{-k})^{l_0}}{n} (\E \bY)^2,
	\end{equation}
	where we denoted $\bY( A)\equiv \bY(G^\circ \cup A)$. This estimate, which is shown at the end of Section \ref{subsubsec:whp:conclusion}, directly implies the conclusion of Proposition~\ref{prop:concenofrescaled}.

	Before moving on, we present an analogue of Corollary \ref{cor:toomanycyc:Zg} for the rescaled partition function. This will function as a useful fact in our later analysis on $\bY$. Due to the rescaling factors in $\bY$, the proof is more complicated than that of Corollary \ref{cor:toomanycyc:Zg}, but still based on similar ideas from Proposition \ref{prop:cycleeffect:moments} and hence we defer it to Section \ref{subsec:app:Ymanycyc}.
	
	\begin{cor}\label{cor:Ymanycyc}
		Let $c>0$, $L>0$, $\lambda\in (0,\lambda^\star_L)$ and $l_0>0 $ be fixed and let $\bY = \bY_{\lambda,l_0}^{(L)}$ as above. Then, for any $\zeta$ such that $||\zeta||\leq l_0$, the following estimates hold true:
		\begin{enumerate}
			\item $\E [ \bY \one \{X(\zeta)\geq c\log n \}] = n^{-\Omega_k(\log\log n)} \E \bwZ_{\lambda}^{(L),\tr}$;
			
			\item $\E [ \bY^2 \one \{X(\zeta)\geq c\log n \}] = n^{-\Omega_k(\log\log n)} \E (\bwZ_{\lambda}^{(L),\tr})^2$;
		\end{enumerate}	
	\end{cor}

	\subsection{Fourier decomposition and the effect of rescaling}\label{subsubsec:whp:fourier}
	
	To see (\ref{eq:incrembd}), we will apply a discrete Fourier transform to $\bY(A)$ and control its Fourier coefficients. We  begin with introducing the following definitions to study the effect of $A$ and $\acute{A}$: Let $B_t^\circ(\KKK)$ denote the ball of graph-distance $t$ in $G^\circ$ around $\KKK$. Hence, for instance, if $t$ is even then the leaves of $B_t^\circ(\KKK)$ are the half-edges adjacent to clauses. Then, we set
	\begin{equation*}
	 T\equiv B_{l_0}^\circ (\KKK).
	\end{equation*}

	Note that $T$ is most likely a union of $|\KKK|$ disjoint trees, but it can contain a cycle with probability $O((dk)^{l_0/2}/n)$. Let $\UUU$ denote the collection of leaves of $T$ other than the ones in $\KKK$, and we write $G^\partial \equiv  G^\circ \setminus T$. 
	
	\begin{remark}[A parity assumption]
		For the rest of Section \ref{subsec:whp:rescaled}, we assume that $l_0$ is \textit{even}. The assumption gives that the half-edges in $\UUU$ are adjacent to clauses of $T$ and hence their counterparts are adjacent to variables of $G^\partial$. For technical reasons in dealing with the rescaling factors (Lemma \ref{lem:localnbd withcycle}), we have to treat the case of odd $l_0$ separately, however it will be apparent that the argument from Sections \ref{subsubsec:whp:fourier}--\ref{subsubsec:whp:conclusion} works the same. In Remark \ref{rmk:l0odd}, we explain the main difference in formulating the Fourier decomposition for an odd $l_0$.
	\end{remark}
	
	Based on the above decomposition of $\GG$, we introduce several more notions as follows. For $\zeta\in \{0,1\}^{2l}$ with $l\leq l_0$, let $X({\zeta})$ and $X^T({\zeta})$ (resp. $\acute{X}({\zeta})$ and $\acute{X}^T({\zeta})$) be the number of $\zeta$-cycles in the graph $G^\circ\cup A = \GG$ and $A\cup T$ (resp. $G^\circ \cup \acute{A} = \acute{\GG}$ and $\acute{A}\cup T$), respectively, and set $$X^\partial (\zeta) \equiv X({\zeta}) - X^T({\zeta}).
	$$
	(Note that this quantity is the same as $\acute{X}({\zeta})-\acute{X}^T({\zeta})$, since the distance from $\UUU$ to $\KKK$ is at least $2l_0$.)
	Based on this notation, we define the \textit{local-neighborhood-rescaled partition function} $\bZ_T$ and $\acute{\bZ}_T$ by
	\begin{equation}\label{eq:def:ZT}
	\begin{split}
	\bZ_T &\equiv
	\bZ'[G^\circ \cup A] \prod_{\zeta: ||\zeta|| \leq l_0} \left(1+\delta_L (\zeta) \right)^{-X^T (\zeta)};\\
	\acute{\bZ}_T &\equiv
	\bZ' [G^\circ \cup \acute{A}]
	\prod_{\zeta: ||\zeta|| \leq l_0} \left(1+\delta_L (\zeta) \right)^{-\acute{X}^T (\zeta)},
	\end{split}
	\end{equation}
	where $\bZ' \equiv \bwZ_{\lambda}^{(L),\tr}$ and $\bZ'[G^\circ \cup A]$ denotes the partition function on the graph $G^\circ \cup A=\GG$. Here, we omitted the dependence on the literals $\uL$ on $\GG$, since we are only interested in their moments.
	
	One of the main ideas of Section \ref{subsec:whp:rescaled} is to relate $\bY$ and $\bZ_T$, by establishing the following lemma:
	\begin{lemma}\label{lem:YvsZT}
		Let $\bY(A) , \bY(\acute{A}), \bZ_T$, $\acute{\bZ}_T$ and $X^\partial$ be defined as above. Then,  we have
		\begin{equation*}
		\begin{split}
		&\E \left[\left(\bY(A)- \bY(\acute{A})\right)^2 \right]\\
		&=
		(1+o(1 ) ) \E \left[\left(\bZ_T - \acute{\bZ}_T \right)^2 \right]
		\exp \left(-\sum_{||\zeta||\leq l_0} \mu(\zeta) (2\delta(\zeta) +\delta(\zeta)^2) \right) 
		+
		O\left(\frac{\log^6 n}{n^{3/2}} \right) \E (\bZ')^2,
		\end{split}
		\end{equation*}
		where  $\bZ'\equiv \bwZ_{\lambda}^{(L),\tr}$ and the error $o(1)$ depends on $L$, $l_0$.
	\end{lemma}
	
	The lemma can be understood as a generalization of Proposition \ref{prop:cycleeffect:moments} to the case of $\bZ_T$. Although the   proof of the lemma is based on similar ideas as the proposition, the analysis becomes more delicate since we need to work with the difference $\bY(A)- \bY(\acute{A})$. The proof will be discussed later in Section \ref{subsec:whp:YvsZT}. 
	
	In the remaining section, we develop ideas to deduce \eqref{eq:incrembd} from Lemma \ref{lem:YvsZT}. To work with  $\bD:=\bZ_T - \acute{\bZ}_T$, we develop a discrete Fourier transform framework as introduced in Section 6 of \cite{dss16}. Recall the definition of the weight factor $w^{\textnormal{lit}}_{\mathcal{G}}(\sig_{\GG})$ on a factor graph $\GG$, which is 
	\begin{equation*}
	w^{\textnormal{lit}}_{\GGG} (\sig_{\GG})\equiv {
	\prod_{v\in V(\GG)} \dot{\Phi}(\sig_v) \prod_{a\in F(\GG)} \hat{\Phi}^{\textnormal{lit}}_a(\sig_a\oplus\uL_a)  \prod_{e\in E(\GG)} \bar{\Phi}(\sigma_e)}.
	\end{equation*}
	 Let $\kappa(\underline{\sigma}_\UUU)$ (resp. $\bZ^\partial(\underline{\sigma}_\UUU)$) denote the contributions to $\bY(A)$ coming from $T\setminus \UUU$ (resp. $G^\partial$) given $\underline{\sigma}_\UUU$, namely,
	\begin{equation}\label{eq:def:Ypartial}
	\begin{split}
	\kappa(\underline{\sigma}_\UUU)
	\equiv
	\kappa(\underline{\sigma}_\UUU, \GGG)
	&\equiv
	\frac{\sum_{\underline{\sigma}_T \sim \underline{\sigma}_\UUU} w^{\textnormal{lit}}_{A\cup T\setminus \UUU}(\underline{\sigma}_{A\cup T\setminus\UUU})^\lambda}{ (1+\underline{\delta}_{L})^{\underline{X}^T} };\\
	\bZ^\partial(\underline{\sigma}_\UUU)
	\equiv
	\bZ^\partial(\underline{\sigma}_\UUU, \GGG)
	&\equiv
	{\sum_{\underline{\sigma}_{G^\partial} \sim \underline{\sigma}_\UUU} w^{\textnormal{lit}}_{G^\partial}(\underline{\sigma}_{G^\partial})^\lambda}.
	\end{split}
	\end{equation}
	where $\underline{\sigma}_T \sim \underline{\sigma}_\UUU$ means that the configuration of $\underline{\sigma}_T$ on $\UUU$ is $\underline{\sigma}_\UUU$. Define $\acute{\kappa}(\underline{\sigma}_\UUU)$ analogously, by  $\acute{\kappa} (\underline{\sigma}_\UUU) \equiv \kappa(\underline{\sigma}_\UUU, \acute{\GGG})$.
	
	
	The main intuition is that the dependence of  $\E \bZ^\partial(\underline{\sigma}_\UUU)$ on $\underline{\sigma}_\UUU$ should be given by the product measure that is i.i.d.~$\dot{q}^\star_{\lambda,L}$ at each $u\in \UUU$, where $\dot{q}^\star_{\lambda,L}$ is the fixed point of the BP recursion we saw in Proposition \ref{prop:BPcontraction:1stmo}. To formalize this idea, we perform a discrete Fourier decomposition with respect to $\underline{\sigma}_\UUU$ in the following setting. Let $(\onb_1,\ldots,\onb_{|\dot{\Omega}_L|})$ be an orthonormal basis for $L^2(\dot{\Omega}_L,\dot{q}^\star_{\lambda,L})$ with $\onb_1\equiv 1$, and let $\bq $ be the product measure $\otimes_{u\in\UUU} \dot{q}^\star_{\lambda,L}$. Extend this to the orthonormal basis $(\onb_{\underline{r}})$ on $L^2((\dot{\Omega}_L)^\UUU, \bq)$ by
	\begin{equation*}
	\onb_{\underline{r}}(\underline{\sigma}_\UUU)
	\equiv
	\prod_{u\in\UUU} \onb_{r(u)}(\sigma_u) \quad \textnormal{for each } \underline{r}\in[|\dot{\Omega}_L|]^\UUU,
	\end{equation*}
	where $[|\dot{\Omega}_L|]:= \{1,2,\ldots, \dot{\Omega}_L \}.$
	 For a function $f$ on $(\dot{\Omega}_L)^\UUU$, we denote its Fourier coefficient by
	\begin{equation*}
	f^\wedge(\underline{r}) \equiv \sum_{\sig_\UUU} f(\sig_\UUU) \onb_{\underline{r}}(\sig_\UUU) \bq(\sig_\UUU).
	\end{equation*}
	Then, defining $\bF(\sig_\UUU)\equiv \bq(\sig_\UUU)^{-1} \bZ^\partial(\sig_\UUU)$, we use Plancherel's identity to obtain that
	\begin{equation*}
	\begin{split}
		\bZ_{\la}^{(L),\tr}(\GGG)\prod_{\zeta: ||\zeta|| \leq l_0} \left(1+\delta_L (\zeta) \right)^{-X^T (\zeta)}= \sum_{\underline{r}} \kappa^\wedge(\underline{r}) \bF^\wedge(\underline{r}).
	\end{split}
	\end{equation*}
	
	\begin{remark}[When $l_0$ is odd]\label{rmk:l0odd}
		If $l_0$ is odd, then the half-edges $\UUU$ are adjacent to the clauses of $G^\partial$. Therefore, the base measure of the Fourier decomposition should be $\hat{q}^\star_{\lambda,L}$ rather than $\dot{q}^\star_{\lambda,L}$. In this case, we rely on the same idea that $\bY^\partial(\sig_{\UUU})$ should approximately be written in terms of the product measure of $\hat{q}^\star_{\lambda,L}$.
	\end{remark}

	To describe the second moment of the above quantity,  we abuse notation and write $\bq$, $\onb$ for the product measure of $\dot{q}_{\lambda,L}^\star \otimes \dot{q}_{\lambda,L}^\star$ on $\UUU$ and the orthonormal basis given by $\onb_{\underline{r}^1,\underline{r}^2}(\sig^1,\sig^2)\equiv \onb_{\underline{r}^1}(\sig^1)\onb_{\underline{r}^2}(\sig^2).$ Moreover, we denote the pair configuration by $\bsig=(\sig^1,\sig^2)$ throughout Section \ref{subsec:whp:rescaled}.
    Let $\prescript{}{2}{\bZ}^\partial(\bsig_\UUU)$ be the contribution of the pair configurations on $G^\partial$ given by
	\begin{equation*}
	\prescript{}{2}{\bZ}^\partial(\bsig_\UUU)
	\equiv
	\prescript{}{2}{\bZ}^\partial(\sig_\UUU^1,\sig_\UUU^2, \GGG)
	\equiv
	{\sum_{\bsig_{G^\partial} \sim \bsig_\UUU} w^\lit_{G^\partial}(\underline{\sigma}^1_{G^\partial})^\lambda w^\lit_{G^\partial}(\sig^2_{G^\partial})^\lambda}.
	\end{equation*}
	Then, denote $\prescript{\bullet}{2}{\bZ}^\partial(\bsig_\UUU)$ by the contribution to $\prescript{}{2}{\bZ}^\partial(\bsig_\UUU)$ from pair coloring profile $||H-H^\bullet_{\la,L}||_1\leq n^{-1/2}\log^{2}n$, where $H^\bullet_{\la,L}$ is defined in Definition \ref{def:opt:coloring:profile:2ndmo}. Recall that $\bZ_T$ is defined in terms of $\bwZ_{\la}^{(L),\tr}$ as in \eqref{eq:def:ZT}. Since $\la<\la^\star_L$ and we restricted our attention to $||H-H^\star_{\la,L}||_1 \leq n^{-1/2}\log^{2}n$ in $\bwZ_{\la}^{(L),\tr}$, the major contribution to the second moment $\E \bD^{2}\equiv \E (\bZ_T-\acute{\bZ}_T)^{2}$ comes from $\E \prescript{\bullet}{2}{\bD}$, where $\prescript{\bullet}{2}{\bD}$ is defined by
	\begin{equation}\label{eq:def:bul2D}
	\prescript{\bullet}{2}{\bD}\equiv 
	\sum_{\bsig_\UUU=(\sig_\UUU^1,\sig_\UUU^2)} (\kappa(\sig_\UUU^1)-\acute{\kappa}(\sig_\UUU^1))(\kappa(\sig_\UUU^2)-\acute{\kappa}(\sig_\UUU^2)) \prescript{\bullet}{2}{\bZ}^\partial(\bsig_\UUU).
	\end{equation}
	Namely, Proposition 4.20 of \cite{nss20a} and Proposition 3.10 of \cite{ssz16} imply that
	\begin{equation*}
	\E \bD^{2}\lesssim_{k} \E 	\prescript{\bullet}{2}{\bD}+e^{-\Omega_k(\log^{4}n)}\big(\E\bwZ_{\la}^{(L),\tr}\big)^{2}.
	\end{equation*}
	Thus, we aim to upper bound $\E \prescript{\bullet}{2}{\bD}$. Let $\E_T$ \textbf{denote the conditional expectation given $T$}. Again using Plancherel's identity, we can write
	\begin{equation}\label{eq:def:2DpartialFourier}
	\E_T \prescript{\bullet}{2}{\bD}\equiv 
	\sum_{(\underline{r}^1,\underline{r}^2)} (\kappa^\wedge(\underline{r}^1)-\acute{\kappa}^\wedge(\underline{r}^1))(\kappa^\wedge(\underline{r}^2)-\acute{\kappa}^\wedge(\underline{r}^2)) \prescript{}{2}{\F}_T^\wedge(\underline{r}^1,\underline{r}^2),
	\end{equation}
	where we wrote 
	\begin{equation}\label{eq:def:2F}
	\prescript{}{2}{\F}_T^\wedge(\underline{r}^1,\underline{r}^2)\equiv 
	\sum_{\bsig_\UUU} \E_T[ \prescript{\bullet}{2}{\bZ}^\partial(\bsig_\UUU)] \onb_{\underline{r}^1,\underline{r}^2}(\bsig_\UUU).
	\end{equation}

	In the remaining subsections, we begin with estimating $\kappa^\wedge$ in Section \ref{subsubsec:whp:localnbd}. This is the part that carries the major difference from \cite{dss16} in the conceptual level, which in turn provides Proposition \ref{prop:concenofrescaled}, a stronger conclusion than Proposition 6.1 of \cite{dss16}.  Then, since the Fourier coefficients $\prescript{}{2}{\F}^\wedge$ deals with the non-rescaled partition function, we may appeal to the analysis given in \cite{dss16} to deduce above (\ref{eq:incrembd}) in Section \ref{subsubsec:whp:conclusion}.
	
	Before moving on, we introduce some notations following \cite{dss16} that are used in the remainder of Section \ref{subsec:whp:rescaled}. We write $\varnothing$ as the index of an all-$1$ vector, that is, $\onb_{\varnothing}\equiv 1$. Moreover, for $\underline{r}=(\underline{r}^1,\underline{r}^2)\in [|\dot{\Omega}_L|]^{2\UUU}$, we define
	\begin{equation*}
	|\{\underline{r}^1 \underline{r}^2 \}| \equiv 
	|\{u\in\UUU: r^1(u) \neq 1 \textnormal{ or } r^2(u)\neq 1  \}|.
	\end{equation*}

	\subsection{Local neighborhood Fourier coefficients}\label{subsubsec:whp:localnbd}
	
	The properties of $\kappa^\wedge$ may vary significantly depending on the structure of $T= B^\circ_{l_0}(\KKK)$. Typically, $T$ consists of $|\KKK|$ disjoint trees, and in this case the rescaling factor has no effect due to the absence of cycles. Therefore, the analysis done in Section 6.4 of \cite{dss16} can be applied to our case as follows. Let $\bT$ be the event that $T$ consists of $|\KKK|$ tree components.  Then, Lemmas 6.8 and 6.9 of \cite{dss16} imply that when $\bT$ tholds, for $\underline{r}\in [|\dot{\Omega}_L|]^{2\UUU}$,
	\begin{itemize}
		\item $\kappa^\wedge(\underline{r}) = \acute{\kappa}^\wedge(\underline{r})$ for all $|\{\underline{r} \}|\leq 1$.
		
		\item $\left.\kappa^\wedge(\varnothing)\right|_{\bT}$ takes a constant value $\overline{\kappa}^\wedge(\varnothing)$ independent of $A$ and the literals on $T$.
		
		\item $|\kappa^\wedge(\underline{r}) - \acute{\kappa}^\wedge(\underline{r})| \lesssim_k  \overline{\kappa}^\wedge(\varnothing)/4^{(k-4)l_0}$ for all $|\{\underline{r}\}|=2$.
		
	\end{itemize}
	
	Moreover, let	$\bC^\circ$ denote the event that $T$ contains a single cycle but consists of $|\KKK|$ connected components, one of which contains a single cycle and others which are trees. In this case,  although the rescaling factor is now non-trivial, it is the same for both $\kappa$ and $\acute{\kappa}$. Therefore, Lemma 6.8 of \cite{dss16} tells us that
	\begin{itemize}
		\item $\kappa^\wedge(\varnothing) = \acute{\kappa}^\wedge(\varnothing)$.
	\end{itemize}

	The case where we notice an important difference is the event 
	$\bC_{t}$, $t\leq l_0$, when $B_{t-1}^\circ(\KKK)$ has $|\KKK|$ connected components but $B_{t'}^\circ$ has $|\KKK|-1$ components for $t\leq t'\leq l_0$. Using the cycle effect, we deduce the following estimate which is stronger than Lemma 6.10 of \cite{dss16}.
	\begin{lemma}\label{lem:localnbd withcycle}
		Suppose that $T\in \bC_{t}$ for some $t\leq l_0$. Then, for any choice of $A$ and $\acute{A}$ of matching $\KKK$ with $k$ clauses, we have
		\begin{equation*}
		\kappa^\wedge(\varnothing) = \acute{\kappa}^\wedge(\varnothing).
		\end{equation*}
	\end{lemma}
	
	\begin{proof}
		 Let $T_0$ and $T_\textsf{link}$ be the connected components of $T$ defined as follows: $T\in \bC_t$ consists of $|\KKK|-2$ copies of isomorphic trees $T_0$ and one tree $T_\textsf{link}$ that contains two half-edges of $\KKK$.  Note that $T\cup A$ and $T\cup \acute{A}$ have different structures only if we are in the following situation:
		\begin{itemize}
			\item One clause in $A$ is connected with both half-edges of $\KKK \cap T_\textsf{link}$. Thus, the connected components of $T\cup A$ are $(k-1)$ copies of $\TT_0$ and one copy of $\TT_\textsf{cyc}$ as illustrated in Figure \ref{fig:localnbdfourier}. (Recall that we assumed $|\KKK|=k^2$ \eqref{eq:incrembd}.) Here, $\TT_0$ is the union of disjoint $k$ copies of $T_0$ and a clause connecting them. Also, $\TT_\textsf{cyc}$ is the union of $k-2$ disjoint copies of $T_0$, one $T_\textsf{link}$, and a clause connecting them. 
			
			\item The two half-edges $\KKK\cap T_{\textsf{link}}$ are connected to different clauses of $\acute{A}$. Therefore, the connected components of $T\cup \acute{A}$ are $(k-2)$ copies of $\TT_0$ and one copy of $\TT_\textsf{link}$. Here, $\TT_{\textsf{link}}$ is the union of $2k-2$ disjoint copies of $T_0$, one $T_{\textsf{link}}$ and two clauses connecting them as illustrated in Figure \ref{fig:localnbdfourier}.
		\end{itemize}
		
	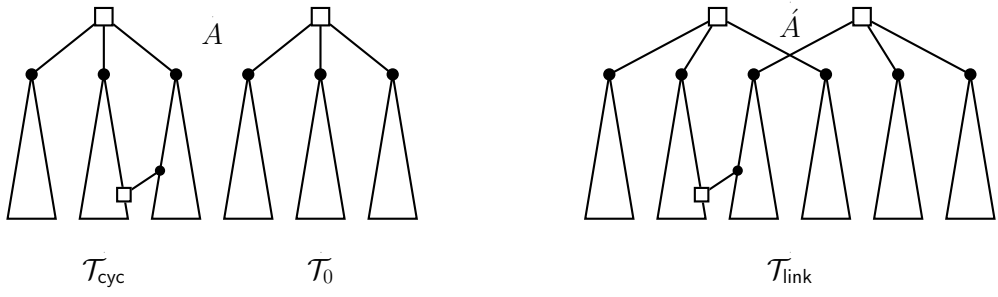
\begin{figure}
		\centering
		\begin{tikzpicture}[square/.style={regular polygon,regular polygon sides=4},thick,scale=0.64, every node/.style={transform shape}]
		\node[square,draw] (a1) at (1.5,4.2) {};
		\node[square,draw] (a2) at (6,4.2) {};
		\node[square,draw] (a3) at (1.5+12+.75,4.2) {};
		\node[square,draw] (a4) at (6+12-.75,4.2) {};
		\foreach \y in {0,...,1}{
			\foreach \x in {0,...,5}{
				\filldraw[black] (1.5*\x+12*\y,3) circle (3pt);	
				\draw (1.5*\x+12*\y,3) -- (1.5*\x+12*\y -0.5,0) -- (1.5*\x+12*\y +0.5,0) --(1.5*\x+12*\y,3);
			}
		}
		\foreach \x in {0,...,2}{
			\draw (a1) -- (1.5*\x,3);
		}
		\foreach \x in {0,...,2}{
			\draw (a2) -- (4.5+1.5*\x,3);
		}
		\draw (a3) -- (12,3);
		\draw (a3) -- (13.5,3);
		\draw (a3) -- (16.5,3);
		\draw (a4) -- (15,3);
		\draw (a4) -- (18,3);
		\draw (a4) -- (19.5,3);
		\draw (2-0.08333333,0.5) -- (2.5+0.16666667,1);
		\draw (12+2-0.08333333,0.5) -- (12+2.5+0.16666667,1);
		\filldraw[black] (2.5+0.16666667,1) circle (2.4pt);
		\filldraw[black] (12+2.5+0.16666667,1) circle (2.4pt);
		\filldraw[white] (2-0.08333333,0.5) +(-4pt,-4pt) rectangle +(4pt,4pt) ;
		\draw (2-0.08333333,0.5) +(-4pt,-4pt) rectangle +(4pt,4pt) ;
		\filldraw[white] (12+2-0.08333333,0.5) +(-4pt,-4pt) rectangle +(4pt,4pt) ;
		\draw (12+2-0.08333333,0.5) +(-4pt,-4pt) rectangle +(4pt,4pt) ;
		
		\filldraw[black] (1.5,-.7) circle (0pt) node[black, anchor=north] {\LARGE{$\TT_\textsf{cyc} $}};
		
		\filldraw[black] (6,-.7) circle (0pt) node[black, anchor=north] {\LARGE{$\TT_0 $}};
		
		\filldraw[black] (12-2.25+6,-.7) circle (0pt) node[black, anchor=north] {\LARGE{$\TT_\textsf{link} $}};
		
			\filldraw[black] (3.75,4.2) circle (0pt) node[black, anchor=north] {\LARGE{$A $}};
			
			\filldraw[black] (3+12.75,4.5) circle (0pt) node[black, anchor=north] {\LARGE{$\acute{A} $}};
		\end{tikzpicture}
		\caption{An illustration of the graphs $A\cup T$ (left) and $\acute{A}\cup T$ (right). }\label{fig:localnbdfourier}
	\end{figure}

	Let $\kappa_0^\wedge$ and $\kappa_\textsf{cyc}^\wedge$ (resp.  $\kappa_\textsf{link}^\wedge$) be the contributions to $\kappa^\wedge(\varnothing)$ (resp. $\acute{\kappa}^\wedge(\varnothing)$) from $\TT_0$ and $\TT_\textsf{cyc}$, respectively (resp. $\TT_{\textsf{link}}$). Then, we have
	\begin{equation}\label{eq:kappahat decomp}
	\kappa^\wedge(\varnothing) = (\kappa_0^\wedge)^{k-1} \kappa^\wedge_\textsf{cyc} ,\quad\textnormal{and}\quad
	\acute{\kappa}^\wedge(\varnothing)
	=(\kappa_0^\wedge)^{k-2} \kappa_\textsf{link}^\wedge.
	\end{equation}
	In what follows, we present an explicit computation of $\kappa_0^\wedge$, $\kappa_\textsf{cyc}^\wedge$ and  $\kappa_\textsf{link}^\wedge$ and show that the two quantities in \eqref{eq:kappahat decomp} are the same.
	
	We begin with computing $\kappa_0^\wedge$. Since we are in a tree, $\kappa_0^\wedge$ does not depend on the assignments of literals, and hence we can replace the weight factor $w^{\textnormal{lit}}$ by its averaged version $w$. Let $e_0$ (resp. $\YYY_0$) be the root half-edge (resp. the collection of leaf half-edges)  of $T_0$.  We define
	\begin{equation}\label{eq:def:varkap}
	\varkappa_0 (\sigma; \sig_{\YYY_0}) 
	\equiv
	\sum_{\sig_{T_0}\sim (\sigma, \sig_{\YYY_0})} w(\sig_{T_0})^\lambda, 
	\end{equation}
	where $\sig_{T_0}\sim (\sigma, \sig_{\YYY_0})$ means that $\sig_{T_0}$ agrees with $\sigma$ and $\sig_{\YYY_0}$ at $e_0$ and $\YYY_0$, respectively. Note that since $\TT_0$ is a tree, the rescaling factor from the cycle effect is trivial. Denoting the number of variables and clauses of $T_0$ by $v(T_0)$ and $a(T_0)$, respectively, the Fourier coefficient of $\varkappa_0(\sigma;\,\cdot\,)$ at $\varnothing$ is given by
	\begin{equation}\label{eq:varkap0hat}
	\varkappa_0^\wedge(\sigma) 
	\equiv
	\sum_{\sig_{\YYY_0}} \varkappa_0(\sigma;\sig_{\YYY_0}) \bq(\sig_{\YYY_0}) = 
	\dot{q}^\star_{\lambda,L}(\sigma) \dot{\ZZZ}^{v(T_0)} \hat{\ZZZ}^{a(T_0)},
	\end{equation}
	where the second equality follows from the fact that $\dot{q}^\star_{\lambda,L}$ and the constants $\dot{\ZZZ} = \dot{\ZZZ}_{{q}^\star_{\lambda,L}}$ and $\hat{\ZZZ} = \hat{\ZZZ}_{q^\star_{\lambda,L}}$ are the fixed point and the normalizing constants of the Belief Propagation recursions \eqref{eq:def:BP}. Thus, we can calculate $\kappa_0^\wedge$ by
	\begin{equation}\label{eq:kap0hat}
	\kappa_0^\wedge = 
	\sum_{\sig \in (\dot{\Omega}_L)^k} \hat{\Phi}(\sig) \prod_{i=1}^k \varkappa_0^\wedge(\sigma_i)
	= \hat{\mathfrak{Z}}\,\dot{\ZZZ}^{k\,v(T_0)} \hat{\ZZZ}^{k\,a(T_0)},
	\end{equation}
	where $\hat{\mathfrak{Z}}$ is the normalizing constant of $\hat{H}^\star_{\lambda,L}$ given by \eqref{eq:H:q:1stmo}. Since $\TT_{\textsf{link}}$ is a tree, we can compute $\kappa_\textsf{link}^\wedge$ using the same argument, namely,
	\begin{equation}\label{eq:kaplinkhat}
	\kappa_\textsf{link}^\wedge 
	=
	\hat{\mathfrak{Z}}\,
	\dot{\ZZZ}^{(2k-2)v(T_0)+v(T_\textsf{link})} \hat{\ZZZ}^{(2k-2)a(T_0)+ a(T_\textsf{link})+1},
	\end{equation}
	since the total number of variables and clauses in $\TT_{\textsf{link}}$ are $(2k-2)v(T_0)+v(T_\textsf{link})$ and $(2k-2)a(T_0)+ a(T_\textsf{link})+2$.
	
	What remains is to calculate $\kappa_\textsf{cyc}^\wedge$.  There is a single cycle of length $2t$ in the graph $T\cup A$, and let this be a $\zeta$-cycle with $\zeta\in \{0,1\}^{2t}$. Unlike the previous two cases, the literal assignment $\zeta$ actually has a non-trivial effect, but still the literals outside of the cycle can be ignored.  We compute
	$$\tilde{\kappa}_\textsf{cyc}^\wedge = \kappa_\textsf{cyc}^\wedge \cdot	Tr\left[ \prod_{i=1}^t \dot{A}_L \hat{A}_L^{\zeta_{2i-1},\zeta_{2i}}  \right], $$
	which does not include the rescaling term by the cycle effect.
	 Let $C$ denote the cycle in  $\TT_{\textsf{cyc}}$ and $2t$ be its length. 
	Let $\YYY_{C}$ be the half-edges that are adjacent to but not contained in $C$. Hence, $t(d-2)$ (resp. $t(k-2)$) half-edges in $\YYY_{C}$ are adjacent to a variable (resp. a clause) in $C$. 
	
	For each $u\in \YYY_{\textsf{cyc}}$, let $T_u$ denote the connected component of $\TT_{\textsf{cyc}}\setminus \{u \}$ that is a tree. Let $e_u$ denote the root half-edge of $T_u$, that is, the half-edge that is matched with $u$ in $\TT_{\textsf{cyc}}$, and $\varkappa_u(\sigma;\,\cdot\,)$ be defined analogously as \eqref{eq:def:varkap}. Then, according to the same computation as \eqref{eq:varkap0hat}, we obtain that
	\begin{equation}\label{eq:varkapuhat}
	\varkappa_u^\wedge(\sigma_u) = 
	\begin{cases}
	\dot{q}_{\lambda,L}^\star (\sigma_u) 
	\dot{\ZZZ}^{v(T_u)} \hat{\ZZZ}^{a(T_u)} , &\textnormal{if } u\textnormal{ is adjacent to a clause in }C,\\
	\hat{q}_{\lambda,L}^\star (\sigma_u) 
	\dot{\ZZZ}^{v(T_u)} \hat{\ZZZ}^{a(T_u)} , &\textnormal{if } u\textnormal{ is adjacent to a variable in }C.
	\end{cases}
	\end{equation}
	Furthermore, for convenience we denote the set of variables, clauses and edges of $C$ by $V, F$, and $E$, respectively and setting $\YYY \equiv \YYY_{C}\cup E$. For each $a\in F$, denote the two literals on $C$ that are adjacent to $a$ by $\zeta_a^1, \zeta_a^2$. Observe that $\kappa_\textsf{cyc}^\wedge$ can be written by
	\begin{align}
	\tilde{\kappa}_\textsf{cyc}^\wedge &=
	\sum_{\sig_{\YYY}} 
	{\prod_{v\in V} \dot{\Phi}(\sig_v)^\lambda \prod_{a\in F} \hat{\Phi}^{\zeta_a^1,\zeta_a^2}(\sig_a)^\lambda}{
	\prod_{e\in E} \bar{\Phi}(\sig_e)^\lambda}
	\prod_{u\in\YYY_{C}} \varkappa_u^\wedge (\sigma_u) \label{eq:kapcychat primary}\\
	&=
	\dot{\ZZZ}^{\sum_{u\in\YYY_{C}} v(T_u) }
	\hat{\ZZZ}^{\sum_{u\in\YYY_{C}} a(T_u) }
	\sum_{\sig_{\YYY}} \frac{
	\prod_{v\in V}	\dot{H}^\star(\sig_v) \prod_{a\in F} \hat{H}^{\zeta_a^1,\zeta_a^2}(\sig_a)
}{\prod_{e\in E} \bar{H}^\star(\sigma_e)}\,
\frac{
\dot{\mathfrak{Z}}^t \hat{\mathfrak{Z}}^t
}{\bar{\mathfrak{Z}}^{2t}}, \label{eq:kapcyc hat rep} 
	\end{align}
	where the second equality is obtained by multiplying  $\prod_{e\in E} \dot{q}^\star_{\lambda,L}(\sigma_e) \hat{q}^\star_{\lambda,L}(\sigma_e)$ both in the numerator and denominator of the first line. Moreover,  the normalizing constant for $\hat{H}^{\zeta_1,\zeta_2}$ is the same regardless of $\zeta_1,\zeta_2$ (see  \eqref{eq:def:hatHlit}). (Note that in the RHS we wrote $\dot{H}^\star\equiv \dot{H}^\star_{\lambda,L}$ and similarly for $\hat{H}^{\zeta_1,\zeta_2}, \bar{H}^\star$.) The literal assignments did not play a role in the previous two cases of $\TT_0$, $\TT_{\textsf{link}}$ which are trees, but in $\TT_{\textsf{cyc}}$ their effect is non-trivial in principle due to the existence of the cycle $C$. Plugging the identities $\dot{\mathfrak{Z}}=\dot{\ZZZ}\bar{\mathfrak{Z}}$ and $\hat{\mathfrak{Z}}=\hat{\ZZZ}\bar{\mathfrak{Z}}$ into \eqref{eq:kapcyc hat rep}, we deduce that 
	\begin{equation*}
	\tilde{\kappa}_\textsf{cyc}^\wedge =
	\dot{\ZZZ}^{v(\TT_{\textsf{cyc}})} \hat{\ZZZ}^{a(\TT_{\textsf{cyc}})} 
	\cdot
	Tr\left[\prod_{i=1}^t \dot{A}_L \hat{A}_L^{\zeta_{2i-1},\zeta_{2i}} \right],
	\end{equation*}
	and hence $\tilde{\kappa}_\textsf{cyc}^\wedge =
	\dot{\ZZZ}^{v(\TT_{\textsf{cyc}})} \hat{\ZZZ}^{a(\TT_{\textsf{cyc}})} $. 
	 Therefore, combining this result with \eqref{eq:kappahat decomp}, \eqref{eq:kap0hat} and \eqref{eq:kaplinkhat}, we obtain the conclusion $\kappa^\wedge(\varnothing) = \acute{\kappa}^\wedge(\varnothing) $.
\end{proof}

\subsection{The martingale increment estimate and the proof of Proposition \ref{prop:concenofrescaled}}\label{subsubsec:whp:conclusion}

We begin with establishing \eqref{eq:incrembd} by combining the discussions in the previous subsections. The proof follows by the same argument as Section 7, \cite{dss16}, along with plugging in the improved estimate Lemma \ref{lem:localnbd withcycle} and obtaining an estimate on $\E \bY$ using Proposition \ref{prop:cycleeffect:moments}. 

To this end, we first review the result from \cite{dss16} that gives the estimate on the Fourier coefficients $\prescript{}{2}{\F}^\wedge $ defined in \eqref{eq:def:2F}. In \cite{dss16} Lemma 6.7 and the discussion below, it was shown that
\begin{equation}\label{eq:F2fourierbd:nonrescaled}
	\frac{\prescript{}{2}{\F}_T^\wedge(\underline{r}^1,\underline{r}^2) }{\prescript{}{2}{\F}_T^\wedge(\varnothing)} \lesssim_{k,L}
	\begin{cases}
	n^{-1/2}, &\textnormal{for } |\{\underline{r}^1,\underline{r}^2 \}|= 1;\\
	n^{-1}, &\textnormal{for } |\{\underline{r}^1,\underline{r}^2 \}|= 2;\\
	n^{-3/2} (\log n)^6 &\textnormal{for } |\{\underline{r}^1,\underline{r}^2 \}|\geq 3,
	\end{cases}
\end{equation}
independent of $T$. (The logarithmic factor for $|\{\underline{r}^1, \underline{r}^2 \}| \geq 3$ is slightly worse than that of \cite{dss16}, since we work with $g$ such that $||g-g^\star||\leq \sqrt{n}\log^2 n$, not $||g-g^\star||\leq \sqrt{n}\log n$.) Based on this fact and the analysis from Section \ref{subsubsec:whp:localnbd}, our first goal in this subsection is to establish the following:

\begin{lemma}\label{lem:increbd:ZT}
	Let $L>0, \lambda\in (0,\lambda^\star_L) $ and $l_0 >0$ be fixed, and let $\bZ_T$ and $\acute{\bZ}_T$ be given as \eqref{eq:def:ZT}. Then, there exist an absolute constant $C>0$ and a constant $C_{k,L}>0$ such that for large enough $n$,
	\begin{equation}\label{eq:increbd:ZT}
	\E\left[ \left(\bZ_T -\acute{\bZ}_T \right)^2 \right]
	\leq \frac{C_{k,L}}{n} (k^C 4^{-k})^{l_0} (\E \bZ')^2,
	\end{equation}
	where $\bZ' = \bwZ_{\lambda}^{(L),\tr} $
\end{lemma}

\begin{proof}
	Let $\prescript{\bullet}{2}{\bD}$ be defined as  \eqref{eq:def:bul2D}. Based on the expression \eqref{eq:def:2DpartialFourier}, we study the conditional expectation $\E_T \prescript{\bullet}{2}{\bD}$ for different shapes of $T$. To this end, we first recall the events $\bT$, $\bC^\circ$ and $\bC_t$ defined in the beginning of Section \ref{subsubsec:whp:localnbd}. We additionally write
	\begin{equation}\label{eq:def:Bevent}
		\bB \equiv \left(\cup_{t\leq l_0} \bC_t \cup \bT \cup \bC^\circ \right)^c .
	\end{equation}
 We remark that the same event $\bB$ was also considered in the proof of Proposition 6.1 in \cite{dss16}. 
 
 Note that $T$ can be constructed from a configuration model in a depth $\ell$ neighborhood of $\KKK$ which is of size $O_k(1)$.  Revealing the edges of these neighborhoods one by one, each new edge creates a cycle with probability $O_k(1/n)$. The event $\bC^\circ$ requires a single cycle so by a union bound $\P(\bC^\circ)=O_k(1/n)$ while the event $\bB$ requires at least two cycles so again by a union bound $\P(\bB)=O_k(n^{-2})$ holds.
 
	For each event above, we can make the following observation. When we have $\bT$, the only contribution to $\E [\prescript{\bullet}{2}{\bD} ; \, \bT ]$ comes from $(\underline{r}^1, \underline{r}^2)$ such that $|\{\underline{r}^1, \underline{r}^2 \}| \geq 2$, due to the properties of $\kappa^\wedge$ discussed in the beginning of Section \ref{subsubsec:whp:localnbd}. 
	Note that the number of choices of $(\underline{r}_1, \underline{r}_2)$ with $|\{\underline{r}^1, \underline{r}^2 \}| = 2$ is $\leq |\dot{\Omega}_L|^4 (k^54^k)^{l_0}$.
	Therefore, \eqref{eq:F2fourierbd:nonrescaled} gives that
	\begin{equation}\label{eq:increbd:T}
	\begin{split}
	\frac{\E [ \prescript{\bullet}{2}{\bD} ; \, \bT] }{(\E \bZ' )^2}
	&\leq
	\left(\frac{\kappa^\wedge(\varnothing)}{4^{(k-4)l_0}} \right)^2
	\frac{\prescript{}{2}{\F}^\wedge_T (\varnothing)   }{ (\E \bZ')^2} \left[\frac{(k^5 4^k)^{l_0} |\dot{\Omega}_{L}|^2}{n} + \frac{\log^6 n}{n^{3/2}} \right]\\
	& \lesssim_{k,L} \frac{(k^C 4^{-k})^{l_0} }{n}.
	\end{split}
	\end{equation}
	Similarly on $\bC^\circ$, the analysis on $\kappa^\wedge$ implies that there is no contribution from $(\underline{r}^1,\underline{r}^2)= \varnothing$. Thus, we obtain from \eqref{eq:F2fourierbd:nonrescaled} that
	\begin{equation}\label{eq:increbd:C0}
		\frac{\E [ \prescript{\bullet}{2}{\bD} ; \, \bC^\circ] }{(\E \bZ' )^2}
		\leq
		\P(\bC^\circ) \cdot O_L\left(n^{-1/2}\right) =O_{k,L}(n^{-3/2})\,.
	\end{equation}
	Since the event $\bB$ has probability $\P(\bB) =O_k(n^{-2})$, we also have that
	\begin{equation}\label{eq:increbd:B}
	\frac{\E [ \prescript{\bullet}{2}{\bD} ; \, \bB] }{(\E \bZ' )^2} =O_k(n^{-2}).
	\end{equation}
	
	The last remaining case is $\bC_t$, and this is where we get a nontrivial improvement compared to \cite{dss16}. Lemma \ref{lem:localnbd withcycle} tells us that there is no contribution from  $(\underline{r}_1,\underline{r}_2) = \varnothing$. Thus, similarly as \eqref{eq:increbd:C0}, for each $t\leq l_0$ we have
	\begin{equation}\label{eq:increbd:Ct}
		\frac{\E [ \prescript{\bullet}{2}{\bD} ; \, \bC_t] }{(\E \bZ' )^2}
		\leq
		\P(\bC_t) \cdot O_L\left(n^{-1/2}\right) \lesssim_{k,L} \frac{\log n}{n^{3/2}}.
	\end{equation}
	Thus, combining the equations \eqref{eq:increbd:T}--\eqref{eq:increbd:Ct}, we obtain the conclusion.
	\end{proof}

To obtain the conclusion of the form \eqref{eq:incrembd}, we need to replace $(\E\bZ')^2$ in \eqref{eq:increbd:ZT} by $(\E\bY)^2$. This follows from Proposition \ref{prop:cycleeffect:moments} and can be summarized as follows.

\begin{cor}\label{cor:EYbyEZ}
	Let $L>0$, $\lambda\in(0,\lambda^\star_L)$ and $l_0>0$ be fixed, and let $\bY\equiv \bY_{\lambda,l_0}^{(L)}$ be the rescaled partition function defined by \eqref{eq:def:rescaledPF}. Further, let $\underline{\mu}$, $\underline{\delta}_L$ be as in Proposition \ref{prop:cycleeffect:moments}. Then, we have
	\begin{equation*}
	\E \bY = \left(1+O\left(\frac{\log^3 n}{n^{1/2}}\right)\right) \E \bZ' 
	\cdot \left\{\exp\left( -\sum_{||\zeta||\leq l_0} \mu(\zeta) \delta_L(\zeta) \right) +o(n^{-1})  \right\},
	\end{equation*}
	where $\bZ^\prime\equiv \bwZ^{(L),\tr}_{\la}$.
\end{cor}

\begin{proof}
	Let $c_\textsf{cyc}=c_\textsf{cyc}(l_0)$ be given as Proposition \ref{prop:cycleeffect:moments}.  Corollary \ref{cor:Ymanycyc} shows that  $\E \bY \one\{||\underline{X}||_\infty \geq c_{\textsf{cyc}} \log n \}$ is negligible for our purposes, and hence we focus on estimating $\E \bY \one\{||\underline{X}||_\infty \leq c_{\textsf{cyc}} \log n \}$.
	
	Note that for an integer  $x\geq 0$, $(1+\theta)^x = \sum_{a\geq 0} \frac{(x)_a}{a!} \theta^a $. Thus, if we define $\tilde{\delta}(\zeta) \equiv (1+\delta_L(\zeta))^{-1} -1$, we can write
	\begin{equation*}
	\begin{split}
	\E [\bY \one \{||\underline{X}||_\infty \leq c_{\textsf{cyc}} \log n \}] 
	&=
	\sum_{\underline{a}\geq 0}  \frac{1}{\underline{a}!} 
	\E \left[ 
	\bZ' (\tilde{\underline{\delta}})^{\underline{a}} (\underline{X})_{\underline{a}} \one \{ ||\underline{X}||_\infty\leq c_{\textsf{cyc}}\log n \}
	\right]\\
	&=
	\left(1+O\left(\frac{\log^3 n}{n^{1/2}}\right)\right)
	\sum_{ ||\underline{a}||_\infty \leq c_{\textsf{cyc}}\log n }
	\frac{1}{\underline{a}!} \E \bZ' \left(\tilde{\underline{\delta}} \underline{\mu} (1+ \underline{\delta}_L) \right)^{\underline{a}},
	\end{split}
	\end{equation*}
	and performing the summation in the \textsc{rhs} easily implies the conclusion.
\end{proof}

We conclude this subsection by presenting the proof of  Proposition \ref{prop:concenofrescaled}.

\begin{proof}[Proof of Proposition \ref{prop:concenofrescaled}]
	As discussed in the beginning of Section \ref{subsec:whp:rescaled}, it suffices to establish \eqref{eq:incrembd} to deduce Proposition \ref{prop:concenofrescaled}.  Combining Lemmas \ref{lem:YvsZT}, \ref{lem:increbd:ZT} and Corollary \ref{cor:EYbyEZ} gives that
	\begin{equation*}
	\frac{\E[(\bY(A) - \bY(\acute{A}))^2 ] }{ (\E \bY)^2 }
	\leq 
	\frac{1}{n} (k^C 4^{-k})^{l_0} \exp \left(\sum_{||\zeta||\leq l_0} \mu(\zeta) \delta_L(\zeta)^2 \right) + O\left(\frac{\log^6 n}{n^{3/2}} \right),
	\end{equation*}
	for some absolute constant $C>0$.
	Moreover, Lemma \ref{lem:deltabound} implies that 
	$$\sum_{\zeta} \mu(\zeta) \delta_L(\zeta)^2 <\infty , $$
	hence establishing \eqref{eq:incrembd}.
\end{proof}

	\subsection{Proof of Lemma \ref{lem:YvsZT}}\label{subsec:whp:YvsZT}
	
	In this subsection, we establish Lemma \ref{lem:YvsZT}. One nontrivial aspect of this lemma is achieving the error $O(n^{-3/2} \log^6 n) \E[(\bZ')^2]$, where $\bZ^\prime\equiv \bwZ_{\la}^{(L),\tr}$. For instance, there can be short cycles in $\GG$ intersecting $T$ (but not included in $T$) with probability $O(n^{-1})$, and in principle this will contribute by $O(n^{-1}) $ in the error term. One observation we will see later is that the effect of these cycles wears off since we are looking at the difference $\bY(A) - \bY(\acute{A}) $ between rescaled partition functions.
	
	To begin with, we decompose the rescaling factor (which is exponential in $\underline{X}^\partial$) into the sum of polynomial factors based on an elementary fact we also saw in the proof of Corollary \ref{cor:EYbyEZ}: for a nonnegative integer $x$, we have $(1+\theta)^x = \sum_{a\geq 0} \frac{(x)_a}{a!} \theta^a$. Let $\tilde{\delta}(\zeta)= (1+\delta_L(\zeta) )^{-2}-1$, and write
	\begin{equation}\label{eq:expbyfallingfac}
	\left(1+\underline{\delta}_L \right)^{-2\underline{X}^\partial}
	=
	\sum_{\underline{a}\geq 0} \frac{1}{\underline{a}!} \tilde{\underline{\delta}}^{\underline{a}} (\underline{X}^\partial)_{\underline{a}}.
	\end{equation}
	
	Therefore, our goal is to understand $\E [(\bZ_T-\acute{\bZ}_T)^2 (\underline{X}^\partial)_{\underline{a}} ]$ which can be described as follows.
	
	\begin{lemma}\label{lem:ZTdiff:poly}
		Let $L>0$, $\lambda\in(0,\lambda_L^\star)$ and $l_0>0$ be fixed, set $\underline{\mu}$, $\underline{\delta}_L$ as in Proposition \ref{prop:cycleeffect:moments}, and let $\bZ_T, \acute{\bZ}_T$ be defined as \eqref{eq:def:ZT}. For any $\underline{a}=(a_\zeta)_{||\zeta||\leq l_0} $ with $||\underline{a}||_\infty \leq \log^2 n$, we have
		\begin{equation}\label{eq:ZTdiff:poly}
		\begin{split}
		\E \left[\left(\bZ_T -\acute{\bZ}_T \right)^2 (\underline{X}^\partial )_{\underline{a}} \right]
		=&
		\left(1+ O\left(\frac{||\underline{a}||_1^2}{n} \right) \right)
			\E \left[\left(\bZ_T -\acute{\bZ}_T \right)^2\right] 
			\left(\underline{\mu} (1+\underline{\delta}_L)^2 \right)^{\underline{a}} \\
			&\quad+ O\left(\frac{||\underline{a}||_1 \log^6 n }{n^{3/2}}  \right) \E [(\bZ')^2].
		\end{split}
		\end{equation}
	\end{lemma}

	The first step towards the proof is to write the \textsc{lhs} of \eqref{eq:ZTdiff:poly} using the Fourier decomposition as in Section \ref{subsubsec:whp:fourier}. To this end, we  recall Definitions 	\ref{def:zetcycle}, \ref{def:cycprofile} (but now $\Delta$ counts the number of \textit{pair-coloring} configurations around variables, clauses, and half-edges) and decompose $(\underline{X}^\partial)_{\underline{a}}$ similarly as the expression \eqref{eq:ZXbyindicators}. Hence, we write
	\begin{equation*}
	\E_T \left[\left(\bZ_T -\acute{\bZ}_T \right)^2 (\underline{X}^\partial )_{\underline{a}} \right]
	=\sum_{\YY} \sum_{\underline{\tau}_\YY} 
	\E_T \left[\left(\bZ_T -\acute{\bZ}_T \right)^2 \one\{\YY,\bsig_\YY \} \right],
	\end{equation*}
	where $\YY = \{\YY_i(\zeta) \}_{i\in[a_\zeta],\, ||\zeta||\leq l_0}$ denotes the locations of $\underline{a}$ $\zeta$-cycles and $\bsig_\YY$ describes a prescribed coloring configuration on them.
	
	 In what follows, we fix a tuple $(\YY,\bsig_\YY)$ and work with the summand of above via Fourier decomposition. Let
	 $$U \equiv \UUU \cap \left(\cup_{v\in V(\YY)} \delta v \right) $$
	 be the set of half-edges in $\UU$ that are adjacent to a variable in $\YY$. Since the colors on $U$ are already given by $\bsig_\YY$, we will perform a Fourier decomposition in terms of $\bsig_{\UUU'}$, with $\UUU' \equiv \UUU \setminus U$. Let $\kappa(\sig_{\UUU'} ;  \sig_\YY) $ (resp. $\acute{\kappa}(\sig_{\UUU'} ;  \sig_\YY) $) be the partition function on $T \cup A$ (resp. $T\cup \acute{A}$) (in terms of the single-copy model), under the prescribed coloring configuration $\sig_{\UUU'}$ on $\UUU'$ and $\sig_{\YY\cap T}$ on $\YY\cap T$. Setting  $$\varpi(\,\cdot\;; \sig_{\YY})  \equiv \kappa(\,\cdot \;;  \sig_\YY) - \acute{\kappa}(\,\cdot \;;  \sig_\YY),$$  and writing $\bsig_\YY = (\sig_\YY^1, \sig_\YY^2 )$, we obtain by following the same idea as \eqref{eq:def:bul2D} that 
	\begin{equation}\label{eq:ZTincre decomposition}
	\begin{split}
		&\E_T \left[\left(\bZ_T -\acute{\bZ}_T \right)^2 \one\{\YY,\bsig_\YY\} \right]\\
		&=
		\sum_{\bsig_{\UUU'}=(\sig_{\UUU'}^1,\sig_{\UUU'}^2)}\varpi(\sig_{\UUU'}^1; \sig_\YY^1)\varpi(\sig_{\UUU'}^2; \sig_\YY^2) \E_T \left[\prescript{}{2}{\bZ}^\partial(\underline{\tau}_{\UUU'};\Gamma_2^\bullet ) \one\{ \YY,\bsig_\YY \} \right] 
		+
		e^{-\Omega(n)} \E [(\bZ')^2].
	\end{split}
	\end{equation}
	Note that $(\underline{X}^\partial)_{\underline{a}}$ is deterministically bounded by $\exp(O(\log^3 n))$, and hence at the end the second term will have a negligible contribution due to  $\exp(-\Omega(n) )$, which comes from the correlated pairs of colorings. Then, we investigate 
	\begin{equation}\label{eq:ZXcycbyindicators}
	\E_T \left[\prescript{}{2}{\bZ}^\partial(\underline{\tau}_\UUU;\Gamma_2^\bullet ) \one\{\YY,\bsig_\YY\} \right] .
	\end{equation}
	  To be specific,  we want to derive the analog of Lemma 6.7, \cite{dss16}, which dealt with $\E_T \left[\prescript{}{2}{\bZ}^\partial(\bsig_\UUU;\Gamma_2^\bullet )\right]$ without having the planted cycles inside the graph. 	To explain the main computation, we introduce several notations before moving on. Let $\bar{\Delta}$, $\bar{\Delta}_U$ be counting measures on $\dot{\Omega}_L^2$ defined as
	\begin{equation*}
	\begin{split}
	&\bar{\Delta}(\btau) = 
	|\{e\in E_c(\YY) \setminus (E(T) \cup U) : \bsigma_{e} = \btau \} |, \quad
	\textnormal{for all } \btau\in\dot{\Omega}_L^2;\\
	&\bar{\Delta}_U ({\btau}) = 
	|\{e\in U: \bsigma_{e} = \btau \} |, \quad
	\textnormal{for all } \btau\in\dot{\Omega}_L^2.
	\end{split}
	\end{equation*}
	Note that $\bar{\Delta}$ and $\bar{\Delta}_U$ indicate empirical counts of edge-colors on disjoint sets. Moreover, for a given coloring configuration $\bsig_\YY$ on $\YY$, we define $\Delta_\partial =(\dot{\Delta}_\partial , (\hat{\Delta}^{\uL}_\partial)_{\uL})$, the \textit{restricted empirical profile on}  $\YY\setminus T$, by
	\begin{equation*}
	\begin{split}
	\dot{\Delta}_\partial ( \bsig) &= 
	|\{v\in V(\YY) \setminus V(T) : \bsig_{\delta v} = \bsig \}|, \quad \textnormal{for all }\bsig \in (\dot{\Omega}_L^2)^d; \\
	\hat{\Delta}_\partial^{\uL} ( \bsig) &= 
		|\{a\in F(\YY) \setminus F(T) : \bsig_{\delta a} = \bsig, \uL_{\delta a} = \uL \}|, \quad \textnormal{for all }\bsig \in (\dot{\Omega}_L^2)^k, \; \uL \in \{0,1\}^k.
	\end{split}
	\end{equation*}
	Note that $\dot{\Delta}_\partial$ carries the information on the colors on $U$, while $\bar{\Delta}$ does not (and hence we use different notations).
	 Lastly, let $\UUU' \equiv \UUU\setminus U$, and for a given coloring configuration $\bsig_{\UUU'}$ on $\UUU'$, define $\bar{h}^{\bsig_{\UUU'}}$ to be the following counting measure on $\dot{\Omega}_L^2$:
	\begin{equation*}
	\bar{h}^{\bsig_{\UUU'}} (\bsigma) 
	=
	|\{e\in \UUU' : \bsigma_e = \bsigma \}|, \quad \textnormal{for all } \bsigma \in \dot{\Omega}_L^2.
	\end{equation*}
	 Then, the next lemma provides a refined estimate on \eqref{eq:ZXcycbyindicators}, which can be thought as a planted-cycles analog of Lemma 6.7, \cite{dss16}.
	
	\begin{lemma}\label{lem:Zplantedexpansion}
		Let $\YY, \bsig_\YY$ be given as above. For any given $\underline{a}$ with $||\underline{a}||_\infty \leq \log^2 n$ and for all $\bsig_{\UUU'}$, we have
		\begin{equation}\label{eq:Zplantedbulkexpansion}
		\begin{split}
		&\E_T \left[\prescript{}{2}{\bZ}^\partial(\bsig_\UUU;\Gamma_2^\bullet ) \one\{\YY, \bsig_\YY \} \right]\\
		&=
		c_0 \left(1+ O\left(\frac{||\underline{a}||_1^2}{n} \right) \right) \E[(\bZ')^2] \;\P_T(\YY)\; \beta_T(\YY,\Delta)  \prod_{e\in \UUU'} \dot{q}^\star_{\lambda,L} (\bsigma_e) \\
		&\quad \times 
		\left\{
		1+ b(\bsig_\YY) + \langle \bar{h}^{\bsig_{\UUU'}},\xi_0 \rangle
		+
		\sum_{j=1}^{C_{k,L}} \langle \bar{h}^{\bsig_{\UUU'}}, \xi_j \rangle^2 + O \left(\frac{\log^{12} n}{n^{3/2}} \right)
		\right\},
		\end{split}
		\end{equation}
		where the terms in the identity can be explained as follows.
		\begin{enumerate}
			\item $c_0>0$ is a constant depending only on $|\UUU|$.
			
			\item $b (\bsig_{\YY})$ is a quantity such that $|\epsilon (\bsig_{\YY})| = O(n^{-1/2} \log^2n)$, independent of $\bsig_{\UUU'}$.
			
			\item $C_{k,L}>0$ is an integer depending only on $k,L$, and  $\xi_j = (\xi_j (\tau))_{\tau \in \dot{\Omega}_L^2}$, $0\leq j\leq C_{k,L}$ are fixed vectors on $\dot{\Omega}_L^2$ satisfying $$ ||\xi_j||_\infty = O(n^{-1/2}). $$ 
			
			\item $\P_T(\YY)$ is the conditional probability given the structure $T$ such that the prescribed half-edges of $\YY$ are all paired together and assigned with the right literals.
			
			\item Write $\dot{H}\equiv\dot{H}^\star_{\lambda,L} $, and similarly for $\hat{H}^{\uL}$, $\bar{H}$. The function $\beta_T( \YY, \Delta)$ is defined as
			\begin{equation*}
			\beta_T(\YY,\Delta ) \equiv \frac{\dot{H}^{\dot{\Delta}_\partial} \prod_{\uL} (\hat{H}^{\uL})^{\hat{\Delta}_\partial^{\uL}} }{ \bar{H}^{\bar{\Delta} + \bar{\Delta}_U }} \times \prod_{e\in U} \dot{q}^\star_{\lambda,L} (\bsigma_e).
			\end{equation*}

		\end{enumerate}
	\end{lemma}
	
	The proof goes similarly as that of Proposition \ref{prop:cycleeffect:moments}, but requires extra care due to the complications caused by the (possible) intersection between $\YY$ and $T$. Due to its technicality, we defer the proof to Section \ref{subsec:app:Zplantedexp} in the appendix.
	
	Based on the expansion obtained from Lemma \ref{lem:Zplantedexpansion}, we conclude the proof of Lemma \ref{lem:ZTdiff:poly}.
	
	\begin{proof}[Proof of Lemma \ref{lem:ZTdiff:poly}]
		We work with fixed $\YY, \bsig_\YY$ as in Lemma \ref{lem:Zplantedexpansion}. For $\underline{r}=(\underline{r}^1,\underline{r}^2)$, define the Fourier coefficient of \eqref{eq:ZXcycbyindicators} as
		\begin{equation}\label{eq:def:ZplantedFourier}
		\prescript{}{2}{\F}^\wedge_{T} (\underline{r}\; ; \YY,\bsig_\YY) \equiv 
		\sum_{\bsig_{\UUU'}} 
		  \E_T \left[\prescript{}{2}{\bZ}^\partial(\bsig_\UUU;\Gamma_2^\bullet ) \one\{\YY, \bsig_\YY \} \right] \onb_{\underline{r}} (\bsig_{\UUU'}) .
		\end{equation}
		We compare this with the Fourier coefficients 
		\begin{equation}\label{eq:def:F2Fourier2}
		\prescript{}{2}{\F}^\wedge_T(\underline{r}) = \sum_{\bsig_{\UUU'}} \E_T \left[\prescript{}{2}{\bZ}^\partial(\bsig_\UUU;\Gamma_2^\bullet ) \right]\onb_{\underline{r}} (\bsig_{\UUU'}), 
		\end{equation}
		of which we already saw the estimates in \eqref{eq:F2fourierbd:nonrescaled}. In addition, it will be crucial to understand the expansion of $\E_T \left[\prescript{}{2}{\bZ}^\partial(\bsig_\UUU;\Gamma_2^\bullet ) \right]$ as in Lemma \ref{lem:Zplantedexpansion}. This was already done in Lemma 6.7 of \cite{dss16} and we record the result as follows.
		
		\begin{lemma}[Lemma 6.7, \cite{dss16}]\label{lem:ZexpansionDSS}
			There exist a constant $C_{k,L}'>0$ and coefficients $\xi_j'\equiv (\xi_j'(\bsigma))_{\bsigma \in \dot{\Omega}_L^2}$ indexed by $0\leq j\leq C_{k,L}'$, such that $||\xi_j'||_\infty = O(n^{-1/2})$ and
			\begin{equation}\label{eq:ZexpansionDSS}
			\frac{ \E_T \left[\prescript{}{2}{\bZ}^\partial(\bsigma_\UUU;\Gamma_2^\bullet ) \right] \cdot c_0 }{ \bq(\bsig_{\UUU} ) \E[(\bZ')^2] }
			=
			1+ \langle \bar{h}^{\bsig_{\UUU'}}, \xi_0' \rangle  + \sum_{j=1}^{C_{k,L}'} \langle \bar{h}^{\bsig_{\UUU'}}, \xi_j'  \rangle^2
			+
			O\left(\frac{\log^{12}n }{n^{3/2}} \right),
			\end{equation} 
			where $c_0$ is the constant appearing in Lemma \ref{lem:Zplantedexpansion}. Moreover, $C_{k,L}'$ and the coefficients $\xi_j'$, $1\leq j \leq C_{k,L}'$ can be set to be the same as $C_{k,L}$ and $\xi_j$ in Lemma \ref{lem:Zplantedexpansion}.
		\end{lemma}
		
		The identity \eqref{eq:ZexpansionDSS} follows directly from Lemma 6.7, \cite{dss16}, and the last statement turns out to be apparent from the proof of Lemma \ref{lem:Zplantedexpansion} (see Section \ref{subsec:app:Zplantedexp}).
		
		Based on Lemma \ref{lem:Zplantedexpansion}, we obtain the following bound on the Fourier coefficient \eqref{eq:def:ZplantedFourier}:
		\begin{equation}\label{eq:ZplantedFourier non2}
		\left|\prescript{}{2}{\F}^\wedge_T (\underline{r}\;;\YY,\bsig_\YY) \right| \lesssim_{k,L}
		\E [(\bZ')^2]  \;	\P_T (\YY) \; \beta_T(\YY,\Delta)\times
		\begin{cases}
	 1& \textnormal{if } |\{\underline{r}\}| =0;\\
		 n^{-1/2}	& \textnormal{if } |\{\underline{r}\}| \geq1;\\
	\frac{\log^{12}n}{n^{3/2}}   & \textnormal{if } |\{\underline{r}\}| \geq 3.
		\end{cases}
		\end{equation}
		Moreover, suppose that $U= \emptyset$, that is, $\YY$ does not intersect with $\UUU$. In this case, we can compare \eqref{eq:def:ZplantedFourier} and \eqref{eq:def:F2Fourier2} in the following way, based on Lemmas \ref{lem:Zplantedexpansion} and \ref{lem:ZexpansionDSS}:
		\begin{equation}\label{eq:ZplantedFourier 2}
		\prescript{}{2}{\F}^\wedge_T (\underline{r}\;;\YY,\bsig_\YY)
		= \P_T (\YY) \; \beta_T(\YY,\Delta) 
		\left(\prescript{}{2}{\F}^\wedge_T(\underline{r})  +O\left(\frac{\log^{12}n}{n^{3/2}} \right)\E[(\bZ')^2] \right), \quad \textnormal{if } |\{\underline{r} \}|=2.
		\end{equation}
		
		Using these observations, we investigate the following formula which can be deduced from \eqref{eq:ZTincre decomposition} by Plancherel's identity:
		\begin{equation}\label{eq:ZdiffplantedFourierdecom}
		\E_T \left[\left(\bZ_T -\acute{\bZ}_T \right)^2 \one\{\YY, \bsig_\YY \} \right]
		=
		\sum_{\underline{r}=(\underline{r}^1,\underline{r}^2)}
		\varpi^\wedge(\underline{r}^1;\sig_{\YY}^1) \varpi^\wedge(\underline{r}^2;\sig_{\YY}^2) 
		\prescript{}{2}{\F}^\wedge_T(\underline{r}\;; \YY,\bsig_\YY),
		\end{equation}
		where the Fourier coefficients of $\varpi$ are given by
		\begin{equation*}
		\varpi^\wedge(\underline{r}^1;\sig_{\YY}^1)
		\equiv
		\sum_{\sig_{\UUU'}^1} \varpi (\sig_{\UUU'}^1; \sig^1_\YY ) \,\onb_{\underline{r}^1} (\sig_{\UUU'}^1)\,\bq (\sig_{\UUU'}^1) .
		\end{equation*}
		
		Define $\eta (\YY)\equiv\eta (\YY;T)\equiv |\bar{\Delta}|+|U|-|\dot{\Delta}_\partial |- |\hat{\Delta}_\partial |$, similarly as \eqref{eq:def:eta}. As before, note that the quantities $|\bar{\Delta}|, |U|, |\dot{\Delta}_\partial |,$ and $ |\hat{\Delta}_\partial |$ are all well-defined if $T$ and $\YY$ are given. Observe that 
		\begin{equation*}
		\#\{ \textnormal{connected components in }\YY \textnormal{ disjoint with } \UUU \} = ||\underline{a}||_1 -\eta(\YY).
		\end{equation*}
		The remaining work is done by a case analysis with respect to $\eta(\YY)$.
		
		\vspace{2mm}
		\noindent \textbf{Case 1.} $\eta(\YY)=0$.
		\vspace{2mm}
		
		In this case, all cycles in $\YY$ are not only pairwise disjoint, but also disjoint with $\UUU$. As we will see below, such $\YY$ gives the most contribution to \eqref{eq:ZdiffplantedFourierdecom}. Recall the events $\bT$, $\bC^\circ$, $\bC_t$ and $\bB$ defined in the beginning of Section \ref{subsubsec:whp:localnbd} and in \eqref{eq:def:Bevent}. 
		
		On the event $\bT^c = \cup_{t\leq l_0} \bC_t \cup \bC^\circ \cup \bB$, we can apply the same approach as in the proof of Lemma \ref{lem:increbd:ZT} using \eqref{eq:ZplantedFourier non2} and obtain that
		\begin{equation*}
			\E \left[\left(\bZ_T -\acute{\bZ}_T \right)^2 \one\{\YY, \bsig_\YY \} \;; \bT^c\right]
			=
			O\left(\frac{\log n}{n^{3/2}} \right)	\E [(\bZ')^2]  \;	\P (\YY|\bT^c) \; \beta_T(\YY,\Delta).
		\end{equation*}
		On the other hand, on $\bT$, $\varpi^\wedge(\underline{r}^1) = 0 $ for $|\{\underline{r}^1 \} |\leq 1$ and hence the most contribution comes from $|\{\underline{r} \} | =2$. To control this quantity, we use the estimate \eqref{eq:ZplantedFourier 2} and get
		\begin{equation*}
		\begin{split}
			&\E \left[\left(\bZ_T -\acute{\bZ}_T \right)^2 \one\{\YY, \bsig_\YY \} \;; \bT\right]\\
			&=
			\P(\bT)\; \P (\YY|\bT) \; \beta_T(\YY,\Delta)
			\left( 	\E_T \left[\left(\bZ_T -\acute{\bZ}_T \right)^2\right] +
			O\left(\frac{\log^{12}n}{n^{3/2}} \right)\E[(\bZ')^2]
			 \right).
		\end{split}
		\end{equation*} 
		If we sum over all $\bsig_\YY$, and then over all $\YY$ such that $\eta(\YY)=0$, we obtain by following the same computations as \eqref{eq:ZsplitbyYtau1}--\eqref{eq:cycleef:trun:smallg} that
		\begin{equation}\label{eq:ZTdiffplanted1}
			\begin{split}
			&\sum_{\YY: \eta(\YY)=0} \sum_{\bsig_\YY} \E \left[\left(\bZ_T -\acute{\bZ}_T \right)^2 \one\{\YY, \bsig_\YY \} \right]\\
			&=
			\left(1+ O\left(\frac{||\underline{a}||_1^2}{n} \right) \right) \left(\underline{\mu} ( 1+\underline{\delta}_L)^2 \right)^{\underline{a}}
			\left( 	\E_T \left[\left(\bZ_T -\acute{\bZ}_T \right)^2\right] +
			O\left(\frac{\log^{12}n}{n^{3/2}} \right)\E[(\bZ')^2]
			\right).
			\end{split}
		\end{equation}
		
		\vspace{2mm}
		\noindent \textbf{Case 2.} $\eta(\YY)=1$.
		\vspace{2mm}

		One important observation we make here is that if $T\in \bT$ $\eta(\YY)=1$, then for any $\bsig_\YY= (\sig_\YY^1,\sig_\YY^2)$, we have 
		\begin{equation*}
		\kappa^\wedge (\varnothing; \sig_\YY^1) = \acute{\kappa}^\wedge (\varnothing; \sig_\YY^1),
		\end{equation*}
		and analogously for the second copy $\sig_\YY^2$. If we had $|U|\leq 1$, then this is a direct consequence of the results mentioned in the beginning of Section \ref{subsubsec:whp:fourier}. 
		
		On the other hand, suppose that $|U|=2$. If we want to have $\eta(\YY)=1$, then the only choice of $\YY$ is that there exists one cycle in $\YY$ that intersects with $\UUU$ at two distinct half-edges, while all others in $\YY$ are disjoint from each other and from $\UUU$. In such a case, since the lenghs of cycles in $\YY$ are all at most $2l_0$, the cycle intersecting with $\UUU$ cannot intersect with $A$ (or $\acute{A}$).  Therefore, the two half-edges $U$ are contained in the same tree of $T$, and hence by symmetry the $\varnothing$-th Fourier coefficient does not depend on $A$ (or $\acute{A}$).
		
		With this in mind, the $\varnothing$-th Fourier coefficient does not contribute to \eqref{eq:ZdiffplantedFourierdecom}, and hence we get
		\begin{equation*}
			\E \left[\left(\bZ_T -\acute{\bZ}_T \right)^2 \one\{\YY, \bsig_\YY \} \;; \bT\right]
			=
			O\left(n^{-1/2}\right)	\E [(\bZ')^2]  \;	\P (\YY| \bT ) \; \beta_T(\YY,\Delta),
		\end{equation*} 
		where $\Delta = \Delta[\bsig_\YY]$.
		
		On the event $\bT^c$, we can bound it coarsely by
		\begin{equation*}
		\begin{split}
		\E \left[\left(\bZ_T -\acute{\bZ}_T \right)^2 \one\{\YY, \bsig_\YY \} \;; \bT^c\right]
		&\lesssim_{k,L}
		\P(\bT^c) \,	\E [(\bZ')^2]  \;	\P_T (\YY|\bT^c) \; \beta_T(\YY,\Delta)\\
		&=O\left(\frac{\log n}{n} \right) 	\E [(\bZ')^2]  \;	\P_T (\YY) \; \beta_T(\YY,\Delta).
		\end{split}
		\end{equation*}
		
		What remains is to sum the above two over $\bsig_\YY$ and $\YY$ such that $\eta({\YY})=1$. Since there can be at most $2$ cycles from $\YY$ that are not disjoint from all the rest, there exists a constant $C=C_{k,L,l_0}$ such that 
		\begin{equation}\label{eq:sumbetabd}
		\sum_{\bsig_\YY} \beta_T (\YY,\Delta) \leq
		\left(1+\underline{\delta}_L \right)^{2\underline{a}} C^2.
		\end{equation}
		(see \eqref{eq:cycleef:trun:Zg bad YY}) Then, we can bound the number choices of $\YY$ as done in \eqref{eq:choice of eta YY} and \eqref{eq:ZsumoverYY}. This gives that
		\begin{equation}\label{eq:ZTdiffplanted2}
		\sum_{\YY: \eta(\YY)=1} \sum_{\bsig_\YY}	\E \left[\left(\bZ_T -\acute{\bZ}_T \right)^2 \one\{\YY, \bsig_\YY \}  \right]
		=
		O\left( \frac{C^2 l_0 ||\underline{a}||_1 }{n^{3/2}} \right) \left(\underline{\mu} (1+\underline{\delta}_L)^2  \right)^{\underline{a}} \E [(\bZ')^2].
		\end{equation}

		\vspace{2mm}
		\noindent \textbf{Case 3.} $\eta(\YY)\geq 2$.
		\vspace{2mm}
		
		In this case, we deduce the conclusion relatively straightforwardly since $\sum_{\YY} \P_T (\YY)$ is too small. Namely, we first have the crude bound from \eqref{eq:ZplantedFourier non2} such that
		\begin{equation*}
		\E \left[\left(\bZ_T -\acute{\bZ}_T \right)^2 \one\{\YY, \bsig_\YY \} \right] = O(1) \E [(\bZ')^2]  \;	\P_T (\YY) \; \beta_T(\YY,\Delta).
		\end{equation*}
		From similar observations as in \eqref{eq:sumbetabd}, we can obtain that
		\begin{equation*}
		\sum_{\underline{\tau}_\YY}\beta_T (\YY,\Delta) \leq \left( 1+\underline{\delta}_L \right)^{2\underline{a} } C^{2\eta},
		\end{equation*}
		where $C$ is as in \eqref{eq:sumbetabd}. Further, we control the number of choices of $\YY$ as before, which gives that
		\begin{equation}\label{eq:ZTdiffplanted3}
			\sum_{\YY: \eta(\YY)=\eta} \sum_{\bsig_\YY} 	\E \left[\left(\bZ_T -\acute{\bZ}_T \right)^2 \one\{\YY, \bsig_\YY \}  \right]
			=
			O\left( \left( \frac{C^2 l_0 ||\underline{a}||_1 }{n}\right)^\eta \right) \left(\underline{\mu} (1+\underline{\delta}_L)^2  \right)^{\underline{a}} \E [(\bZ')^2].
		\end{equation}
		
		Combining \eqref{eq:ZTdiffplanted1}, \eqref{eq:ZTdiffplanted2} and \eqref{eq:ZTdiffplanted3}, we obtain the conclusion.		
	\end{proof}
	
	Having Lemma \ref{lem:ZTdiff:poly} in hand, we are now ready to finish the proof of Lemma \ref{lem:YvsZT}.
	
	\begin{proof}[Proof of Lemma \ref{lem:YvsZT}]
		Set $\tilde{{\delta}}(\zeta) = (1+\delta_L(\zeta))^{-2}-1$. Using the identity $(1+\theta)^x = \sum_{a\geq 0} \frac{(x)_a}{a!} \theta^a$ (which holds for all nonnegative integer $x$), we can express that
		\begin{equation*}
		\begin{split}
			&\E \left[\left(\bZ_T -\acute{\bZ}_T \right)^2 (1+\underline{\delta}_L)^{-2\underline{X}^\partial} \one\{||\underline{X}^\partial ||_\infty \leq \log n \} \right]\\
			&=
			\sum_{||\underline{a}||_\infty \leq \log n} \frac{1}{\underline{a}!} \E\left[  
			\left(\bZ_T -\acute{\bZ}_T \right)^2 \tilde{\underline{\delta}}^{\underline{a}} (\underline{X}^\partial)_{\underline{a}}
			\right] 
			+
			n^{-\Omega(\log\log n)} \E[ (\bZ')^2],
		\end{split}
		\end{equation*}
		where we used Corollary \ref{cor:Ymanycyc} to obtain the error term in the RHS. Also note that $(\underline{X}^\partial)_{\underline{a}}=0$ if $||\underline{a}||_\infty>\log n$ and $||\underline{X}^\partial||_\infty \leq \log n$. 
		Therefore, by applying Lemma \ref{lem:ZTdiff:poly}, we see that the above is the same as
		\begin{equation*}
		\left(1+ O\left(\frac{\log^2 n}{n} \right) \right) \E \left[\left(\bZ_T -\acute{\bZ}_T \right)^2\right] \sum_{||\underline{a}||_\infty \leq \log n} \frac{1}{\underline{a}!} \left(
		\tilde{\underline{\delta}} \underline{\mu} (1+\underline{\delta}_L)^2 \right)^{\underline{a}} 
		+
		O\left(\frac{\log^{12}n }{n^{3/2}} \right) \E [(\bZ')^2], 
		\end{equation*}
		and from here we can directly deduce conclusion from performing the summation.
	\end{proof}

	\section{Small subgraph conditioning and the proof of Theorem \ref{thm:lowerbd}}\label{subsec:whp:smallsubcon}
	
	In this section, we prove Theorem \ref{thm:lowerbd} by small subgraph conditioning method. To do so, we derive the condition (d) of Theorem \ref{thm:smallsubcon} first for the truncated model, and then deduce the analogue for the untruncated model based on the continuity of the coefficients, which was proved in \cite[Lemma 4.17]{nss20a}.
	
	\begin{prop}\label{prop:Zmomentratio:trun}
		Let $L>0$ and $\lambda\in (0,\lambda^\star_L)$ be given. Moreover, set $\mu(\zeta), \delta_L(\zeta;\lambda)$ as in Proposition \ref{prop:cycleeffect:moments}. Recalling $\bwZ_{\la}^{(L),\tr}$ defined in \eqref{eq:def:rescaledPF}, we have
		\begin{equation}\label{eq:momentratio:trun:inlemma}
		\lim_{n\to\infty}\frac{\E\big(\bwZ^{(L),\tr}_{\la}\big)^{2} }{\big(\E\bwZ_{\lambda}^{(L),\tr}\big)^2 }=
		\exp \bigg( \sum_{\zeta} \mu(\zeta) \delta_L(\zeta;\lambda)^2 \bigg).
		\end{equation}
	\end{prop}
	
	\begin{proof}
	    Fix $\la<\la^\star_L$ and abbreviate $\delta_L(\zeta)\equiv\delta_L(\zeta;\la)$ as before. We first show that the \textsc{lhs} is lower bounded by the \textsc{rhs} in \eqref{eq:momentratio:trun:inlemma}. Let $\underline{X} = (X(\zeta))_\zeta$ be the number of $\zeta$-cycles in $\GGG$. For an integer $l_0>0$, we write $\underline{X}_{\leq l_0}=(X(\zeta))_{||\zeta||\leq l_0}$ (note the difference from the notations used in the previous subsections). Note that Proposition \ref{prop:cycleeffect:moments}-(1) gives us that the limiting law of $\underline{X}_{\le l_0}$ reweighted by $\bwZ^{(L),\tr}_{\lambda}$ must be independent Pois$\big(\mu(\zeta)(1+\delta_L(\zeta)\big)$, since the moments of falling factorials are given by \eqref{eq:cycmo:trun:1st}. Namely, for a given collection of integers $\ux_{\le l_0} = (x(\zeta))_{||\zeta|| \le l_0}$, we have
		\begin{equation*}
		\lim_{n\to\infty} \frac{\E\big[ \bwZ^{(L),\tr}_{\lambda} \mathds{1} \{\underline{X}_{\le l_0} = \ux_{\le l_0} \}\big] }{ \E \bwZ^{(L),\tr}_{\lambda} } = \prod_{||\zeta|| \le l_0 } \P \bigg( \textnormal{Pois}\Big(\mu(\zeta)\big(1+\delta_L(\zeta)\big)\Big) = x(\zeta) \bigg).
		\end{equation*}
		Recall that the unweighted $\underline{X}_{\le l_0}$ has the limiting law given by \eqref{eq:cyclejointasymp}. Thus, we have
		\begin{equation}\label{eq:momentlim:cycconditioning}
		\lim_{n\to\infty} \frac{\E \big[ \bwZ^{(L),\tr}_{\lambda} \, \big| \, \underline{X}_{\le l_0} = \ux_{\le l_0} \big]}{ \E \bwZ^{(L),\tr}_{\lambda} } = \prod_{||\zeta||\leq l_0 } (1+\delta_L(\zeta))^{x(\zeta)} e^{-\mu(\zeta)\delta(\zeta)},
		\end{equation}
		for any $\ux_{\le l_0} = (x(\zeta))_{||\zeta|| \le l_0}$. Thus, by Fatou's Lemma, we have
		\begin{equation}\label{eq:vardecomp}
			\liminf_{n\to\infty} \frac{\E\left[\E \big[\bwZ^{(L),\tr}_{\lambda} \,\big|\, \underline{X}_{\leq l_0}\big]^{2} \right]}{
				\big(\E \bwZ^{(L),\tr}_{\lambda}\big)^2
				}\geq \exp \bigg( \sum_{||\zeta||\leq l_0} \mu(\zeta) \delta_L(\zeta)^2 \bigg) 
		\end{equation}
		Since this holds for any $l_0$, we obtain the lower bound of \textsc{lhs} of \eqref{eq:momentratio:trun:inlemma}.
		
		To work with the upper bound, recall the definition of the rescaled partition function $\bY_{l_0} \equiv \bY_{\lambda,l_0}^{(L)}$ in \eqref{eq:def:rescaledPF}. For any $\eps>0$, Proposition \ref{prop:concenofrescaled} implies that there exists $l_0(\eps)>0$ such that for  $l_0\geq l_0(\eps)$,
		\begin{equation}\label{eq:Ymomentratio eps}
		\lim_{n\to\infty} \frac{\E \bY_{l_0}^2}{ (\E \bY_{l_0})^2 } \leq 1+\eps.
		\end{equation}
		
		On the other hand, we make the following observation which are the consequences of Corollaries \ref{cor:toomanycyc:Zg}, \ref{cor:Ymanycyc} and Proposition \ref{prop:cycleeffect:moments}:
		\begin{equation}\label{eq:YvsZexponents}
		\begin{split}
		\E \bY_{l_0} &= (1+o(1)) \E \bwZ^{(L),\tr}_{\lambda}  \exp \bigg(-\sum_{||\zeta||\leq l_0} \mu(\zeta) \delta_L(\zeta) \bigg)\,,\\
		\E \bY_{l_0}^2 &= (1+o(1)) \E \big(\bwZ^{(L),\tr}_{\lambda}\big)^2  \exp \bigg(-\sum_{||\zeta||\leq l_0} \mu(\zeta) \left(2\delta_L(\zeta)+\delta_L(\zeta)^2 \right) \bigg)\,.
		\end{split}
		\end{equation}
		We briefly explain how to obtain \eqref{eq:YvsZexponents}. First, note that it suffices to estimate $\E[\bY_{l_0} \one_{\{||\underline{X} ||_\infty \leq \log n \} }] $ due to Corollary \ref{cor:Ymanycyc}. Then, we expand the rescaling factor of $\bY_{l_0}$ by falling factorials using the formula \eqref{eq:expbyfallingfac}. Each correlation term $\E [\bZ^{(L),\tr}_{\lambda} (\underline{X})_{\underline{a}}\one_{\{||\underline{X} ||_\infty \leq \log n \} }]$ can then be studied based on Proposition \ref{prop:cycleeffect:moments} and Corollary \ref{cor:toomanycyc:Zg}. We can investigate the second moment of $\bY_{l_0}$ analogously.
		
		Combining \eqref{eq:Ymomentratio eps} and \eqref{eq:YvsZexponents} shows
		\begin{equation*}
		\lim_{n\to\infty} \frac{\E \big(\bwZ^{(L),\tr}_{\lambda}\big)^2 }{\big(\E \bZ^{(L),\tr}_{\lambda}\big)^2 } \leq (1+\eps) \exp \left(\sum_{||\zeta||\leq l_0} \mu(\zeta) \delta_L(\zeta)^2 \right),
		\end{equation*}
		which holds for all $l_0 \geq l(\eps)$ and $\eps>0$. Therefore, letting $l_0\to\infty$ and $\eps \to 0$ gives the conclusion.
	\end{proof}
	
	The next step is to deduce the analogue of Proposition \ref{prop:Zmomentratio:trun} for the untruncated model. To do so, we first review the following notions from \cite{nss20a}: for coloring configurations $\sig^{1},\sig^{2}\in \Omega^{E}$, let $x^{1}\in \{0,1,\ff\}^{V}$ (resp. $x^{2}\in \{0,1,\ff\}^{V}$) be the frozen configuration corresponding to $\sig^{1}$ (resp. $\sig^{2}$) via Lemma \ref{lem:model:bij:frozen and msg} and \eqref{eq:msg col 1to1}. Then, define the \textit{overlap} $\rho(\sig^{1},\sig^{2})$ of $\sig^1$ and $\sig^2$ by
	\begin{equation*}
	    \rho(\sig^{1},\sig^{2}):=\frac{1}{n}\sum_{i=1}^{n}\one\{x^{1}_i\neq x^{2}_i\}.
	\end{equation*}
	Then, for $\la\in [0,1]$ and $s\in [0,\log 2)$, denote the contribution to $(\bZ^{\tr}_{\la,s})^{2}$ from the \textit{near-independence} regime $|\rho(\sig^{1},\sig^{2})-\frac{1}{2}|\leq k^{2}2^{-k/2}$ by
	\begin{equation*}
	    \bZ^{2}_{\la,s,\ind}:=\sum_{\sig^{1},\sig^{2}\in \Omega^{E}}w_{\GGG}^{\lit}(\sig^{1})^{\la} w_{\GGG}^{\lit}(\sig^{2})^{\la}\one\Big\{w_{\GGG}^{\lit}(\sig^{1}),w_{\GGG}^{\lit}(\sig^{2})\in [e^{ns},e^{ns+1}), ~~\Big|\rho(\sig^1,\sig^2)-\frac{1}{2}\Big|< \frac{k^2}{2^{k/2}}\Big\}.
	\end{equation*}
	Similarly, we respectively denote $\bZ^{2}_{\la,\ind}$, $\bZ^{2,(L)}_{\la,\ind}$ and $\bZ^{2,(L)}_{\la,s,\ind}$ by the contribution to $(\bZ^{\tr}_{\la})^{2}$, $(\bZ^{(L),\tr}_{\la})^2$ and $(\bZ^{(L),\tr}_{\la,s})^{2}$ from the near-independence regime $|\rho(\sig^{1},\sig^{2})-\frac{1}{2}|\leq k^{2}2^{-k/2}$.

	\begin{prop}\label{prop:Zmomentratio:untrun}
		Let $\mu(\zeta)$ and $\delta(\zeta;\la^\star)$ be the constants from Proposition \ref{prop:cycleeffect:moments}. Then, for $(s_n)_{n \geq 1}$ converging to $s^\star$ with $|s_n-s^\star|\leq n^{-2/3}$, we have
		\begin{equation*}
		\lim_{n\to\infty}\frac{\E\bZ^{2}_{\la^\star,s_n,\ind}}{\big(\E\bZ^{\tr}_{\la^\star,s_n}\big)^{2}} = \exp \bigg( \sum_{\zeta} \mu(\zeta) \delta(\zeta;\la^\star)^2 \bigg).
		\end{equation*}
	\end{prop}
	
	\begin{proof}
	Note that for $\la<\la^\star_L$, Proposition 4.20 of \cite{nss20a} shows that the contribution to $\E\big( \bwZ_{\la}^{(L),\tr}\big)^{2}$ from the correlated regime $|\rho(\sig^{1},\sig^{2})-\frac{1}{2}|\geq k^{2}2^{-k/2}$ is negligible compared to the near-independence regime. Also, since $\bwZ^{(L),\tr}_{\la}$ is defined to be the contribution to $\bZ^{(L),\tr}_{\la}$ from $||H-H^\star_{\la,L}||_1 \leq n^{-1/2}\log^{2}n$, Proposition 3.10 of \cite{ssz16} shows that the contribution to $\E\big( \bwZ_{\la}^{(L),\tr}\big)^{2}$ from the near-independence regime is $\big(1-o(1)\big)\E\bZ^{2,(L)}_{\la,\ind}$. Similarly, Proposition 3.4 of \cite{ssz16} shows $\E \bwZ_{\la}^{(L),\tr}= \big(1-o(1)\big)\E \bZ_{\la}^{(L),\tr}$. Therefore, for $\la<\la^\star_L$, we have
	\begin{equation}\label{eq:prop:Zmomentratio:untrun:1}
	\lim_{n\to\infty}\frac{\E\bZ^{2,(L)}_{\la,\ind} }{\big(\E\bZ_{\lambda}^{(L),\tr}\big)^2 }=	\lim_{n\to\infty}\frac{\E\big(\bwZ^{(L),\tr}_{\la}\big)^{2} }{\big(\E\bwZ_{\lambda}^{(L),\tr}\big)^2 }=
		\exp \bigg( \sum_{\zeta} \mu(\zeta) \delta_L(\zeta;\lambda)^2 \bigg),
	\end{equation}
	where the last inequality is from Proposition \ref{prop:Zmomentratio:trun}. By Theorem 3.21, Proposition 4.15 and Proposition 4.18 in \cite{nss20a}, we can send $L\to\infty$ and $\la\nearrow \la^\star$ to have
		\begin{equation}\label{eq:prop:Zmomentratio:untrun:2}
		\lim_{n\to\infty}\frac{\E\bZ^2_{\la^\star,\ind}}{ \big(\E \bZ^\tr_{\lambda^\star}\big)^2 } = \lim_{\la\nearrow\la^\star}\lim_{L\to\infty}
			\lim_{n\to\infty}\frac{\E\bZ^{2,(L)}_{\la,\ind} }{\big(\E\bZ_{\lambda}^{(L),\tr}\big)^2 }=
			\exp \left( \sum_{\zeta} \mu(\zeta) \delta(\zeta;\lambda^\star)^2 \right),
		\end{equation}
		where in the last inequality, we used \eqref{eq:prop:Zmomentratio:untrun:1} and Proposition \ref{prop:cycleeffect:moments}-(4). Finally, Lemma 4.17 and Proposition 4.19 of \cite{nss20a} shows the \textsc{lhs} of the equation above equals $\lim_{n\to\infty}\frac{\E\bZ^{2}_{\la^\star,s_n,\ind}}{\big(\E\bZ^{\tr}_{\la^\star,s_n}\big)^{2}}$, so \eqref{eq:prop:Zmomentratio:untrun:2} concludes the proof.
	\end{proof}
	\begin{cor}\label{cor:conditional:var:negligible}
	Let $\underline{X}_{\leq l_0}=(X(\zeta))_{||\zeta||\leq l_0}$ be collection of the number of $\zeta$-cycles in $\GGG$ with size $||\zeta||\leq l_0$. Denote $s_\circ(C)\equiv s^\star-\frac{\log n}{2\la^\star n}-\frac{C}{n}$. Recalling the definition of $\bwZ^{\tr}_{\la,s}$ from \eqref{eq:def:tilde:Z:untruncated}, we have
	\begin{equation}
	    \lim_{C\to\infty} \limsup_{l_0\to\infty}\limsup_{n\to\infty} \frac{\E\Big[\Var\big(\bwZ^{\tr}_{\la^\star,s_{\circ}(C)} \,\big|\,\underline{X}_{\leq l_0}\big)\Big]}{\big(\E \bwZ^{\tr}_{\la^\star, s_{\circ}(C)}\big)^{2}}=0.
	\end{equation}
	\end{cor}
	\begin{proof}
	Proceeding in the same fashion as \eqref{eq:vardecomp} in the proof of Proposition \ref{prop:Zmomentratio:trun}, Proposition \ref{prop:cycleeffect:moments}-(3) shows 
	\begin{equation}\label{eq:cor:conditional:var:negligible:lower}
	    \liminf_{n\to\infty}\frac{\E\left[\E \big[\bwZ^{\tr}_{\lambda^\star,s_\circ(C)} \,\big|\, \underline{X}_{\leq l_0}\big]^{2} \right]}{
				\big(\E \bwZ^{\tr}_{\lambda^\star,s_\circ(C)}\big)^2
				}\geq \exp \bigg( \sum_{||\zeta||\leq l_0} \mu(\zeta) \delta(\zeta;\la^\star)^2 \bigg).
	\end{equation}
	To this end, we aim to find a matching upper bound for $\frac{\E\big(\bwZ^{\tr}_{\la^\star,s_{\circ}(C)}\big)^{2}}{\big(\E\bwZ^{\tr}_{\la^\star,s_{\circ}(C)}\big)^{2}}$. Note that Proposition 4.20 of \cite{nss20a} shows that the contribution to $\E\big(\bwZ^{\tr}_{\la^\star,s_{\circ}(C)}\big)^{2}$ from the correlated regime $|\rho(\sig^{1},\sig^{2})-\frac{1}{2}|\geq k^{2}2^{-k/2}$ is bounded above by $\lesssim_{k}e^{2n\la^\star s_{\circ}(C)}\E\bN_{s_{\circ}(C)}+e^{-\Omega_{k}(n)}$. Thus, we have
	\begin{equation}\label{eq:bound:tilde:Z:squared}
	    \E\big(\bwZ^{\tr}_{\la^\star,s_{\circ}(C)}\big)^{2}\leq C_k e^{2n\la^\star s_{\circ}(C)}e^{\la^\star C}+\E\bZ^{2}_{\la^\star,s_\circ(C),\ind}\leq \Big(1+C_k^\prime e^{-\la^\star C}\Big)\E\bZ^{2}_{\la^\star,s_\circ(C),\ind},
	\end{equation}
	where the $C_k$ and $C_k^\prime$ are constants which depend only on $k$, and the last inequality holds because of Proposition 4.16 in \cite{nss20a}. Moreover, by Proposition 3.17 in \cite{nss20a}, $\E\bwZ^{\tr}_{\la^\star,s_{\circ}(C)}=\big(1-o(1)\big)\E\bZ^{\tr}_{\la^\star,s_{\circ}(C)}$ holds. Thus, combining \eqref{eq:bound:tilde:Z:squared} and Proposition \ref{prop:Zmomentratio:untrun}, we have
	\begin{equation}\label{eq:cor:conditional:var:negligible:upper}
	\limsup_{C\to\infty}\limsup_{n\to\infty} \frac{\E\big(\bwZ^{\tr}_{\la^\star,s_{\circ}(C)}\big)^{2}}{\big(\E\bwZ^{\tr}_{\la^\star,s_{\circ}(C)}\big)^{2}}\leq \exp \bigg( \sum_{\zeta} \mu(\zeta) \delta(\zeta;\la^\star)^2 \bigg).
	\end{equation}
	Therefore, \eqref{eq:cor:conditional:var:negligible:lower}, \eqref{eq:cor:conditional:var:negligible:upper}, and Lemma \ref{lem:deltabound} conclude the proof.
	\end{proof}
	
	\begin{proof}[Proof of Theorem \ref{thm:lowerbd}]
	Fix $\eps>0$. Having Corollary \ref{cor:conditional:var:negligible} in mind, for $\delta>0, C>0$ and $l_0\in \N$, we bound
	\begin{equation}\label{eq:proof:thm:whp:1}
	\P\bigg(\frac{\bwZ^{\tr}_{\la^\star,s_{\circ}(C)}}{\E\bwZ^{\tr}_{\la^\star,s_{\circ}(C)}}\leq \delta\bigg)\leq \P\bigg(\bigg|\frac{\bwZ^{\tr}_{\la^\star,s_{\circ}(C)}}{\E\bwZ^{\tr}_{\la^\star,s_{\circ}(C)}}-\frac{\E\big[\bwZ^{\tr}_{\la^\star,s_{\circ}(C)}\,\big|\, \underline{X}_{\leq l_0}\big]}{\E\bwZ^{\tr}_{\la^\star,s_{\circ}(C)}}\bigg|\geq \delta \bigg)+\P\bigg(\frac{\E\big[\bwZ^{\tr}_{\la^\star,s_{\circ}(C)}\,\big|\, \underline{X}_{\leq l_0}\big]}{\E\bwZ^{\tr}_{\la^\star,s_{\circ}(C)}}\leq 2\delta\bigg).
	\end{equation}
	We first control the second term of the \textsc{rhs} of \eqref{eq:proof:thm:whp:1}: Proposition \ref{prop:cycleeffect:moments}-(3) shows (cf. \eqref{eq:momentlim:cycconditioning})
	\begin{equation*}
	    \frac{\E\big[\bwZ^{\tr}_{\la^\star,s_{\circ}(C)}\,\big|\, \underline{X}_{\leq l_0}\big]}{\E\bwZ^{\tr}_{\la^\star,s_{\circ}(C)}}\; {\overset{\textnormal{d}}{\longrightarrow}}\; W_{l_0}:=\prod_{||\zeta||\leq l_0}\Big(1+\delta(\zeta;\la^\star)\Big)^{\bar{X}(\zeta)}e^{-\mu(\zeta)\delta(\zeta)},
	\end{equation*}
	where $\{\bar{X}(\zeta)\}_{\zeta}$ are independent Poisson random variables with mean $\{\mu(\zeta)\}_\zeta$. Moreover, we have
	\begin{equation*}
	    W_{\ell_0} \; {\overset{\textnormal{a.s.}}{\longrightarrow}}\; W:=\prod_{\zeta}\Big(1+\delta(\zeta;\la^\star)\Big)^{\bar{X}(\zeta)}e^{-\mu(\zeta)\delta(\zeta)}\quad\textnormal{and}\quad W>0\quad\textnormal{a.s.},
	\end{equation*}
	where the infinite product in $W$ is well defined a.s. due to Lemma \ref{lem:deltabound} (see Theorem 9.13 of \cite{jlrrg} for a proof). Thus, for small enough $\delta\equiv \delta_{\eps}$ which does not depend on $C$ and large enough $l_0\geq l_0(\eps)$, we have
	\begin{equation}\label{eq:proof:thm:whp:2}
	    \limsup_{n\to\infty}\P\bigg(\frac{\E\big[\bwZ^{\tr}_{\la^\star,s_{\circ}(C)}\,\big|\, \underline{X}_{\leq l_0}\big]}{\E\bwZ^{\tr}_{\la^\star,s_{\circ}(C)}}\leq 2\delta_{\eps} \bigg)\leq \frac{\eps}{2}.
	\end{equation}
	We now turn to the first term of the \textsc{rhs} of \eqref{eq:proof:thm:whp:1}. By Chebyshev's inequality and Corollary \ref{cor:conditional:var:negligible}, for large enough $C\geq C_{\eps}$ and $\ell_0\geq \ell_0(\eps)$, we have
	\begin{equation}\label{eq:proof:thm:whp:3}
	   \limsup_{n\to\infty} \P\bigg(\bigg|\frac{\bwZ^{\tr}_{\la^\star,s_{\circ}(C)}}{\E\bwZ^{\tr}_{\la^\star,s_{\circ}(C)}}-\frac{\E\big[\bwZ^{\tr}_{\la^\star,s_{\circ}(C)}\,\big|\, \underline{X}_{\leq l_0}\big]}{\E\bwZ^{\tr}_{\la^\star,s_{\circ}(C)}}\bigg|\geq \delta_{\eps} \bigg)\leq (\delta_{\eps})^{-2}\limsup_{n\to\infty}\frac{\E\Big[\Var\big(\bwZ^{\tr}_{\la^\star,s_{\circ}(C)} \,\big|\,\underline{X}_{\leq l_0}\big)\Big]}{\big(\E \bwZ^{\tr}_{\la^\star, s_{\circ}(C)}\big)^{2}}\leq \frac{\eps}{2}.
	\end{equation}
	Therefore, by \eqref{eq:proof:thm:whp:1}, \eqref{eq:proof:thm:whp:2} and \eqref{eq:proof:thm:whp:3}, for $\delta\equiv \delta_{\eps}$ and $C\geq C_{\eps}$, we have
	\begin{equation}\label{eq:proof:thm:whp:4}
	    \limsup_{n\to\infty} \P\bigg(\frac{\bwZ^{\tr}_{\la^\star,s_{\circ}(C)}}{\E\bwZ^{\tr}_{\la^\star,s_{\circ}(C)}}\leq \delta\bigg)\leq \eps. 
	\end{equation}
	Since $\E\bwZ^{\tr}_{\la^\star,s_{\circ}(C)}=\big(1-o(1)\big)\E\bZ^{\tr}_{\la^\star,s_{\circ}(C)}$ holds by Proposition 3.17 of \cite{nss20a} and $\bZ^{\tr}_{\la^\star,s}\asymp e^{n\la^\star s}\bN^{\tr}_{s}$ holds by definition, \eqref{eq:proof:thm:whp:4} concludes the proof.
	\end{proof}

	\section*{Acknowledgements}
	We thank Amir Dembo, Nike Sun and Yumeng Zhang for helpful discussions. We thank the anonymous reviewers for their careful reading and valuable feedbacks which improved our paper. DN is supported by a Samsung Scholarship. AS is  supported  by NSF grants DMS-1352013 and DMS-1855527, Simons Investigator grant and a MacArthur Fellowship. YS is partially supported by NSF grants DMS-1613091 and DMS-1954337.
	
	\bibliography{naesatref}
	\newpage

	\appendix
	\section{Proof of technical lemmas}\label{subsec:app:deltabd}
		In this section, we provide the omitted proofs from Section \ref{sec:whp} and \ref{subsec:whp:rescaled}, which deals with the effect of short cycles in $\E \bZ_\lambda$. We begin with establishing Lemma \ref{lem:deltabound} and Proposition \ref{prop:cycleeffect:moments}-(4) in Section \ref{subsec:app:whp:delta}. Then, we discuss details of
		Corollary \ref{cor:Ymanycyc} in Section \ref{subsec:app:Ymanycyc}.
		In Section \ref{subsec:app:cyceff:moments}, we establish Proposition \ref{prop:cycleeffect:moments}-(3).
		The final subsection, Section \ref{subsec:app:Zplantedexp}, is devoted to the proof of 
		Lemma \ref{lem:Zplantedexpansion}.

		\subsection{Proof of Proposition \ref{prop:cycleeffect:moments}-(4)}\label{subsec:app:whp:delta}
		The goal of this subsection is to study $\delta(\zeta;\lambda)$ and $\delta_L(\zeta;\lambda)$ defined in \eqref{eq:def:delta by trace}. We first establish Lemma \ref{lem:deltabound}, and then show (4) of Proposition \ref{prop:cycleeffect:moments}.  Our approach is based on a rather direct study on the matrix $(\dot{A}\hat{A})^\zeta$. Once we obtain an explicit formula of the matrix, we use the combinatorial properties of free trees and the estimates on the belief propagation fixed point.

		\begin{proof}[Proof of Lemma \ref{lem:deltabound}]

	  Throughout the proof, we fix $\la\in (0,\la^\star]$. Moreover, we assume that $\zeta =\underline{0}\in\{0,1\}^{2l}$, and write $ \hat{A}_L\equiv \hat{A}^{0,0}_L$. It will be apparent that the same proof works for different choices of $\zeta$. We first introduce several notations that will be crucial in the proof as follows.
		
		On the finite-dimensional vector space $\R^{\Omega_L}$, we define the inner product $\langle\, \cdot\,,\,\cdot\,\rangle_\star$ by
		\begin{equation*}
		\langle f_1, f_2 \rangle_\star \equiv \sum_{\sigma \in \Omega_L} f_1(\sigma) f_2(\sigma) \bar{H}_{\la,L}^\star (\sigma),
		\end{equation*}
		and denote $||f||_\star^2 \equiv \langle f,f\rangle_\star$. Note that both $\dot{A}_L$ and $\hat{A}_L$ are stochastic matrices, since from the fact that $\dot{q}^\star_{\la,L}$ is the BP fixed point, we have for every $\tau_1\in \Omega$ and $(\tL_1,\tL_2)\in \{0,1\}^2$ that
  \begin{equation}\label{eq:optimal:H:identity}
  \sum_{\tau_2} \dot{H}^\star_{\la}(\tau_1,\tau_2)=\bar{H}^\star_{\la}(\tau_1)=\sum_{\tau_2}\hat{H}^{\tL_1,\tL_2}(\tau_1,\tau_2)\,.
  \end{equation}
  Thus, all-1 vector $\one$ is an eigenvector with eigenvalue $1$ for both of the matrices $\dot{A}_L$ and $\hat{A}_L$. Also, note that if $f$ is orthogonal to $\one$ (denote $f\perp_\star \one$), then 
		\begin{equation*}
		\langle \dot{A}_Lf , \one \rangle_\star = \langle \hat{A}_L f, \one\rangle_\star =\langle \dot{A}_L\hat{A}_Lf,\one \rangle_\star  =0\,,
		\end{equation*}
  where the equalities follow from \eqref{eq:optimal:H:identity}. Moreover, it is straightforward to see that $(\dot{A}_L\hat{A}_L)$ defines a transition matrix of an ergodic Markov chain on $\Omega_L$. Thus, 1 is the largest eigenvalue with single multiplicity. Define the matrix $B_L\in \R^{|\Omega_L|\times |\Omega_L|}$ by
		\begin{equation}\label{eq:def:Bmatrix}
		B_L(\sigma, \tau) \equiv \dot{A}_L\hat{A}_L (\sigma,\tau) - \bar{H}^\star_{\la,L} (\tau), \quad \forall \sigma, \tau \in \Omega_L\,.
		\end{equation}
		That is, $B_L \equiv \dot{A}_L \hat{A}_L-\one (\bar{H}^\star)^{\sf T}$, where we denoted $\bar{H}^\star$ by the vector $\bar{H}^\star\equiv (\bar{H}^\star_{\la,L}(\tau))_{\tau\in \Omega_L}$ with abuse of notation. Since $\dot{A}_L \hat{A}_L\one=\one$ and $(\bar{H}^\star)^{\sf T}\dot{A}_L \hat{A}_L=(\bar{H}^\star)^{\sf T}$ holds, we have that $B_L \one = (\bar{H}^\star)^{\sf T} B_L=0$. Thus, $(\dot{A}_L\hat{A}_L)^{l}=B_L^{l}+\one (\bar{H}^\star)^{\sf T}$, so
		\begin{equation*}
		Tr \left[ (\dot{A}_L\hat{A}_L)^l\right] = 1+ Tr \left[ B^l_L \right].
		\end{equation*}
		The remaining work is to understand the \textsc{rhs} of above.
		
		Let $\Omega_\circ \equiv \{\bb_0,\bb_1,\rr_0,\rr_1,\fs  \} $, and $\Omega_{\texttt{f}}\equiv \Omega_L \setminus \Omega_\circ$.
		We first need to understand how the entries of $B_L$ are defined, especially $B_L(\sigma,\tau)$ with $\sigma,\tau \in \Omega_{\texttt{f}}$. If $\sigma, \tau \in \Omega_{\texttt{f}}$, then we have the following observations:
		\begin{itemize}
			\item $\dot{A}_L(\sigma, \tau)=0$, unless both $\sigma$ and $\tau $ define the same free tree, and their \textit{root edges} can be embedded into the tree as distinct edges adjacent  to the same variable.
			
			\item When $\sigma,\tau$ satisfy the above condition, denote $\sigma = \sigma_v(e;\ttt)$ and $\tau= \sigma_v(e';\ttt)$, where $\ttt$ denotes the free tree given by $\sigma, \tau$ and $v$, $e$ describe the variable and the half-edge in $\ttt$ where $\sigma$ can be embedded. Then, we can observe that
			\begin{equation*}
			\dot{A}_L(\sigma,\tau) = \frac{1}{d-1} \left| \left\{e'': e''\sim v, \, e''\neq e,\, \sigma_v(e'';\ttt) = \sigma_v(e';\ttt)  \right\} \right|.
			\end{equation*}
			
			\item This holds the same for $\hat{A}$, and hence we have for all $\sigma,\tau \in \Omega_{\texttt{f}}$ that
			\begin{equation*}
			\hat{A}_L(\sigma,\tau) =
			\frac{1}{k-1} \left| \left\{e'': e''\sim a, \, e''\neq e,\, \sigma_a(e'';\ttt) = \sigma_a(e';\ttt)  \right\} \right|, 
			\end{equation*}
			if and only if there exists some $\ttt, a,e,e'$ such that $\sigma= \sigma_a(e;\ttt)$, $\tau=\sigma_a(e';\ttt)$. Otherwise it is 0.
		\end{itemize}
		
		For a free tree $\ttt$, suppose that $v,a\in \ttt$ with $v\sim a$, and $e\sim v$, $e'\sim a$ satisfy $e\neq (va) \neq e'$. Then, letting $\sigma = \sigma_v(e;\ttt)$ and $\tau=\sigma_a(e';\ttt)$, we have
		\begin{equation}\label{eq:Aformula}
		\dot{A}_L\hat{A}_L (\sigma, \tau) = \frac{|\{(a'',e''): e''\sim a'' \sim v, \, e''\neq (va''),\, \sigma_{a''}(e'';\ttt) = \sigma_{a}(e';\ttt)   \} |  }{(d-1)(k-1)} .
		\end{equation}
		Here, note that there cannot be $\tau'\in \Omega_\circ$ such that $\dot{A}(\sigma,\tau') \hat{A}(\tau',\tau) \neq 0$. Further, since $\bar{H}^\star_L (\Omega_{\texttt{f}}) \leq (k^C2^{-k})^2$, for such $\sigma, \tau$ we have
		\begin{equation}\label{eq:Bformula}
			B_L (\sigma, \tau) =  \frac{|\{(a'',e''): e''\sim a'' \sim v, \, e''\neq (va''),\, \sigma_{a''}(e'';\ttt) = \sigma_{a}(e';\ttt)   \} |  }{(1+O(k^C2^{-k})) (d-1)(k-1)} .
		\end{equation}
		For $\sigma, \tau \in \Omega_{\texttt{f}}$ that do not satisfy the above condition, we have $B_L(\sigma,\tau) =-\bar{H}^\star_{\la,L}(\sigma)= O((k^C2^{-k})^2)$. Having these observations in mind, the main analysis is to establish the following.
		
		\begin{claim}\label{claim:tracesegment}
	
		 There exists an absolute constant $C>0$ such that the following hold true:		For any positive integer $l$, we have
			\begin{align}
			&\sum_{\sigma_1 ,\ldots, \sigma_{l-1}   \in \Omega_{\textnormal{\texttt{f}}}}
			\prod_{i=1}^{l-1} B_L(\sigma_i,\sigma_{i+1}) 
			\leq (k^C 2^{-k})^{l }, \quad \forall\sigma_0 , \sigma_l \in\Omega_\circ; \label{eq:traceseg:nonfree} \\
			&\sum_{\sigma_1,\ldots,\sigma_l\in \Omega_{\textnormal{\texttt{f}}}  } \prod_{i=1}^{l-1} B_L(\sigma_i, \sigma_{i+1})  \leq (k^C 2^{-k})^{l}. \label{eq:traceseg:free}
			\end{align}
		\end{claim}

		We first assume that the claim holds true and finish the proof of Lemma \ref{lem:deltabound}. In the formula
		\begin{equation*}
		\begin{split}
			Tr\left[B^l_L \right] &= \sum_{\sigma_1 , \ldots, \sigma_l}
			\prod_{i=1}^l B_L(\sigma_i , \sigma_{i+1}) \\
			&=
			\sum_{\sigma_1 , \ldots, \sigma_l\in \Omega_{\textnormal{\texttt{f}}}}
			\prod_{i=1}^l B_L(\sigma_i , \sigma_{i+1})
			+
					\sum_{\substack{\sigma_1 , \ldots, \sigma_l:\\
							\exists \sigma_i \in \Omega_\circ}}
					\prod_{i=1}^l B_L(\sigma_i , \sigma_{i+1})
		\end{split}
		\end{equation*}
		(with $\sigma_{l+1}\equiv \sigma_1$), we see that the first sum in the last line can be controlled by \eqref{eq:traceseg:free}. To be specific, 	we define for $\sig = (\sigma_i)_{i=1}^l \in \Omega_L^l$ that
		\begin{equation}\label{eq:def:freetreecollec}
		\ttt [ \sig] \equiv \{\ttt(\sigma_i) : i\in [l] \},
		\end{equation} 
		where $\ttt(\sigma)$ denotes the free tree associated with the color $\sigma$.
		If $\sig=(\sigma_i)_{i=1}^l \subset \Omega_{\texttt{f}}^{l}$ contributes to the above sum, then $|\ttt[\sig]|>1$, since $|\ttt[\sig]|=1$ would imply that the free component given by $\sig$ forms a cycle. Therefore, we can bound
		\begin{equation*}
		\sum_{\sigma_1 , \ldots, \sigma_l\in \Omega_{\textnormal{\texttt{f}}}}
		\prod_{i=1}^l B_L(\sigma_i , \sigma_{i+1})
		\leq 
		\sum_{\substack{\sig=(\sigma_i)_{i=1}^{l+1}\subset \Omega_{\textnormal{\texttt{f}}} :\\ |\ttt[\sig]|>1} } \prod_{i=1}^{l} B_L(\sigma_i, \sigma_{i+1})  \leq (k^C 2^{-k})^{l}. 
		\end{equation*}
		 For the second sum, there are some $i$ with $\sigma_i\in \Omega_{\texttt{f}}$, and in this case we can use \eqref{eq:traceseg:nonfree} to control the summation. When there are multiple such colors, we estimate the sum within each interval between $\sigma_i, \sigma_{i'} \in \Omega_\circ$ by \eqref{eq:traceseg:nonfree}. Since the number of ways to choose a subset of the indices $\{i\mid \sigma_i\in \Omega_\circ\}\subseteq [l]$ is at most $2^l$, it can be absorbed into $(k^C2^{-k})^l$ and hence we obtain the conclusion of Lemma \ref{lem:deltabound}.
	\end{proof}
	
	\begin{proof}[Proof of Claim \ref{claim:tracesegment}]
			
			According to \eqref{eq:Aformula} and \eqref{eq:Bformula}, it suffices to establish  \eqref{eq:traceseg:nonfree} for $A_L\equiv \dot{A}_L\hat{A}_L$. This is because the  contribution  to $B_L(\sigma, \tau)$ from $\sigma$, $\tau$ such that $A_L(\sigma, \tau) =0$ is bounded by $O((k^C2^{-k})^2)$, which is of smaller order than $k^C2^{-k}$ as we can see from \eqref{eq:Bformula}.

				 In order to obtain \eqref{eq:traceseg:nonfree}, let $\sig=(\sigma_i)_{i=1}^{l-1} \subset \Omega_{\texttt{f}}^{l-1}$, and observe that we need $|\ttt[\sig]|=1$ to have
			\begin{equation*}
			\prod_{i=1}^{l-2} A_L(\sigma_i,\sigma_{i+1}) >0.
			\end{equation*} 
			For a fixed $\sigma_1\in \Omega_{\texttt{f}}$, let $\ttt, v, e$ be such that $\sigma_1 =\sigma_v(e;\ttt)$. Moreover, define $\ttt_{v\setminus e}$ to be the connected component of $\ttt \setminus \{e \}$ containing $v$, and let $$\partial N_l(v; \ttt_{v\setminus e} ) := \{u\in V(\ttt_{v\setminus e}): \textnormal{dist}(u,v) = 2l \}.$$  Then, the formula \eqref{eq:Aformula} tells us that 
			\begin{equation}\label{eq:tracesec:nbdsize}
			\sum_{\sigma_2,\ldots,\sigma_{l-1}\in \Omega_{\texttt{f}}} \prod_{i=1}^{l-2} A_L(\sigma_i,\sigma_{i+1}) = \frac{ |\partial N_l (v;\ttt_{v\setminus e}) |}{((d-1)(k-1))^{l}}
			\leq 
			\frac{v(\ttt) }{((d-1)(k-1))^l}
			.
			\end{equation}
			Since $A_L(\sigma_0,\sigma_1) \leq (k^C2^{-k})^{v(\ttt)}$ for any $\sigma_0\in \Omega_\circ$ and $\sigma_1$ with $\ttt(\sigma_1)= \ttt$, we see that
			\begin{equation}\label{eq:traceseg:sumnbdsize}
			\sum_{\sigma_1 ,\ldots, \sigma_{l-1}   \in \Omega_{\textnormal{\texttt{f}}}}
			\prod_{i=0}^{l-1} A_L(\sigma_i,\sigma_{i+1}) 
			\leq \sum_{\ttt} \sum_{\sigma\,:\, \ttt(\sigma)=\ttt} (k^C2^{-k})^{v(\ttt)} \frac{v(\ttt)}{2^{kl}}
			\leq (k^C 2^{-k})^{l }.
			\end{equation}
			
			The inequality \eqref{eq:traceseg:free} can be proven in a similar way.  Let $\sig=(\sigma_i)_{i=1}^l$, and note that $|\ttt[\sig]=1|$ does not give any contribution to \eqref{eq:traceseg:free}, since it implies that the free component given by $\sig$ contains a cycle. Suppose that $|\ttt[\sig]|=2$, and assume that $|\ttt[\sigma_1,\ldots,\sigma_{i_0-1}]|=|\ttt[\sigma_{i_0},\ldots,\sigma_{l}]|=1$. Using \eqref{eq:tracesec:nbdsize}, we obtain that
			\begin{equation}\label{eq:traceseg:freenbdsize}
			\sum_{\substack{\sig\subset \Omega_{\texttt{f}}:\\  
				\ttt[\sigma_1,\ldots,\sigma_{i_0-1}] = \{\ttt_1\}\\ \ttt[\sigma_{i_0},\ldots,\sigma_{l}]=\{\ttt_2\}	}}
		\prod_{i=1}^{l-1} B_L(\sigma_i,\sigma_{i+1}) 
		\leq 
		(k^C2^{-k})^{v(\ttt_1) + v(\ttt_2)} \frac{v(\ttt_1) v(\ttt_2)}{2^{kl}},
			\end{equation}
			where the term $(k^C2^{-k})^{v(\ttt_1) + v(\ttt_2)}$ comes from $$B_L(\sigma_l,\sigma_1) \leq (k^C2^{-k})^{v(\ttt_1)}, \quad B_L(\sigma_{i_0-1}, \sigma_{i}) \leq (k^C2^{-k})^{v(\ttt_2)}.$$  
			Thus, summing \eqref{eq:traceseg:freenbdsize} over all $i_0$, $\ttt_1$, $\ttt_2$ as \eqref{eq:traceseg:sumnbdsize}, we obtain \eqref{eq:traceseg:free}. The case where $|\ttt[\sig]|>2$ can be derived analogously and is left to the interested reader.
	\end{proof}

	The final goal of this subsection is demonstrating Proposition \ref{prop:cycleeffect:moments}-(4). This comes as a rather straight-forward application of Claim \ref{claim:tracesegment}, and hence we briefly sketch the proof without all the details.
	
	\begin{proof}[Proof of Proposition \ref{prop:cycleeffect:moments}-(4)]
		
		Define the matrix $B$ analogously as \eqref{eq:def:Bmatrix}. Let $L_0>0$ and let $B|_{L_0}$ be the $\Omega_{L_0} \times \Omega_{L_0}$ submatrix of $B$. Then, we can write
		\begin{equation*}
		Tr\left[B^l \right]- Tr\left[(B|_{L_0})^l \right] 
		=\sum_{\substack{\sigma_1,\ldots,\sigma_l\in \dot{\Omega}:\\ \exists \sigma_i\in \Omega \setminus \Omega_{L_0}}} \prod_{i=1}^l B(\sigma_i,\sigma_{i+1}), 
		\end{equation*}
		where $\sigma_{l+1}\equiv \sigma_1$. Since $\ttt[\sig]$ cannot be a singleton for $\sig = (\sigma_1)_{i=1}^l$ that contributes to the above sum due to the same reason as in the proof of \eqref{eq:traceseg:free}, there should be some $i_0$ such that $\sigma_{i_0}\in \Omega\setminus\Omega_{L_0}$ and $\ttt(\sigma_{i_0-1}) \neq \ttt(\sigma_{i_0})$. For such $i_0$, we get 
		\begin{equation*}
		B(\sigma_{i_0-1},\sigma_{i_0}) \leq (k^C2^{-k})^{-v(\sigma_{i_0})},
		\end{equation*}
		and hence the above sum can be controlled by
		\begin{equation}\label{eq:traceseg:prop:Bbd}
		\sum_{\substack{\sigma_1,\ldots,\sigma_l\in \Omega:\\ \exists \sigma_i\in \Omega \setminus \Omega_{L_0}}} \prod_{i=1}^l B(\sigma_i,\sigma_{i+1}) 
		\leq (k^C2^{-k})^{l+L_0-1}.
		\end{equation}
		
		 In order to compare $Tr[B^l]$ to $Tr[B_L^l]$, we set $L>L_0 >0$, and obtain that
		\begin{equation}\label{eq:traceseg:prop:BLbd}
		Tr\left[B_L^l \right]- Tr\left[((B_L)|_{L_0})^l \right] 
		\leq (k^C2^{-k})^{l+L_0-1}.
		\end{equation}
		Moreover, we can see that $Tr\left[((B_L)|_{L_0})^l \right] $ converges to $Tr\left[(B|_{L_0})^l \right] $ as $L\to\infty$ since $H^\star_L \to H^\star$. Therefore, we obtain the conclusion of Proposition \ref{prop:cycleeffect:moments}-(4) by combining \eqref{eq:traceseg:prop:Bbd} and \eqref{eq:traceseg:prop:BLbd}.
	\end{proof}

	\subsection{Proof of Corollary \ref{cor:Ymanycyc}}\label{subsec:app:Ymanycyc}
	
	In this section, we present the proof of Corollary \ref{cor:Ymanycyc}. The proof is based on ideas from Proposition \ref{prop:cycleeffect:moments} and Corollary \ref{cor:toomanycyc:Zg}. We show (1) of the corollary, and then the derivation of (2) will be analogous. 
	
	Note that for any nonnegative integer $x$, we have $(1+\theta)^x = \sum_{a\geq 0} \frac{(x)_a}{a!} \theta^a$. Set $\tilde{\delta}(\zeta) = (1+\delta_L(\zeta))^{-1}-1$, we can write
	\begin{equation*}
	\bY = \bZ' \sum_{\underline{a}}\frac{1}{\underline{a}!} \tilde{\underline{\delta}}^{\underline{a}} (\underline{X})_{\underline{a}},
	\end{equation*}
	where we abbreviated $\bZ' = \bwZ^{(L),\tr}_{\lambda}$.
	Let $c_{\textsf{cyc}}=c_\textsf{cyc}(l_0)$ be as Proposition \ref{prop:cycleeffect:moments}, and set $c' = \frac{1}{3}(c\wedge c_{\textsf{cyc}})$. We will control $\E [\bZ' \cdot (\underline{X})_{\underline{a}} \one\{||\underline{X}||_\infty \geq c\log n \}] $ for each $\underline{a}$ as follows.
	
	\vspace{2mm}
	\noindent \textbf{Case 1.} $||\underline{a}||_\infty \leq c'\log n$.
	\vspace{2mm}
	
	Controlling the indicator crudely by $\one\{||\underline{X}||_\infty \geq c\log n \} \leq \sum_{||\zeta'||\leq l_0 } \one\{X(\zeta') \geq c\log n \}$, we study
	\begin{equation*}
	\E [\bZ' \cdot (\underline{X})_{\underline{a}} \one\{X(\zeta') \geq c\log n \}]
	\end{equation*}
	for each $\zeta'$. Define $\underline{a}'$ by
	\begin{equation*}
	a'(\zeta ) = 
	\begin{cases}
	a(\zeta) & \textnormal{if } \zeta \neq \zeta';\\
	a(\zeta') + c'\log n & \textnormal{if } \zeta=\zeta'.
	\end{cases}
	\end{equation*}
	Since $||\underline{a}'||_\infty  \leq \frac{2}{3} (c_{\textsf{cyc}}\wedge c)\log n$, we can see that
	\begin{equation*}
	\begin{split}
		\E [\bZ' \cdot (\underline{X})_{\underline{a}} \one\{X(\zeta') \geq c\log n \}]
		&\leq 
		\left(\frac{c}{3}\log n \right)^{-c'\log n}	\E [\bZ' \cdot (\underline{X})_{\underline{a}'} ]\\
		&\leq n^{-\Omega(\log\log n)} \E\bZ',
	\end{split}
	\end{equation*}
	where the last inequality follows from Proposition \ref{prop:cycleeffect:moments}.
	
	\vspace{2mm}
	\noindent \textbf{Case 2.} $||\underline{a}||_\infty >c'\log n$.
	\vspace{2mm}
	
	In this case, it will be enough to study $\E [\bZ'\cdot(\underline{X})_{\underline{a}} ]$, similarly as Proposition \ref{prop:cycleeffect:moments}. However, the proof of Proposition \ref{prop:cycleeffect:moments} apparently breaks down when $||\underline{a}||_1$ is large, and hence we work with a more general but weaker approach to control Case 2.
	
	To begin with, as \eqref{eq:ZXbyindicators} we write
	\begin{equation*}
		\E \left[\bZ' (\underline{X})_{\underline{a}}  \right]
		=
		\sum_{\YY} \sum_{\underline{\tau}_\YY} \E \left[\bZ' \one\{\YY, \underline{\tau}_\YY \} \right],
	\end{equation*}
	where $\YY = \{\YY_i(\zeta) \}_{i\in[a_\zeta],\, ||\zeta||\leq l_0}$ denotes the locations of $\underline{a}$ $\zeta$-cycles and $\underline{\tau}_\YY$ describes a prescribed coloring configuration on them (recall Definition \ref{def:zetcycle}). As before, we derive an estimate on the summand for each fixed $(\YY, \underline{\tau}_\YY)$. Let $\Delta=\Delta[\underline{\tau}_\YY]$ be given as Definition \ref{def:cycprofile}. Consider a literal assignment $\uL_E$ on and an empirical count measure $g=(\dot{g}, (\hat{g}^{\uL})_{\uL\in\{0,1\}^k } , \bar{g} )$ on $\GGG$ that contributes to $\E\bZ'$. Here, we assume that $\uL_E$ and $(\hat{g}^{\uL})$ are compatible in the sense that $|\{a\in F: (\uL_E)_a = \uL \}| = |\hat{g}^{\uL}| $ for each $\uL \in \{0,1\}^k$. Based on the expression in the first line of \eqref{eq:cycleef:trun:Zg expansion}, we have that
	\begin{equation}\label{eq:Zgexpansion:manycyc}
 \frac{ 	\E \left[\left.\bZ'[g] \one\{\YY, \underline{\tau}_\YY\} \,\right|\, \uL_E \right] }{ 	\E \left[\left.\bZ'[g]  \,\right|\, \uL_E \right]}	
 =
\frac{1}{ (\bar{g})_{\bar{\Delta}_c} }
 \frac{(\dot{g})_{\dot{\Delta}}}{(n)_{|\dot{\Delta}| }} \prod_{\uL\in\{0,1\}^k}  \frac{
 	(\hat{g}^{\uL})_{\hat{\Delta}^{\uL}}
 	}{ (|\hat{g}^{\uL}|)_{|\hat{\Delta}^{\uL}|} }.
	\end{equation}
	Define the quantity $H(g,\Delta)$ to be 
	\begin{equation*}
	\HH(g,\Delta) \equiv \frac{(\dot{g})_{\dot{\Delta}} \prod_{\uL}(\hat{g}^{\uL})_{\hat{\Delta}^{\uL}} }{(\bar{g})_{\bar{\Delta}_c} }.
	\end{equation*}
	 Moreover, let $\hat{\Delta} \equiv \sum_{\uL} \hat{\Delta}^{\uL}$, and define
	\begin{equation*}
	\begin{split}
	\eta&\equiv \eta(\YY) \equiv |\bar{\Delta}_c| -|\dot{\Delta}| -|\hat{\Delta}| .
	\end{split}
	\end{equation*}
	Our goal is to deduce a general upper bound on $\HH(g,\Delta)$ that depends only on $\eta(\YY)$, not on $g$ or $||\underline{a}||$. 
	
	We can interpret $\bar{\Delta}_c$ as a partition of the set $[|\bar{\Delta}_c|]$. That is, $\bar{\Delta}_c(\sigma)$ for each $\sigma\in \dot{\Omega}_L$ corresponds to a (disjoint) interval of length $|\bar{\Delta}_c(\sigma)|$ inside $[|\bar{\Delta}_c|]$. Similarly, we can think of a partition of the set $[|\dot{\Delta}|+|\hat{\Delta}| ]$ by disjoint intervals of length $|\dot{\Delta}(\sigma)|$ and $|\hat{\Delta}^{\uL}(\sigma)|$, for each $\sigma\in \dot{\Omega}_L$ and $\uL \in \{0,1\}^k$. Since $\bar{\Delta}$ corresponds to a marginal measure of $\dot{\Delta}$ and $\hat{\Delta}$, we see that the latter partition of $[|\dot{\Delta}|+|\hat{\Delta}|]$ can be chosen as a subpartition of the former of $[|\bar{\Delta}|]$. This means that the expression in the numerator of $\HH(g,\Delta)$ must be smaller than its denominator. Furthermore, note that $|\bar{\Delta}|$ exceeds $ |\dot{\Delta}|+|\hat{\Delta}|$ by $\eta$, and for any nonnegative integers $\{y(\sigma)\}_{\sigma\in \dot{\Omega}_L}$ such that $\sum_{\sigma} y(\sigma) \geq  \eta$, it holds that
	\begin{equation*}
	\prod_{\sigma\in\dot{\Omega}_L} y(\sigma)! \geq  \left( \left\lfloor \frac{\eta}{|\dot{\Omega}_L|} \right\rfloor ! \right)^{|\dot{\Omega}_L|}.
	\end{equation*}
	Thus, $\HH(g,\Delta)$ can be crudely controlled as follows:
	\begin{equation*}
	\HH(g,\Delta) \leq \left( \left\lfloor \frac{\eta}{|\dot{\Omega}_L|} \right\rfloor ! \right)^{-|\dot{\Omega}_L|}.
	\end{equation*}
	
	On the other hand, for a fixed $\eta$, we can bound the number of possible choices of $\YY$ analogously as \eqref{eq:choice of eta YY}. Setting $a^\dagger = \sum_{||\zeta||\leq l_0} ||\zeta|| a_\zeta$ and implementing \eqref{eq:choice of eta YY} on \eqref{eq:Zgexpansion:manycyc}, we deduce that 
	\begin{equation*}
	\sum_{\YY: \eta(\YY)=\eta} \sum_{\underline{\tau}_\YY}
	 \frac{ 	\E \left[\bZ'[g] \one\{\YY, \underline{\tau}_\YY\}  \right] }{ 	\E \left[\bZ'[g]  \right]}	
	\leq (|\dot{\Omega}_L|^d dk)^{2a^\dagger} ((4l_0)^2a^\dagger)^\eta  \left( \left\lfloor \frac{\eta}{|\dot{\Omega}_L|} \right\rfloor ! \right)^{-|\dot{\Omega}_L|}.
	\end{equation*}
	Therefore, we can sum this over all $\eta $ and obtain that
	\begin{equation*}
	\frac{ 	\E \left[\bZ'[g]  (\underline{X})_{\underline{a}}  \right] }{ 	\E \left[\bZ'[g]   \right]}	
	\leq C^{a^\dagger},
	\end{equation*}
	where $C$ is a constant depending on $k,L,$ and $l_0$. Averaging over $g$ and summing the above for $||\underline{a}||_\infty \geq \frac{1}{3}\log n$, we see that
	\begin{equation*}
	\sum_{\underline{a}: ||\underline{a}||_\infty \geq c'\log n}	\frac{ 	\E \left[\bZ' \cdot  (\underline{X})_{\underline{a}}  \right] }{ 	\E \left[\bZ'   \right]}	= n^{-\Omega_{k} (\log\log n)}.
	\end{equation*}
	
	The conclusion for (2) can be obtained analogously if we work with the pair model.  \qed

		\subsection{Proof of Proposition \ref{prop:cycleeffect:moments}-(3)}\label{subsec:app:cyceff:moments}
		
		Here present the proof of Proposition \ref{prop:cycleeffect:moments}-(3), by establishing \eqref{eq:cycmo:trun:1st} for $\bZ_\lambda$. The proof for $\bZ_{\lambda,s_n} $ will be analogous from the former case.   The main difference from the truncated model is that the optimal empirical measure $H^*$ is no longer bounded below by a constant. This aspect requires an extra care in the derivation of (\ref{eq:cycleef:trun:Zg expansion}), which indeed  is no longer true in general for the untruncated model. To overcome such difficulty, let $\dot{q}^\star = \dot{q}^\star_{\lambda^\star}\in \PPP(\dot{\Omega})$ be the BP fixed point, and we split the space  $\dot{\Omega}$ into two types:
		\begin{align}
		&\dot{\Omega}^{\textnormal{\textsf{typ}}} \equiv \dot{\Omega}^{\textnormal{\textsf{typ}}}(n) \equiv 
		\{\tau \in \dot{\Omega}: \dot{q}^\star(\tau) \ge n^{-1/4} \},
		&\dot{\Omega}^{\textnormal{\textsf{atyp}}} \equiv \dot{\Omega}^{\textnormal{\textsf{atyp}}}(n)\equiv
		\{\tau \in \dot{\Omega} : \dot{q}^\star(\tau) < n^{-1/4} \}. 
		\end{align}
		
		Now, recall the expression \eqref{eq:ZXbyindicators} for $\bwZ_{\la}^{\tr}$:
		\begin{equation}\label{eq:ZXbyindicators:app}
		\begin{split}
		\E \left[\bwZ_{\la}^{\tr} (\underline{X})_{\underline{a}}  \right]
		=
		\sum_{\YY} \sum_{\underline{\tau}_\YY} \E \left[\bwZ_{\la}^{\tr} \one\{\YY, \underline{\tau}_\YY \} \right],
		\end{split}
		\end{equation}
		where $\YY = \{\YY_i(\zeta) \}_{i\in[a_\zeta],\, ||\zeta||\leq l_0}$ denotes the locations of $\underline{a}$ $\zeta$-cycles and $\underline{\tau}_\YY$ describes a prescribed coloring configuration on them.
		
		As before, we work with an empirical profile count $g= (\dot{g},(\hat{g}^{\uL})_{\uL}, \bar{g})$ that satisfies $||g-g^\star_{\la}||_1\leq \sqrt{n} \log^2 n$. We additionally assume that 
		\begin{equation}\label{eq:cycmo:untrun:gcondition}
		\sum_{\sig \nsubseteq \dot{\Omega}^{\textsf{typ}} } |\dot{g}(\sig)| \leq n^{4/5},
		\end{equation}
		and analogous conditions for $\hat{g}^{\uL}$ and $\bar{g}$.
	   The empirical counts $g$ that does not satisfy the equation above are excluded due to the same reason as \eqref{eq:g:faraway:negligbile:truncated}. We additionally write $H = (\dot{H}, (\hat{H}^{\uL})_{\uL}, \bar{H})$ for their normalized versions, that is,
		\begin{equation*}
		\dot{H} \equiv \frac{\dot{g}}{n}, \quad \hat{H}^{\uL} \equiv \frac{\hat{g}^{\uL} }{|\hat{g}^{\uL}|} , \quad \bar{H} \equiv \frac{\bar{g}}{nd}.
		\end{equation*} 
		Recall the definition of the empirical profile $\Delta = (\dot{\Delta}, (\hat{\Delta}^{\uL})_{\uL}, \bar{\Delta}_c)$ on $\YY$ (Definition \ref{def:cycprofile}). Then, as in  \eqref{eq:Zgexpansion:manycyc}, we fix a literal assignment $\uL_E$ that is compatible with $(\hat{g}^{\uL})_{\uL}$ and write 
		\begin{equation}\label{eq:Zgdecomp and HHdef}
		\frac{ \E [\bZ_{\la}^{\tr}[g] \one\{\YY, \underline{\tau}_\YY \}\,|\, \uL_E  ] }{ \E [\bZ_{\la}^{\tr}[g] \,|\,\uL_E]}
		= 
		\frac{1}{ (nd)^{|\bar{\Delta}_c|}} \frac{ (\dot{H})_{\dot{\Delta}} \prod_{\uL} (\hat{H}^{\uL})_{\hat{\Delta}^{\uL}} }{ (\bar{H})_{\bar{\Delta}_c} } \equiv  
			\frac{\HH(H,\Delta)}{ (nd)^{|\bar{\Delta}_c|}}.
		\end{equation}
		Moreover, we define
		\begin{equation*}
		\eta \equiv \eta(\YY) \equiv |\bar{\Delta}_c| - |\dot{\Delta}| - |\hat{\Delta}|
		\end{equation*}
		as before, noting that it is well-defined without knowing $\underline{\tau}_\YY.$ In what follows, we perform case analysis depending on $\eta(\YY)$. It turns out that the case $\eta = 0$ gives the main contribution, but the analysis for both cases become more complicated than in the proof of Proposition \ref{prop:cycleeffect:moments}-(1) or in Section \ref{subsec:app:Ymanycyc} due to the existence of $\dot{\Omega}^{\textsf{atyp}}$.
		
	The key analysis lies 	in the computation of $\sum_{\underline{\tau}_\YY} \HH(H,\Delta[\underline{\tau}_\YY] )$. In what follows, we carry on this analysis in two different cases, when $\eta=0$ and when it is not.
		
		\subsubsection{Case 1. $\eta=0$}
		
		Since $\YY$ consists of pairwise disjoint cycles, we can consider $\HH$ as a product of the corresponding function defined on each cycle and work out separately when summing over $\underline{\tau}_\YY$. Therefore, we will assume that $\YY = \{\YY(\zeta) \}$ for some $||\zeta||\leq l_0$, and later take products over different cycles.
		
	We may separate the sum  $\sum_{\underline{\tau}_\YY} \HH(H,\Delta[\underline{\tau}_\YY] )$ into two cases, when $\underline{\tau}_\YY \subset \dot{\Omega}^{\textsf{typ}}$ and when it is not.
		
		\vspace{2mm}
		\noindent \textbf{Case 1-1.} $\underline{\tau}_\YY \subset \dot{\Omega}^{\textsf{typ}}$.
		\vspace{2mm}
		
	 If $||g-g^\star||\leq \sqrt{n} \log^2 n$, then for all $\sigma\in \dot{\Omega}^{\textsf{typ}}$ we have
		\begin{equation}\label{eq:cycleef:untrun:gestim}
		\left| \frac{H(\sigma)}{H^\star(\sigma)} -1 \right| \leq n^{-1/4} \log^2n.
		\end{equation}
		Moreover, recall the matrices $(\dot{A}\hat{A})^\zeta$ defined in \eqref{eq:def:AdotAhat zet}. Similarly, we introduce 
		\begin{equation*}
		(\dot{A}\hat{A})^\zeta_{\textsf{typ}}  \equiv
		\prod_{i=0}^{||\zeta||-1} \left(\dot{A}_{\textsf{typ}} \hat{A}^{\zeta_{2i},\zeta_{2i+1}}_{\textsf{typ}} \right),
		\end{equation*}
		where $\dot{A}_{\textsf{typ}}$ and $\hat{A}_{\textsf{typ}}^{\texttt{L}_1,\texttt{L}_2}$ denote the $\dot{\Omega}^{\textsf{typ}} \times \dot{\Omega}^{\textsf{typ}} $ submatrices of $\dot{A}$ and $\hat{A}^{\texttt{L}_1,\texttt{L}_2}$.
		Then, for $H$ of our interest, we can express
		\begin{equation*}
		\sum_{\underline{\tau}_\YY\subset \dot{\Omega}^{\textsf{typ}}}
		\HH(H,\Delta) = \left(1+ O\left(\frac{\log^2 n}{n^{1/4}} \right) \right) Tr \left[ (\dot{A}\hat{A})^\zeta_{\textsf{typ}} \right].  
		\end{equation*}
		Following the same analysis done in the proof of Proposition \ref{prop:cycleeffect:moments}-(4) in Section \ref{subsec:app:deltabd}, we obtain that
		\begin{equation*}
		Tr \left[ (\dot{A}\hat{A})^\zeta \right]-Tr \left[ (\dot{A}\hat{A})^\zeta_{\textsf{typ}} \right]
		\lesssim_{k,d} n^{-1/4},
		\end{equation*}
		which gives us that
		\begin{equation*}
			\sum_{\underline{\tau}_\YY\subset \dot{\Omega}^{\textsf{typ}}}
			\HH(H,\Delta) = 1+ \delta(\zeta) + O(n^{-1/4}).	
		\end{equation*}

		\vspace{2mm}
		\noindent \textbf{Case 1-2.} $\underline{\tau}_\YY \nsubseteq \dot{\Omega}^{\textsf{typ}}$.
		\vspace{2mm}
		
		This case can be treated by a similar way as the proof of Proposition \ref{prop:cycleeffect:moments}-(4) in Section \ref{subsec:app:deltabd}. Let $l=||\zeta||$, and without loss of generality we assume that $\zeta= \underline{0}$. Denoting $\hat{A} \equiv \hat{A}^{0,0}$, we can write
		\begin{equation}\label{eq:HHsum:disj:atyp}
		\sum_{\underline{\tau}_\YY \nsubseteq \dot{\Omega}^{\textsf{typ}} } \HH (H, \Delta) 
		=
		\sum_{\underline{\sigma} \nsubseteq \dot{\Omega}^{\textsf{typ}} } \prod_{i=0}^{l-1} \frac{\dot{H}(\sigma_{2i}, \sigma_{2i+1})}{ \bar{H}(\sigma_{2i}) } \frac{ \hat{H}(\sigma_{2i+1}, \sigma_{2i+2})}{\bar{H}(\sigma_{2i-1})},
		\end{equation}
		with $\sigma_0 = \sigma_{2l}$.
		
			Observe that in a tuple $(\sigma_1,\ldots,\sigma_{2l})$ that contributes to the above sum, there should exists $j\in[2l]$ such that $\sigma_j \in \{\bb_0,\bb_1,\fs \}$ and $\sigma_{j+1}\in \dot{\Omega}^{\textnormal{\textsf{atyp}}}.$ Otherwise, it would imply that the tuple $(\sigma_1,\ldots,\sigma_{2l})$ forms a free component that has a cycle (of lengh $2l$), which contradicts the assumption that the set $\dot{\Omega}$ only contains the colors which induce a free tree. Without loss of generality, suppose that $j=2l-1$ satisfies the above criterion (the case of $j$ being even can also be covered by the same argument). Then,
			\begin{equation*}
			\frac{\hat{H}(\sigma_{2l-1},\sigma_{2l})}{ \bar{H}(\sigma_{2l-1})}
			\leq
			\frac{\bar{H}{(\sigma_{2l})}}{\bar{H}(\sigma_{2l-1})}
			\lesssim n^{-1/5}.
			\end{equation*}
			(Note that this holds not only for $H^\star$, but for any $H$ satisfying \eqref{eq:cycmo:untrun:gcondition})
			Thus, plugging this into \eqref{eq:HHsum:disj:atyp} and summing over the rest of the colors gives that
			\begin{equation*}
		\sum_{\underline{\tau}_\YY \nsubseteq \dot{\Omega}^{\textsf{typ}} } \HH (H, \Delta) \lesssim_{k,d,l}  n^{-1/5}.
			\end{equation*}
			
		\vspace{2mm}
		Combining Cases 1-1 and 1-2, we obtain that  for $\YY$ with $\eta(\YY)=0$,
		\begin{equation*}
		\sum_{\underline{\tau}_\YY} 	\frac{ \E [\bZ_{\la}^{\tr}[g] \one\{\YY, \underline{\tau}_\YY \}\,|\, \uL_E  ] }{ \E [\bZ_{\la}^{\tr}[g] \,|\,\uL_E]} = \frac{1+ O(n^{-1/4}\log^2n) }{(nd)^{|\bar{\Delta}_c|} } (1+\delta(\zeta)+ O(n^{-1/5})).
		\end{equation*}
		
		Therefore, in the general case when $\YY$ consists of $\underline{a}$ disjoint $\zeta$-cycles, averaging over $g$, $\uL_E$ and then summing over $\YY$ gives
		\begin{equation}\label{eq:cycmo:untrun:final:case1}
		\sum_{\YY : \eta(\YY)=0} \sum_{\underline{\tau}_\YY} 
			\frac{ \E \big[\bwZ_{\la}^{\tr} \one\{\YY, \underline{\tau}_\YY \}\big] }{ \E \bwZ_{\la}^{\tr}} = \left(1+ O\left(\frac{\log^2 n}{n^{1/4}} \right) \right) \left( \underline{\mu}(1+\underline{\delta}) \right)^{\underline{a}}.
		\end{equation}

		\subsubsection{Case 2. $\eta>0$}
		
		In this case, $\YY$ decomposes into $||\underline{a}||_1-\eta$ connected components, and each component can be considered separately. If a component in $\YY$ is a single cycle, it can be treated analogously as the previous case. Therefore, we assume that $\YY = \{\YY(\zeta_1), \ldots, \YY(\zeta_j) \}$ such that the cycles $\YY(\zeta_1),\ldots, \YY(\zeta_j)$ form a single connected component in $\GGG$. Moreover, without loss of generality, we consider the case that all $\zeta_i$, $1\leq i \leq j$ are identically $0$.
		
		We define the \textit{orientation} on $\YY $ as follows: 
		\begin{itemize}
			\item[O1.] For each half edge $e=(va)\in E_c(\YY)$, make it a directed edge by assigning a direction, either $v\to a$ or $a\to v$.
			
			\item[O2.] An assignment of directions on $E_c(\YY)$ is called an \textit{orientation} if every variable and clause has at least one incoming edge adjacent to it.
		\end{itemize}
      Note that we can always construct an orientation as follows: Take a spanning tree of $\YY$ and pick a variable (or clause)  that has an edge not included in the tree. Starting from the selected vertex (root), we can assign directions on the tree so that all vertices but root has an incoming edge. Then, set the direction of the edge at root which is not in the tree to complete the orientation.
      
		We fix an orientation of $\YY$, and for each variable $v\in V(\YY)$ (resp. clause $a\in F(\YY)$), fix $e(v)$ (resp. $e(a)$) to be an incoming edge. Note that $e(v),\; v\in V(\YY)$ and $e(a), \;a\in F(\YY)$ are all distinct by definition.
		
		Denoting $E_c = E_c (\YY), \, V'=V(\YY)$ and $F'=F(\YY)$, let $$E_\circ = E_c\setminus \{e\in E_c: e=e(v) \textnormal{ or } e=e(a) \textnormal{ for some } v\in V',\, a\in F' \}.$$ 
		Here, note that $\eta(\YY) = |E_\circ|$.
		Additionally, for each $v\in V'$ and $a\in F'$, we define
		\begin{equation*}
		\delta_c(v) \equiv \{ e\in E_c\setminus E_\circ: e\sim v  \}, \quad \delta_c(a) \equiv \{e\in E_c\setminus E_\circ : e\sim a \}.
		\end{equation*}
		(Note that $\delta_c (v)$ is a singleton unless $v$ is an overlapping variable. Same goes for $\delta_c(a)$.) For a fixed $\sig_{E_c}$ we express the sum of $\HH(H,\Delta) \equiv \HH(H, \underline{\tau}_\YY)$ as follows.
		\begin{equation}\label{eq:HHdecom:conditional}
		\begin{split}
		\sum_{\underline{\tau}_\YY : \underline{\tau}_{E_c} = \sig_{E_c} }\HH (H,\underline{\tau}_\YY)
		=
		\prod_{v\in V'} \dot{H}(\sig_{\delta_c(v)} \,|\,\sigma_{e(v)} )
			\prod_{a\in F'} \hat{H}(\sig_{\delta_c(a)} \,|\,\sigma_{e(a)} )
			\left\{\prod_{e\in E_\circ} \bar{H} (\sigma_e)	\right\}^{-1},
		\end{split}
		\end{equation}
		where the conditional measures in the formula are defined as 
		\begin{equation*}
		\dot{H} (\sig_{\delta_c(v)} \,|\,\sigma_{e(v)}) \equiv
		\frac{1}{\bar{H}(\sigma_{e(v)})} \sum_{\underline{\tau}_{\delta v} } \dot{H} (\underline{\tau}_{\delta v}) \one_{\{(\underline{\tau}_{\delta_c(v)}, \tau_{e(v)}  ) = (\sig_{\delta_c(v)}, \sigma_{e(v)}  ) \} }.
		\end{equation*}
		
		We study the sum of \eqref{eq:HHdecom:conditional} over $\sig_{E_c}$, in two cases: when $\sig_{E_c} \subset \dot{\Omega}^{\textsf{typ}}$ and when it is not.
		
		\vspace{2mm}
		\noindent \textbf{Case 2-1.} $\sig_{E_c}\subset \dot{\Omega}^{\textsf{typ}}$.
		\vspace{2mm}
		
		In this case, since $|E_\circ| = \eta$, we have
		\begin{equation}\label{eq:HHsum:eta:typ}
		\sum_{\underline{\tau}_\YY : \underline{\tau}_{E_c} = \sig_{E_c} }\HH (H,\underline{\tau}_\YY)
		\leq n^{\eta/4}
		\prod_{v\in V'} \dot{H}(\sig_{\delta_c(v)} \,|\,\sigma_{e(v)} )
		\prod_{a\in F'} \hat{H}(\sig_{\delta_c(a)} \,|\,\sigma_{e(a)} ).
		\end{equation}
		Since each conditional measure $\dot{H}(\;\cdot\;|\sigma_{e(v)})$, $\hat{H}(\;\cdot\;|\sigma_{e(a)})$ has total mass equal to 1 on $\dot{\Omega}$, we sum the above over all $\sigma_{E_c} \subset \dot{\Omega}^{\textsf{typ}}$ and deduce that
		\begin{equation}\label{eq:HHsum:intersec:typ}
		\sum_{\sig_{E_c}\subset \dot{\Omega}^{\textsf{typ}}} 	\sum_{\underline{\tau}_\YY : \underline{\tau}_{E_c} = \sig_{E_c} }\HH (H,\underline{\tau}_\YY)
		\lesssim n^{\eta/4}.
		\end{equation}

			\vspace{2mm}
			\noindent \textbf{Case 2-2.} $\sig_{E_c}\nsubseteq \dot{\Omega}^{\textsf{typ}}$.
			\vspace{2mm}

	As done in Case 1-2, there should exist two adjacent edges $e',e''\in E_c$ such that $\sigma_{e'} \in \{\bb_0, \bb_1, \bs \}$ and $\sigma_{e''}\in \dot{\Omega}^{\textsf{atyp}}$. Assume that both $e', e''$ are adjacent to a variable $v$ and  $e'=e(v)$
	In such a setting, we have
	\begin{equation}\label{eq:HHconditional:atyp:bd}
	\sum_{\sig_{\delta_c(v) }\nsubseteq \dot{\Omega}^{\textsf{typ}} } \dot{H} (\sig_{\delta_c(v) }\,|\,\sigma_{e'} ) \leq n^{-1/4}.
	\end{equation}
	
	Having this property in mind, fix $\sig_{E_c} \nsubseteq \dot{\Omega}^{\textsf{typ}}$, and let $E_\circ^{\textsf{atyp}}$ be
	\begin{equation*}
E_\circ^{\textsf{atyp}}\equiv E_\circ^{\textsf{atyp}}(\sig_{E_c}) \equiv \{e\in E_\circ : \sigma_{e} \in \dot{\Omega}^{\textsf{atyp}}  \},
	\end{equation*}
	and define $\eta'\equiv \eta'(\sig_{E_c}) \equiv |E_\circ^{\textsf{atyp}}|$.
	Then, similarly as 	\eqref{eq:HHsum:eta:typ}, we can write
	\begin{equation}\label{eq:HHpresum:atyp}
		\sum_{\underline{\tau}_\YY : \underline{\tau}_{E_c} = \sig_{E_c} }\HH (H,\underline{\tau}_\YY)
		\leq n^{\eta/4} n^{3\eta'/4}
		\prod_{v\in V'} \dot{H}(\sig_{\delta_c(v)} \,|\,\sigma_{e(v)} )
		\prod_{a\in F'} \hat{H}(\sig_{\delta_c(a)} \,|\,\sigma_{e(a)} ),
	\end{equation}	 
	where we crudely bounded $\bar{H}(\sigma_e) \geq n^{-1}$ for $\sigma_e \in \dot{\Omega}^{\textsf{atyp}}.$ We claim that there should be at least $\eta'+1$ variables or clauses such that \eqref{eq:HHconditional:atyp:bd} happen. 
			
	For each $e \in E_\circ^{\textsf{atyp}}$, consider the following ``backtracking'' algorithm:

	\begin{enumerate}
		\item Let $e_0=e$, and let $x(e_0)$ be the variable or clause that has $e_0$ as an outgoing edge. 
		
		\item Let $e_1 = e(x(e_0))\in E_c\setminus E_\circ$ be the unique incoming edge into $x(e_0)$ as defined above. If $\sigma_{e_1} \in \{\bb_0,\bb_1, \fs \}$, then we terminate the algorithm and output $e_\star(e)= e_1$.
		
		\item If not, define $e_{i+1}= e(x(e_i))$ as (1), (2), and continue until termination as mentioned in (2).
	\end{enumerate}
	For each $e\in E_\circ^{\textsf{atyp}}$, this algorithm must terminate, otherwise it will imply that $\sig_{E_c}$ contains a cycle in a free component. Also, we introduce a similar algorithm which outputs $e_{\star\star}(e)\in E_c$ for each $e\in E_\circ^{\textsf{atyp}}$:
	\begin{itemize}
		\item[(a)] Let $y(e_0)$ be the variable or clause that has $e_0=e$ as an incoming edge. 
		
		\item[(b)] Let $e_1= e(y(e_0)) \in E_c \setminus E_\circ$ be the unique incoming edge into $y(e_0)$ as defined above. If $\sigma_{e_1} \in \{\bb_0,\bb_1, \bs \}$, then we terminate the algorithm and output $e_{\star\star}(e)= e_1$.
		
		\item[(c)] If not, define $e_{i+1} = e( x(e_i))$ ($i\geq 1$), where $x(e_i)$ is defined as (1) in the previous algorithm. Continue until termination as mentioned in (b).
	\end{itemize}
	This algorithm should also terminate in a finite time as we saw above. Moreover, $e_\star(e)$ and $e_{\star\star}(e)$ should be different for each $e\in E_\circ^{\textsf{atyp}}$, since if they were the same it would mean that the free component containing $e$ has a cycle.
	
	Consider the graph $\mathfrak{G}=(\mathfrak{V},\mathfrak{E})$ defined as follows:
	\begin{itemize}
		\item $\mathfrak{V} \equiv \{e_\star(e), e_{\star\star}(e): 
		\, e\in E_\circ^{\textsf{atyp}}  \}.$
		\item $e_1, e_2 \in \mathfrak{V}$ are adjacent if there exists $e\in E_\circ^{\textsf{atyp}}$ such that $e_1 = e_\star(e)$ and $e_2 = e_{\star\star}(e)$.
	\end{itemize}
	Observe that $\mathfrak{G}$ should not contain any cycles, since a cycle inside $\mathfrak{G}$ will imply the existence of a free component containing a cycle. Since $|\mathfrak{E}|=\eta'$, this implies that $|\mathfrak{V}|\geq \eta'+1$. Since the set $\mathfrak{V}$ locates the edges $e\in E_c$  where \eqref{eq:HHconditional:atyp:bd} happens, we have at least $\eta'+1$ distinct edges (or vertices) that satisfy   \eqref{eq:HHconditional:atyp:bd}.

	Having this in mind, we  sum \eqref{eq:HHpresum:atyp} over all $\sig_{E_c}\nsubseteq \dot{\Omega}^{\textsf{typ}}$ and deduce that 
	 \begin{equation}\label{eq:HHsum:intersec:atyp}
	 	\sum_{\sig_{E_c}\nsubseteq \dot{\Omega}^{\textsf{typ}}} 	\sum_{\underline{\tau}_\YY : \underline{\tau}_{E_c} = \sig_{E_c} }\HH (H,\underline{\tau}_\YY)
	 	\lesssim n^{3\eta/4}.
	 \end{equation}

	 \vspace{2mm}
	 \noindent \textbf{Back to the proof of Case 2.}
	 \vspace{2mm}
	 
	 Now we go back to the general setting, where $\YY$ contains multiple connected components with $\eta(\YY)>0$. When we sum $\E [\bZ_{\la}^{\tr} \one\{\YY,\underline{\tau}_\YY \}]$ over all $\underline{\tau}$, each $\zeta$-cycle in $\YY$ that is disjoint with all others will provide a contribution of $(1+\delta(\zeta) +O(n^{-1/5} )$ as discussed in Case 1. On the other hand, the contributions from components that are not a single cycle will be bounded by $n^{3\eta/4}$ due to \eqref{eq:HHsum:intersec:typ}, \eqref{eq:HHsum:intersec:atyp}. Summarizing the discussion, we have
	 \begin{equation*}
	 		\sum_{\underline{\tau}_\YY} 	\frac{ \E [\bZ_{\la}^{\tr}[g] \one\{\YY, \underline{\tau}_\YY \}\,|\, \uL_E  ] }{ \E [\bZ_{\la}^{\tr}[g] \,|\,\uL_E]} \lesssim \frac{1}{(nd)^{|\bar{\Delta}_c|}} (1+\underline{\delta})^{\underline{a}} n^{3\eta/4}.
	 \end{equation*}
	 Summing over all $\YY$ satisfying $\eta(\YY)=\eta$ can then be done using \eqref{eq:choice of eta YY:simplified}. This gives that
	 \begin{equation*}
	  \sum_{\YY: \eta(\YY)=\eta} \sum_{\underline{\tau}_\YY}	\frac{  \E [\bZ_{\la}^{\tr}[g] \one\{\YY, \underline{\tau}_\YY\} ],|\, \uL_E}{
	  	\E [\bZ_{\la}^{\tr}[g],|\, \uL_E] 
	  }
	  \leq 	2^{2a^\dagger} \left(\underline{\mu}(1+\underline{\delta}_L)\right)^{\underline{a}}\, \left(\frac{C' a^\dagger}{n^{1/4}} \right)^{\eta},
	 \end{equation*}
	 where $C'$ is a constant depending only on $k, d$ and $a^\dagger \equiv \sum_{||\zeta||\leq l_0} ||\zeta||a_\zeta$. We can choose $c_{\textsf{cyc}}=c_{\textsf{cyc}}(l_0)$ so that $2^{2a^\dagger} \leq n^{1/8}$ for any $||\underline{a}||_\infty \leq c_{\textsf{cyc}}\log n$. Then, we obtain the following conclusion by summing the above over all $\eta\geq 1$ and averaging over  $\uL_E$ and $g$ satisfying $||g-g^\star_{\la}||_1\leq \sqrt{n}\log^2 n$ and \eqref{eq:cycmo:untrun:gcondition}:
	 	\begin{equation}\label{eq:cycmo:untrun:final:case2}
	 	\sum_{\YY : \eta(\YY)\geq 1} \sum_{\underline{\tau}_\YY} 
	 		\frac{ \E [\bwZ_{\la}^{\tr} \one\{\YY, \underline{\tau}_\YY \} ] }{ \E \bwZ_{\la}^{\tr}} \lesssim n^{-1/8} \left( \underline{\mu}(1+\underline{\delta}) \right)^{\underline{a}}.
	 		\end{equation}
	  Finally, we conclude the proof of Proposition \ref{prop:cycleeffect:moments}-(5) by combining \eqref{eq:cycmo:untrun:final:case1} and \eqref{eq:cycmo:untrun:final:case2}. \qed
	 
	 \subsection{Proof of Lemma \ref{lem:Zplantedexpansion}}\label{subsec:app:Zplantedexp}
	
	In this section, we present the proof of Lemma \ref{lem:Zplantedexpansion}. Our approach relies on applying similar ideas as  Lemma 6.7 of \cite{dss16}		and Proposition \ref{prop:cycleeffect:moments} to 
	\begin{equation}\label{eq:Zplantedformula}
	\E_T \left[\prescript{}{2}{\bZ}^\partial(\underline{\tau}_\UUU;\Gamma_2^\bullet )\, \one\{\YY, \underline{\tau}_\YY \} \right].
	\end{equation} 
	
	\begin{proof}[Proof of Lemma \ref{lem:Zplantedexpansion}]
	For a given $\underline{\tau}_{\UUU}$, let 	$\dot{\epsilon}$ and $(\hat{\epsilon}^{\uL})_{\uL}$ be integer-valued measures on $(\dot{\Omega}_L^2)^d$ and $(\dot{\Omega}_L^2)^k$, respectively, such that
	\begin{equation}\label{eq:epsilbasiccond}
	\hat{M}\sum_{\uL}\hat{\epsilon}^{\uL}- \dot{M}\dot{\epsilon} = \bar{h}^{\underline{\tau}_{\UUU}}.
	\end{equation}
	In particular, we can first define $\dot{\epsilon}$ and $\sum_{\uL}\hat{\epsilon}^{\uL}$, following the construction of $(\dot{\epsilon}, \hat{\epsilon})$ given in (60), \cite{dss16} and Lemma 4.4, \cite{ssz16}: there exist $(\dot{\epsilon}^\tau, \hat{\epsilon}^\tau)_{\tau \in \dot{\Omega}_L^2}$ such that
	\begin{equation*}
	\dot{\epsilon}\equiv \sum_{\tau\in \dot{\Omega}_L^2} \bar{h}^{\underline{\tau}_{\UUU}}(\tau) \;\dot{\epsilon}^\tau, \quad \textnormal{and }\quad 
		\sum_{\uL}\hat{\epsilon}^{\uL}\equiv \sum_{\tau\in \dot{\Omega}_L^2} \bar{h}^{\underline{\tau}_{\UUU}}(\tau)\;\hat{\epsilon}^\tau
	\end{equation*}
	satisfy the desired condition \eqref{eq:epsilbasiccond}. After that, we distribute the mass $\hat{\epsilon} \equiv \sum_{\uL} \hat{\epsilon}^{\uL}$, which can be done in the following way:
	\begin{itemize}
		\item For each $\underline{\tau}\in (\dot{\Omega}_L^2)^k$, pick one $\uL\in\{0,1\}^k$ such that $\underline{\tau} \oplus \uL$ defines a valid coloring around a clause.  Then, set $\hat{\epsilon}^{\uL}(\underline{\tau}) = \hat{\epsilon}(\underline{\tau})$.
	\end{itemize}
	 For such $\dot{\epsilon}$ and $\hat{\epsilon}$, let
	\begin{equation*}
	\nu \equiv |\dot{\epsilon}|\equiv \langle \dot{\epsilon},1\rangle, \quad \textnormal{and }  \mu\equiv |\hat{\epsilon}| \equiv \langle \hat{\epsilon} ,1 \rangle,
	\end{equation*}
	where both depend only on $|\UUU|$, not on $\underline{\tau}_{\UUU}$.
	
	Similarly as in the proof of Proposition \ref{prop:cycleeffect:moments}, we study \eqref{eq:Zplantedformula} by computing the contribution from each empirical profile. If $g-\epsilon = (\dot{g}-\dot{\epsilon}, (\hat{g}^{\uL}-\hat{\epsilon}^{\uL})_{\uL})$ is an empirical profile contributing to \eqref{eq:Zplantedbulkexpansion}, then $g$ contributes to the full random $(d,k)$-regular graph with $\tilde{n}=n-|V(T)|+\nu$ variables and $\tilde{m}=m-|F(T)|+\mu$ clauses. Let $\Xi(g|\uL_{\tilde{E}})$ be the contribution of $g$ to $\E[\bZ^2 |\uL_{\tilde{E}} ]$ on such random graph with literal assignment $\uL_{\tilde{E}}$, given by
	\begin{equation*}
	\Xi(g|\uL_{\tilde{E}})\equiv { |\dot{g}|\choose \dot{g} } {|\hat{g}|\choose \hat{g}} {|\dot{M}\dot{g}| \choose  \dot{M} \dot{g}}^{-1} w(g)^\lambda,
	\end{equation*}
	where $w(g)$ is given by \eqref{eq:def:wg}. 
	
	Let $\Xi_c (g,\epsilon,\Delta ,U\,|\uL{E} )$ be the contribution of the profile $g-\epsilon$ to \eqref{eq:Zplantedformula}, conditioned on the literal assignments being $\uL_E$. We can write down its explicit formula as follows.
	\begin{equation}\label{eq:Xicformula}
	\begin{split}
	\Xi_c (g,\epsilon,\Delta,U|\uL_E) 
	&=
	 {|\dot{g}|-|\dot{\epsilon}| - |\dot{\Delta}_\partial| \choose \dot{g}-\dot{\epsilon} -\dot{\Delta}_\partial } \prod_{\uL} {|\hat{g}^{\uL}| - |\hat{\epsilon}^{\uL}| -|\hat{\Delta}_\partial^{\uL}| \choose \hat{g}^{\uL} -\hat{\epsilon}^{\uL} -\hat{\Delta}_\partial^{\uL}  }
	 \times
	 \frac{ (\dot{M}(\dot{g}-\dot{\epsilon} ) -\bar{\Delta} -\bar{\Delta}_U )! }{(n_\partial d)! }\\
	 &\quad \quad \times
	\frac{ (\dot{M}(\dot{g}-\dot{\epsilon})- \bar{\Delta} -\bar{g}^{\underline{\tau}_{\UUU}  } )_{\dot{M}\dot{\Delta}_\partial -\bar{\Delta} -\bar{\Delta}_U}   }{  (\dot{M}(\dot{g}-\dot{\epsilon})- \bar{\Delta} -\bar{\Delta}_U   )_{\dot{M}\dot{\Delta}_\partial -\bar{\Delta} -\bar{\Delta}_U} }  
	 \times
	 w(\dot{g}-\dot{\epsilon}, (\hat{g}^{\uL}-\hat{\epsilon}^{\uL})_{\uL}) ,
	\end{split}
	\end{equation}
	where the meaning of each term in the \textsc{rhs} can be described as follows.
	\begin{enumerate}
		\item The first term counts the number of ways to locate the variables and clauses except the ones given by $\YY$ and $\underline{\tau}_\YY$.
		
		\item The second denotes the probability of getting a valid matching between variable- and clause-adjacent half-edges. Note that $\bar{\Delta}+\bar{\Delta}_U$ is subtracted since the edges on $\YY$ should be matched through specific choices prescribed by $\YY$.
	
		\item In (2), we should exclude the cases that the half-edges in $\cup_{v\in V(\YY) } \delta v \setminus E_c(\YY)$ are matched with the boundary half-edges of $T$. The probability of not having such an occasion is given by the third term. For future use, we define
		\begin{equation*}
		b_1 (g,\epsilon,\Delta,U)
		\equiv
			\frac{ (\dot{M}(\dot{g}-\dot{\epsilon}) -\bar{\Delta} -\bar{g}^{\underline{\tau}_{\UUU}  } )_{\dot{M}\dot{\Delta}_\partial -\bar{\Delta} -\bar{\Delta}_U}   }{  (\dot{M}(\dot{g}-\dot{\epsilon})- \bar{\Delta} -\bar{\Delta}_U   )_{\dot{M}\dot{\Delta}_\partial -\bar{\Delta} -\bar{\Delta}_U} }  
		\end{equation*}
		
		\item The last term denotes the product of variable, clause and edge factors in $G^\partial$.

	\end{enumerate}

	Then, we compare $\Xi_c(g,\epsilon,\Delta,U| \uL_E)$ and $\Xi(g|\uL_{\tilde{E}})$, for $g$ that satisfies $||g-g_\star|| \leq \sqrt{n}\log^2 n$, where we wrote $g_\star \equiv g^\star_{\lambda,L}$. Note that in such setting, $\uL_{\widetilde{E}}$ and $\uL_E$ should differ by $|\hat{\epsilon}^{\uL}|$ for each $\uL\in\{0,1\}^k$.  Moreover, set $\hat{g}=\sum_{\uL} \hat{g}^{\uL}$ and  $\hat{\Delta}_\partial = \sum_{\uL} \hat{\Delta}^{\uL}_\partial$. We can write
	\begin{equation}\label{eq:Xicexpansion intermed1}
	\begin{split}
	\frac{\Xi_c(g,\epsilon,\Delta,U|\uL_E) }{\Xi(g|\uL_{\tilde{E}})} 
	&=
	\frac{(|\dot{g}|d)_{|\dot{\epsilon}|d}}{ (|\dot{g}|)_{|\dot{\epsilon}|+ |\dot{\Delta}_\partial| } (|\hat{g}|)_{|\hat{\epsilon}|+|\hat{\Delta}_\partial|} }
	\times
	\frac{
	(\dot{g})_{\dot{\epsilon}+\dot{\Delta}_\partial} \prod_{\uL}( \hat{g}^{\uL} )_{\hat{\epsilon}^{\uL} + \hat{\Delta}^{\uL}_\partial } \;\bar{g}_\star^{\dot{M}\dot{\epsilon} + \bar{\Delta} + \bar{\Delta}_U}  	
		}{\dot{g}_\star^{\dot{\epsilon}+\dot{\Delta}_\partial} \prod_{\uL} (\hat{g}_\star^{\uL})^{\hat{\epsilon}^{\uL}+\hat{\Delta}^{\uL}_\partial }  
		\; (\dot{M}\dot{g})_{\dot{M}\dot{\epsilon} +\bar{\Delta} + \bar{\Delta}_U }
		 } \\
		 &\quad\times \frac{ \dot{g}_\star^{\dot{\epsilon}+\dot{\Delta}_\partial } \prod_{\uL} (\hat{g}^{\uL}_\star)^{\hat{\epsilon}^{\uL}+ \hat{\Delta}^{\uL}_\partial} 
		 	}{ \bar{g}_\star^{\dot{M}\dot{\epsilon}+\bar{\Delta} +\bar{\Delta}_U }
		 	}
		 	\times
		 	b_1 ( g,\epsilon,\Delta,U) 
		 		\times
	\frac{ \bar{\Phi}^{\dot{M}\dot{\epsilon} }}{ \dot{\Phi}^{\dot{\epsilon}} \prod_{\uL} (\hat{\Phi}^{\uL})^{\hat{\epsilon}^{\uL}}  },
	\end{split}
	\end{equation}
	where we define $\hat{\Phi}^{\uL} (\underline{\tau}) \equiv \hat{\Phi}^{\textnormal{lit}}(\underline{\tau} \oplus \uL) .$ We also set 
	\begin{equation*}
	b_2 (g,\epsilon,\Delta, U) \equiv \frac{
		(\dot{g})_{\dot{\epsilon}+\dot{\Delta}_\partial} \prod_{\uL}( \hat{g}^{\uL} )_{\hat{\epsilon}^{\uL} + \hat{\Delta}^{\uL}_\partial } \;\bar{g}_\star^{\dot{M}\dot{\epsilon} + \bar{\Delta} + \bar{\Delta}_U}  	
	}{\dot{g}_\star^{\dot{\epsilon}+\dot{\Delta}_\partial} \prod_{\uL} (\hat{g}_\star^{\uL})^{\hat{\epsilon}^{\uL}+\hat{\Delta}^{\uL}_\partial }  
	\; (\dot{M}\dot{g})_{\dot{M}\dot{\epsilon} +\bar{\Delta} + \bar{\Delta}_U }
},
	\end{equation*}
and rearrange 	\eqref{eq:Xicexpansion intermed1} to obtain that
\begin{equation}\label{eq:Xicexpansion intermed2}
\begin{split}
	\frac{\Xi_c(g,\epsilon,\Delta,U|\uL_E) }{\Xi(g|\uL_{\tilde{E}})} 
	&=
	\frac{\bar{z}^{|\dot{\epsilon}|d} }{ \dot{z}^{|\dot{\epsilon}|} \hat{z}^{|\hat{\epsilon} |}  }
	\times
	\frac{  n^{|\dot{\epsilon}|+|\dot{\Delta}_\partial |} m^{|\hat{\epsilon}|+|\hat{\Delta}_\partial |}  (|\dot{g}|d)_{|\dot{\epsilon}|d}  
		}{
		(|\dot{g}|)_{|\dot{\epsilon}|+|\dot{\Delta}_\partial |} (|\hat{g}|)_{|\hat{\epsilon}|+|\hat{\Delta}_\partial |}
		(nd)^{|\dot{\epsilon}|d +|U|+|\bar{\Delta}| } }\\
	&\quad \times
	b_1 \times b_2 \times
	\prod_{e\in \UUU} \dot{q}^\star_{\lambda,L}(\tau_e) 
	\times
	\frac{
	\dot{H}_\star^{\dot{\Delta}_\partial} \prod_{\uL} (\hat{H}_\star^{\uL} )^{\hat{\Delta}_\partial^{\uL} }
		}{ 
	\bar{H}_\star^{ \bar{\Delta} + \bar{\Delta}_U }	
		}.
\end{split}
\end{equation}
We define
\begin{equation*}
c_0 \equiv \frac{\bar{z}^{|\dot{\epsilon}|d} }{ \dot{z}^{|\dot{\epsilon}|} \hat{z}^{|\hat{\epsilon} |}  },
\end{equation*}
which is the constant $c_0$ in the statement of the lemma. Moreover, since $n-|\dot{g}|$ and $m-|\hat{g}|$ are both bounded by $O((dk)^{l_0})$, we can write
\begin{equation*}
	\frac{  n^{|\dot{\epsilon}|+|\dot{\Delta}_\partial |} m^{|\hat{\epsilon}|+|\hat{\Delta}_\partial |}  (|\dot{g}|d)_{|\dot{\epsilon}|d}  
	}{
	(|\dot{g}|)_{|\dot{\epsilon}|+|\dot{\Delta}_\partial |} (|\hat{g}|)_{|\hat{\epsilon}|+|\hat{\Delta}_\partial |}
	(nd)^{|\dot{\epsilon}|d +|U|+|\bar{\Delta}| } } =
\left(1+ O\left(\frac{ ||\underline{a}||_1^2 }{n}  \right) \right) (nd)^{-|U|-|\bar{\Delta}|},
\end{equation*} 
and this quantity is independent of $\underline{\tau}_{\UUU'}$.

What remains is to analyze the error terms $b_1$ and $b_2$. The estimate for $b_1$ can be obtained by the following direct expansion:
\begin{equation}\label{eq:b1expression}
\begin{split}
&b_1(g,\epsilon,\Delta,U)
=
\prod_{\tau\in \dot{\Omega}_L^2} \prod_{i=1}^{(\dot{M}\dot{\Delta}_\partial -\bar{\Delta} -\bar{\Delta}_U )(\tau) }
\left(
1- \frac{
	\bar{h}^{\underline{\tau}_{\UUU'}}-i+1
	}{(\dot{M}(\dot{g}-\dot{\epsilon})- \bar{\Delta} -\bar{\Delta}_U   )(\tau)-i+1 }
\right)\\
&=  
1- \left\langle
\bar{h}^{\underline{\tau}_{\UUU'}}, \frac{ \dot{M}(\dot{g}-\dot{\epsilon})- \bar{\Delta} -\bar{\Delta}_U  }{ \bar{g}_\star }
 \right\rangle + \left\langle 1, \frac{ (\dot{M}(\dot{g}-\dot{\epsilon})- \bar{\Delta} -\bar{\Delta}_U)_2 }{ 2\bar{g}_\star}
  \right\rangle
  +O\left(\frac{\log^4 n}{n^{3/2}} \right).
\end{split}
\end{equation}
	
	On the other hand, $b_2$ can be studied based on the same approach as Lemma 6.7 of \cite{dss16}. Define $A[g] \equiv (\dot{A}[g], \hat{A}[g], \bar{A}[g])$ and $B[g] \equiv (\dot{B}[g], \hat{B}[g], \bar{B}[g])$ to be
	\begin{equation*}
	A[g] = \frac{g-g_\star}{g_\star}, \quad \textnormal{and } B[g] = \left(\frac{g-g_\star}{ g_\star}\right)^2 -\frac{1}{g_\star}.
	\end{equation*}
	We can write $b_2$ using the above, namely,
	\begin{equation}\label{eq:b2expression1}
	\frac{(\dot{g})_{\dot{\epsilon}+\dot{\Delta}_\partial} }{\dot{g}_\star^{\dot{\epsilon}+\dot{\Delta}_\partial } } = 1+ \left\langle \dot{\epsilon} +\dot{\Delta}_\partial , \dot{A}[\dot{g}] \right\rangle  + \left\langle\frac{(\dot{\epsilon}+ \dot{\Delta}_\partial )_2}{2}, \dot{B}[\dot{g}]  \right\rangle
	+
	O\left(\frac{(||\epsilon||_1+||\Delta||_1)^3 \log^6 n }{n^{3/2}} \right),
	\end{equation}
	and similarly for the terms including $\hat{g}$ and $\dot{M}\dot{g}$ (See the proof of Lemma 6.7 (page 480) of \cite{dss16} for its precise derivation). Moreover, since the leading exponent of $\Xi(g)$ is negative-definite at $g_\star$, the averages $A^{\textnormal{avg}}$, $B^{\textnormal{avg}}$ defined by
	\begin{equation*}
	A^{\textnormal{avg}} \equiv \sum_{||g-g_\star||\leq \sqrt{n}\log^2 n} \frac{\Xi(g) A[g]}{\Xi(g)}, \quad \textnormal{and } \quad
	B^{\textnormal{avg}} \equiv \sum_{||g-g_\star||\leq\sqrt{n} \log^2 n} \frac{\Xi(g) B[g]}{\Xi(g)}
	\end{equation*}
	satisfy the bounds $||A^{\textnormal{avg}}||_\infty =O(n^{-1/2})$, $||B^{\textnormal{avg}}||_\infty= O(n^{-1})$.
	Meanwhile, we can write
	\begin{equation}\label{eq:b2expression2}
	\begin{split}
	\langle \dot{\epsilon}, \dot{A}^{\textnormal{avg}} \rangle
	&=\langle \bar{h}^{\underline{\tau}_{\UUU'}}, \xi_0' \rangle, \quad \textnormal{where } \xi_0'(\tau) \equiv \langle\dot{\epsilon}^{\tau} , \dot{A}^{\textnormal{avg}} \rangle;\\
	\dot{\epsilon}(\underline{\tau})^2 \dot{B}^{\textnormal{avg}}(\underline{\tau})
	&= \langle \bar{h}^{\underline{\tau}_{\UUU'}}, \xi'_{\underline{\tau}} \rangle^2, \quad \textnormal{where } \xi'_{\underline{\tau}} (\tau) \equiv \dot{\epsilon}^{\tau} (\underline{\tau}) (\dot{B}^{\textnormal{avg}} (\underline{\tau}) )^{1/2},
	\end{split}
	\end{equation}
	and similarly the terms involving $\hat{\epsilon}^{\uL}$ and $\dot{M}\dot{\epsilon}$.
	
	One more thing to note when averaging \eqref{eq:Xicexpansion intermed2} is that only $2^{-|\bar{\Delta}|}$ fraction of $\uL_E$ gives a non-zero value (as written in \eqref{eq:Xicexpansion intermed2}), since the literals prescribed by $\YY$ should be fixed. Having this in mind, averaging \eqref{eq:Xicexpansion intermed2} based on the observations \eqref{eq:b1expression}, \eqref{eq:b2expression1} and \eqref{eq:b2expression2} gives us the conclusion.
	\end{proof}

\end{document}